\documentclass{amsart}

\usepackage{amsfonts,mathrsfs}
\usepackage{amsmath,amssymb,amsthm, amscd, bm}
\usepackage{tikz}
\usetikzlibrary{arrows,calc,matrix,topaths,positioning,scopes,shapes,decorations,decorations.markings} 
\usepackage{dsfont}
\usepackage{epsfig}
\usepackage{tikz-cd}
\usepackage{hyperref}
\usepackage{fullpage}
\usepackage[foot]{amsaddr}
\usepackage{adjustbox}
\usepackage{enumitem}

\usepackage[utf8x]{inputenc}

\makeatletter

\newcommand{\sslash}{\mathbin{/\mkern-6mu/}}

\newcommand{\PS}{\operatorname{PS}}

\newcommand{\eqabchar}{\mathcal{X}^{\sigma}_{\mathbb{C}^*}(\hat{\mathbf{\Sigma}})}

\newcommand{\CF}{\mathcal{Z}_{\omega}(\mathbf{\Sigma}, \Delta)}

\newcommand{\na}{\mathcal{NA}}
\newcommand{\Tr}{\operatorname{Tr}}

\newcommand{\St}{\operatorname{St}}

\newcommand{\SB}{\operatorname{SB}}

\newcommand{\id}{id}

\newcommand{\Hom}{\operatorname{Hom}}

\newcommand{\SL}{\operatorname{SL}}
\newcommand{\PSL}{\operatorname{PSL}}

\newcommand{\End}{\operatorname{End}}

\newcommand{\tr}{\operatorname{Tr}}
\newcommand{\Mod}{\operatorname{Mod}}

\newcommand{\quotient}[2]{{\raisebox{.2em}{$#1$}\left/\raisebox{-.2em}{$#2$}\right.}}

\AtEndDocument{
		\bigskip{\footnotesize%
				\textsc{Universidade de S\~ao Paulo - Instituto de Ci\^encias Matem\'aticas e de Computa\c{c}\~ao.  
					Avenida Trabalhador S\~ao-Carlense, 400 - Centro CEP: 13566-590 - S\~ao Carlos - SP Brazil} \par  
				\textit{E-mail address:} \email{julien.korinman@gmail.com} \\
		\par	\textit{E-mail address:} \email{math@quesney.org}\\
			}}

\begin{document}

\def\co{\colon\thinspace}
\newcommand{\ot}{\otimes} 
\newcommand{\ov}[1]{\overline{#1}}
\newcommand{\un}[1]{\underline{#1}}
\newcommand{\wid}[1]{\widetilde{#1}}
\newcommand{\ca}[1]{\mathcal{#1}}
\newcommand{\cat}[1]{\textbf{#1}}
\newcommand{\ie}{\emph{i.e.} }
\newcommand{\alex}[1]{{\color{blue}{\texttt{#1}}}}
\newcommand{\alexx}[1]{{\color{green!50!black}{\texttt{#1}}}}
\newcommand{\bsout}[1]{\sout{\color{blue}{#1}}}
\newcommand{\bxsout}[1]{\xout{\color{blue}{#1}}}
\newcommand{\alnote}[1]{\marginnote{\color{blue}{\texttt{\small{#1}}}}}
\newcommand{\D}{\mathbb{D}}
\newcommand{\R}{\mathbb{R}}
\newcommand{\C}{\mathbb{C}}
\newcommand{\Z}{\mathbb{Z}}
\newcommand{\T}{\mathbb{T}}
\newcommand{\Hex}{\hat{\mathbb{T}}}
\newcommand{\Ring}{\mathcal{R}}
\newcommand{\eqSkpre}{\mathcal{S}_{\omega}^{\mathbb{C}^*,\sigma}}
\newcommand{\nab}{\mathbf{NA}}
\newcommand{\ea}{\mathfrak{a}}
\newcommand{\eb}{\mathfrak{b}}
\newcommand{\ec}{\mathfrak{c}}

\newcommand{\heightexch}[3]{
	\begin{tikzpicture}[baseline=-0.4ex,scale=0.5, >=stealth]
	\draw [fill=gray!60,gray!45] (-.7,-.75)  rectangle (.4,.75)   ;
	\draw[#1] (0.4,-0.75) to (.4,.75);
	\draw[line width=1.2] (0.4,-0.3) to (-.7,-.3);
	\draw[line width=1.2] (0.4,0.3) to (-.7,.3);
	\draw (0.65,0.3) node {\scriptsize{$#2$}}; 
	\draw (0.65,-0.3) node {\scriptsize{$#3$}}; 
	\end{tikzpicture}
}
\newcommand{\heightcurve}{
\begin{tikzpicture}[baseline=-0.4ex,scale=0.5]
\draw [fill=gray!20,gray!45] (-.7,-.75)  rectangle (.4,.75)   ;
\draw[-] (0.4,-0.75) to (.4,.75);
\draw[line width=1.2] (-.7,-0.3) to (-.4,-.3);
\draw[line width=1.2] (-.7,0.3) to (-.4,.3);
\draw[line width=1.15] (-.4,0) ++(-90:.3) arc (-90:90:.3);
\end{tikzpicture}
}
\newcommand{\heightexchor}[3]{
\begin{tikzpicture}[>=stealth,thick, arrow/.style={->},baseline=-0.4ex,scale=0.5]
\draw [fill=gray!60,gray!45] (-.7,-.75)  rectangle (.4,.75)   ;
\draw[line width=.7,#1] (0.4,-0.75) to (.4,.75);
\draw[line width=1.1,#2] (0.4,-0.3) to (-.7,-.3);
\draw[line width=1.1,#3] (0.4,0.3) to (-.7,.3);
\end{tikzpicture}
}

\theoremstyle{plain}
\newtheorem{theorem}{Theorem}[section]
\newtheorem{proposition}[theorem]{Proposition}
\newtheorem{corollary}[theorem]{Corollary}
\newtheorem{lemma}[theorem]{Lemma}
\theoremstyle{definition}
\newtheorem{notation}[theorem]{Notations}
\newtheorem{convention}[theorem]{Convention}
\newtheorem{problem}[theorem]{Problem}
\newtheorem{definition}[theorem]{Definition}
\theoremstyle{remark}
\newtheorem{remark}[theorem]{Remark}
\newtheorem{conjecture}[theorem]{Conjecture}
\newtheorem{example}[theorem]{Example}
\newtheorem{strategy}[theorem]{Strategy}
\newtheorem{question}[theorem]{Question}


\title{The quantum trace as a quantum non-abelianization map}

\author{Julien Korinman \quad and \quad Alexandre Quesney}


%
%
%
%
%
\date{}
\maketitle

\begin{abstract} 
We prove that the balanced Chekhov-Fock algebra of a punctured triangulated surface is isomorphic to a skein algebra which is a deformation of the algebra of regular functions of some abelian character variety. We first deduce from this observation a classification of the irreducible representations of the balanced Chekhov-Fock algebra at odd roots of unity, which generalizes to open surfaces the classification of Bonahon, Liu and Wong. We  re-interpret Bonahon and Wong's quantum trace map as a non-commutative deformation of some regular morphism between this abelian character variety and the $\mathrm{SL}_2$-character variety. This algebraic morphism shares many resemblance with the non-abelianization map of Gaiotto, Moore, Hollands and Neitzke. When the punctured surface is closed, we prove that this algebraic non-abelianization map induces a birational morphism between a smooth torus and the relative $\SL_2$ character variety.
\vspace{2mm}
\end{abstract}


\section{Introduction}

\par \textbf{Skein algebras, quantum Teichm\"uller spaces and  character varieties}
\vspace{3mm}
\par  A \textit{punctured surface} is a pair $\mathbf{\Sigma}=(\Sigma,\mathcal{P})$, where $\Sigma$ is a compact oriented surface and $\mathcal{P}$ is a (possibly empty) finite subset of $\Sigma$ which intersects non-trivially each boundary component.  We write $\Sigma_{\mathcal{P}}:= \Sigma \setminus \mathcal{P}$.
  The set $\partial \Sigma\setminus \mathcal{P}$ consists of a disjoint union of open arcs which we call \textit{boundary arcs}.
   \\ \textbf{Warning:} In this paper, the punctured surface $\mathbf{\Sigma}$ will be called open if the surface $\Sigma$ has non empty boundary and closed otherwise. This convention differs from the traditional one, where some authors refer to open surface a punctured surface $\mathbf{\Sigma}=(\Sigma, \mathcal{P})$ with $\Sigma$ closed and $\mathcal{P}\neq \emptyset$ (in which case $\Sigma_{\mathcal{P}}$ is not closed).
   
\vspace{2mm}\par   In this paper, we will consider different related objects associated to punctured surfaces, namely the Kauffman-bracket skein algebras, the balanced Chekhov-Fock algebras and the character varieties. Let us briefly introduce them.

\vspace{2mm}

	The \textit{Kauffman-bracket skein algebra} was introduced by Turaev (\cite{Tu88}) for closed punctured surfaces; it was recently generalized to open surfaces by L\^e in \cite{LeStatedSkein}, following Bonahon-Wong \cite{BonahonWongqTrace},  under the name of \textit{stated skein algebra}. 
	For a commutative unital ring $\mathcal{R}$, an invertible element $\omega\in \mathcal{R}^{\times}$ and a punctured surface $\mathbf{\Sigma}$, the Kauffman-bracket skein algebra  $\mathcal{S}^{\SL_2}_{\omega}(\mathbf{\Sigma})$ is the free  $\mathcal{R}$--module generated by isotopy classes of some stated framed links in a cylinder over $\mathbf{\Sigma}$, modulo some local skein relations derived from $\SL_2(\C)$. 
	This algebra has deep relations with the $\SL_2(\C)$ character varieties and the Witten-Reshetikhin-Turaev Topological Quantum Field Theories (TQFT) defined in \cite{Wi2,RT}. 
	
	 When the punctures of $\mathbf{\Sigma}$ can be joined to form a triangulation  $\Delta$,  one can consider the  \textit{balanced Chekhov-Fock algebra} $\CF$. 
	It was introduced by Bonahon and Wong in \cite{BonahonWongqTrace} as a refinement of (an exponential version of) the quantum Teichm\"uller space defined by Chekhov and Fock in \cite{ChekhovFock} (see \cite{Kashaev98} for a related independent construction).  
	A key tool in the construction of representations of the skein algebra in \cite{BonahonWong2,BonahonWong3} is the  \textit{quantum trace map}. It is morphism of algebras 
	$\tr_{\omega}: \mathcal{S}^{\SL_2}_{\omega}(\mathbf{\Sigma}) \rightarrow \mathcal{Z}_{\omega}(\mathbf{\Sigma}, \Delta)$ 
	introduced by Bonahon and Wong in \cite{BonahonWongqTrace}. 
	\\
	
	Let $G$ be the group $\mathbb{C}^*$ or $\SL_2(\mathbb{C})$. 
	The \textit{character variety} $\mathcal{X}_G(\mathbf{\Sigma})$ of a punctured surface $\mathbf{\Sigma}$ is an affine Poisson variety. It was first introduced by Culler and Shalen in \cite{CullerShalenCharVar} for closed surfaces and generalized to open ones in \cite{KojuTriangularCharVar}.  
	When $G=\mathbb{C}^*$, the character variety is a smooth torus: indeed, its algebra of regular functions is the group algebra $\mathbb{C}[\mathrm{H}_1(\Sigma_{\mathcal{P}}, \partial \Sigma_{\mathcal{P}};\mathbb{Z})]$ so $\mathcal{X}_{\mathbb{C}^*}(\mathbf{\Sigma}) \cong (\mathbb{C}^*)^n$, where $n$ is the rank of the free abelian group $\mathrm{H}_1(\Sigma_{\mathcal{P}}, \partial \Sigma_{\mathcal{P}}; \mathbb{Z})$ (see Lemma \ref{lemma_free}).
	When $G=\SL_2(\mathbb{C})$ and the surface is closed, the character variety is singular and its smooth part $\mathcal{X}^0_G(\Sigma)$ is a smooth manifold. 
	\vspace{2mm} \par
	For closed punctured surfaces $\mathbf{\Sigma}$, the character variety is closely related to the moduli space $\mathcal{M}_G(\mathbf{\Sigma})$ of classes of flat connections on a trivial $G$ bundle over $\mathbf{\Sigma}$, modulo gauge equivalences. 
	More precisely, there is a map $p : \mathcal{M}_G(\mathbf{\Sigma}) \rightarrow \mathcal{X}_G(\mathbf{\Sigma})$ which is a bijection when $G=\mathbb{C}^*$. When $G=\SL_2(\mathbb{C})$, the map $p$ is surjective and, writing  $\mathcal{M}_{\SL_2}^0(\mathbf{\Sigma}):= p^{-1}\left( \mathcal{X}^0_{\SL_2}(\mathbf{\Sigma}) \right)$, the restriction $p : \mathcal{M}_{\SL_2}^0(\mathbf{\Sigma}) \rightarrow \mathcal{X}^0_{\SL_2}(\mathbf{\Sigma})$ is a bijection (see  \cite{LabourieCharVar, MarcheCours09, MarcheCharVarSkein} for surveys). 
	
	\vspace{2mm} \par
	The Kauffman-bracket skein algebras and the $\SL_2(\mathbb{C})$ character varieties are related as follows. When $\mathcal{R}=\mathbb{C}$ and $\omega= +1$, the Kauffman-bracket skein algebra $\mathcal{S}^{\SL_2}_{+1}(\mathbf{\Sigma})$ has a natural Poisson bracket arising from deformation quantization (see Section $3.1$).  
	
	For a closed punctured surface and a spin structure $S$ on it,  there exists an  isomorphism of Poisson algebras 
	$\Psi^S : \mathcal{S}^{\SL_2}_{+1}(\mathbf{\Sigma}) \xrightarrow{\cong} \mathbb{C}[\mathcal{X}_{\SL_2}(\mathbf{\Sigma})]$ from the skein algebra at $+1$ to the Poisson algebra of regular functions of the character variety (see  \cite{Bullock,PS00,ChaMa,Barett,Turaev91}). 
	For (not necessarily closed) triangulated punctured surfaces, there is an isomorphism of Poisson algebras  $\Psi^{(\mathfrak{o}, S)} : \mathcal{S}^{\SL_2}_{+1}(\mathbf{\Sigma}) \xrightarrow{\cong} \mathbb{C}[\mathcal{X}_{\SL_2}(\mathbf{\Sigma})]$ which depends on the choice of an orientation $\mathfrak{o}$ of its boundary arcs and a relative spin structure $S$ (see  \cite{KojuQuesneyClassicalShadows}). Note that when the punctured surface is open, the stated skein algebra at $A=-1$  is not commutative; this explains our choice of parameter $+1$ rather the more traditional one $-1$.
	
	\vspace{2mm}
	\par In TQFT,  skein algebras appear through their non-trivial finite dimensional representations. Skein algebras admit such representations if and only if the parameter $\omega$ is a root of unity. One motivation for the construction of the quantum trace map is that the representation category of the balanced Chekhov-Fock algebra is easier to study than the representation category of skein algebras. In \cite{BonahonLiu, BonahonWong2} it is shown that the balanced Chekhov-Fock algebras at root of unity of closed surfaces are semi-simple and their simple modules are classified. 
	For an irreducible representation $r': \mathcal{Z}_{\omega}(\mathbf{\Sigma}, \Delta) \rightarrow \End(V)$ of the balanced Chekhov-Fock algebra, one obtains a representation $r$ of skein algebras by composition:
	$$ r: \mathcal{S}_{\omega}^{\SL_2}(\mathbf{\Sigma}) \xrightarrow{\Tr_{\omega}} \mathcal{Z}_{\omega}(\mathbf{\Sigma}, \Delta) \xrightarrow{r'} \End(V). $$
	Such a representation is called  \textit{quantum Teichm\"uller representation}. 
	When $\omega$ is a root of unity of odd order $N>1$, there exists an injective morphism 
	$$j_{\mathbf{\Sigma}} : \mathcal{S}_{+1}^{\SL_2}(\mathbf{\Sigma}) \hookrightarrow  \mathcal{S}^{\SL_2}_{\omega}(\mathbf{\Sigma}) $$
	whose image lies in the center of the skein algebra $\mathcal{S}^{\SL_2}_{\omega}(\mathbf{\Sigma})$. 
	This was proved in \cite{BonahonWong1} for closed surfaces and generalized in \cite{KojuQuesneyClassicalShadows} for open surfaces as well.
	A quantum Teichm\"uller representation $r$ sends an element of the image of $j_{\mathbf{\Sigma}}$ to a scalar operator, hence it induces a character on the commutative algebra $\mathcal{S}^{\SL_2}_{+1}(\mathbf{\Sigma}) \xrightarrow[\cong]{\Psi} \mathbb{C}[\mathcal{X}_{\SL_2}(\mathbf{\Sigma})]$ and thus a point $[\rho]\in \mathcal{X}_{\SL_2}(\mathbf{\Sigma})$. 
	The latter is called the \textit{non abelian classical shadow} of $r$ and it only depends on the isomorphism class of $r$. 
	It was proved in \cite{FrohmanKaniaLe_UnicityRep} that, when $\Sigma$ is closed, "generic" irreducible representations of $\mathcal{S}^{\SL_2}_{\omega}(\mathbf{\Sigma})$ are quantum Teichm\"uller representations.
	
	\vspace{2mm}
	\par To a triangulated punctured surface $(\mathbf{\Sigma}, \Delta)$ one can associate a $2$-fold covering  $\hat{\mathbf{\Sigma}}$ with one branching point per triangle (see Section $2.6$). Let us denote by $\sigma$ its covering involution. 
	In  \cite{GMN12, GMN13, GMN14, HollandsNeitzke}, the authors considered the moduli space $\mathcal{M}_{\mathbb{C}^*}^{\sigma}(\hat{\mathbf{\Sigma}})\subset \mathcal{M}_{\mathbb{C}^*}(\hat{\mathbf{\Sigma}})$ of gauge classes of $\mathbb{C}^*$ flat connections whose holonomy along a curve $\gamma$ is the inverse of the holonomy along $\sigma(\gamma)$. They also considered a moduli space $\mathcal{M}_{\SL_2}^{fr}(\mathbf{\Sigma})$ of gauge classes of $\SL_2$ flat connections equipped with some additional decoration (called a framing). The main result in \cite{HollandsNeitzke}, inspired by \cite{GMN12, GMN13, GMN14} and the classical work of Hitchin in \cite{HitchinSelfDuality},  is the construction of a bijection $\mathcal{NA}^{HN} : \mathcal{M}_{\mathbb{C}^*}^{\sigma}(\hat{\mathbf{\Sigma}}) \xrightarrow{\cong} \mathcal{M}_{\SL_2}^{fr}(\mathbf{\Sigma})$, called \textit{non-abelianization map}, between the two moduli spaces. 
	\vspace{2mm}
	\par Since the Kauffman-bracket skein algebra $\mathcal{S}_{\omega}^{\SL_2}(\mathbf{\Sigma})$ is a deformation of the Poisson algebra of regular functions of the character variety $\mathcal{X}_{\SL_2}(\mathbf{\Sigma})$ which is closely related to the moduli space of flat connections $\mathcal{M}_{\SL_2}(\mathbf{\Sigma})$, it is natural to conjecture that the balanced Chekhov-Fock algebra $\mathcal{Z}_{\omega}(\mathbf{\Sigma}, \Delta)$ is a deformation of some character variety $\mathcal{X}_{\mathbb{C}^*}^{\sigma}(\hat{\mathbf{\Sigma}})$ itself related to the moduli space $\mathcal{M}_{\mathbb{C}^*}^{\sigma}(\hat{\mathbf{\Sigma}})$. It is also natural to expect that the quantum trace is a deformation of some algebraic non-abelianization map $\mathcal{NA} : \mathcal{X}_{\mathbb{C}^*}^{\sigma}(\hat{\mathbf{\Sigma}}) \rightarrow \mathcal{X}_{\SL_2}(\mathbf{\Sigma})$  related to the construction in \cite{GMN13, HollandsNeitzke}. The purpose of this paper is to provide such a construction and study some consequences.
\vspace{2mm}
\par The authors were recently informed by Allegretti and Kim that the relation between the quantum trace and the non-abelianization map was first emphasized by Gabella in \cite{Gabella_QNonAb} (see also \cite{KyuMiri_QNonAb}) though our treatment and resulting theorems are quite different.

\vspace{5mm}
\par \textbf{Main results of the paper}
\vspace{3mm}

	Let $\hat{\mathbf{\Sigma}}$ be the $2$--fold branched covering associated to a punctured triangulated surface $\mathbf{\Sigma}$. 
	We show that the balanced Chekhov-Fock algebra $\CF$ can be interpreted as an equivariant $\C^*$ skein algebra of $\hat{\mathbf{\Sigma}}$, that we denote by   $\mathcal{S}_{\omega}^{\mathbb{C}^*, \sigma}(\mathbf{\hat{\Sigma}})$. 
	On the other hand, we introduce the  \emph{equivariant $\C^*$ character variety} $\mathcal{X}_{\mathbb{C}^*}^{\sigma}(\hat{\mathbf{\Sigma}})$. It is a Poisson affine variety; its  set of closed points is in bijection with the moduli space 
	$\mathcal{M}_{\mathbb{C}^*}^{\sigma}(\hat{\mathbf{\Sigma}})$
	 of \cite{HollandsNeitzke}. 
	We show that   $\mathcal{S}_{\omega}^{\mathbb{C}^*, \sigma}(\mathbf{\hat{\Sigma}})$ provides a deformation of 
	$\mathcal{X}_{\mathbb{C}^*}^{\sigma}(\mathbf{\hat{\Sigma}})$. 
	The relation with the Chekhov-Fock algebra is the following. 
	\begin{theorem}\label{theorem1}
		For each leaf labeling $\ell$ of $\hat{\mathbf{\Sigma}}$ (see Section $2.2$), there is an isomorphism of algebras: $$\Phi_{\ell} : \mathcal{Z}_{\omega}(\mathbf{\Sigma}, \Delta) \xrightarrow{\cong} \mathcal{S}_{\omega}^{\mathbb{C}^*, \sigma}(\mathbf{\hat{\Sigma}}).$$
	\end{theorem}

In Section $2.8$, we show that if $\omega$ is a root of unity of odd order $N>1$, then the $\mathbb{C}^*$ skein algebra $\mathcal{S}_{\omega}^{\mathbb{C}^*, \sigma}(\mathbf{\hat{\Sigma}})$ is semi-simple;  we  provide tools to classify its irreducible representations and compute their dimensions. 
Using Theorem \ref{theorem1}, we deduce a classification of the irreducible representations of the balanced Chekhov-Fock algebras. More precisely, we show that there is a central inclusion  $\mathcal{O}[\mathcal{X}_{\mathbb{C}^*}^{\sigma}(\mathbf{\hat{\Sigma}})] \hookrightarrow  \mathcal{Z}_{\omega}(\mathbf{\Sigma}, \Delta)$. Therefore, each irreducible representation of the balanced Chekhov-Fock algebra induces a point in $\mathcal{X}_{\mathbb{C}^*}^{\sigma}(\mathbf{\hat{\Sigma}})$; we call it its \textit{abelian classical shadow}. 
To each inner puncture $p$ of $\mathbf{\Sigma}$ we associate some curves $\hat{\gamma}_p$ in $\mathbf{\hat{\Sigma}}$ and to each boundary component $\partial$ of $\mathbf{\Sigma}$ we associate some arcs $\hat{\alpha}_{\partial}$ in $\mathbf{\hat{\Sigma}}$. The following theorem was established by Bonahon and Wong in \cite{BonahonWong2} for closed punctured surfaces.
\begin{theorem}\label{theorem2}
	If $\omega$ is a root of unity of odd order $N>1$, the balanced Chekhov-Fock algebra $\mathcal{Z}_{\omega}(\mathbf{\Sigma}, \Delta)$ is semi-simple. When $\Sigma$ is connected, of genus $g$, and setting $s:= |\mathcal{P}|$ the number of punctures and $n_{\partial}$ the number of boundary components, its simple modules all have dimension $N^{3g-3+s+n_{\partial}}$, and are classified, up to isomorphism, by:
	\begin{enumerate}
		\item an abelian classical shadow $[\rho^{ab}]\in \mathcal{X}_{\mathbb{C}^*}^{\sigma}(\mathbf{\hat{\Sigma}})$; 
		\item for each inner puncture $p$, a $N$-{th} root $h_p$ of the holonomy of $[\rho^{ab}]$ around $\hat{\gamma}_p$ (puncture invariant); 
		\item for each boundary component $\partial$, a $N$-{th} root $h_{\partial}$ of the holonomy of $[\rho^{ab}]$ along $\hat{\alpha}_{\partial}$ (boundary invariant).
	\end{enumerate}
\end{theorem} 

\par 
Note that, by a theorem of De Concini and Procesi, the balanced Chekhov-Fock algebras at roots of unity (not necessary odd) are Azumaya of constant rank, hence they are semi-simple, their simple modules all have the same dimension and the set of their isomorphism classes is in bijection with the set of characters over the center of $\mathcal{Z}_{\omega}(\mathbf{\Sigma}, \Delta)$. With Theorem \ref{theorem2}, however, we provide a characterization of this center and compute the dimension of the simple modules. It will follow from  Proposition \ref{propdecomposition} which exhibits a tensor decomposition of  $\mathcal{Z}_{\omega}(\mathbf{\Sigma}, \Delta)$ into simpler quantum tori. We will not use the De Concini-Procesi theorem since the simple modules of those elementary algebras can be easily explicited.
\par 
An  important class of representations of the balanced Chekhov-Fock algebras are the local representations defined in \cite{BonahonBaiLiuLocalRep}. They are classified, up to isomorphism, by a classical shadow $[\rho^{ab}]\in \mathcal{X}_{\mathbb{C}^*}^{\sigma}(\mathbf{\hat{\Sigma}})$ and a complex number $h_C \in \mathbb{C}^*$ named its \textit{central charge}. The following corollary was proven by Toulisse in \cite{Toulisse16} when $\Sigma$ is a closed surface.

\begin{corollary}\label{corollary1}
Let $\omega$ be a root of unity of odd order $N>1$ and consider $W$ a local representation of the balanced Chekhov-Fock algebra $\mathcal{Z}_{\omega}(\mathbf{\Sigma}, \Delta)$ with classical shadow $[\rho^{ab}]$ and central charge $h_C$ of an arbitrary punctured surface $\mathbf{\Sigma}$. The set of isomorphism classes of  simple submodules of $W$ is the set of isomorphism classes of irreducible representations with classical shadow $[\rho^{ab}]$ and whose product of every boundary and inner puncture invariants is equal to $h_C$. Moreover, each simple summand arises with multiplicity $N^g$, where $g$ denotes the genus of $\Sigma$.
\end{corollary}

We define the (algebraic) non-abelianization map $\mathcal{NA} :  \mathcal{X}_{\mathbb{C}^*}^{\sigma}(\mathbf{\hat{\Sigma}}) \rightarrow \mathcal{X}_{\SL_2}(\mathbf{\Sigma})$ as the regular morphism associated to the algebra morphism 
$$ \mathcal{NA}^* : \mathcal{O}[\mathcal{X}_{\SL_2}(\mathbf{\Sigma})] \xrightarrow[\cong]{\Psi} \mathcal{S}_{+1}^{\SL_2}(\mathbf{\Sigma}) \xrightarrow{\tr_{+1}} \mathcal{Z}_{+1}(\mathbf{\Sigma}, \Delta) \xrightarrow[\cong]{\Phi_{\ell}} \mathcal{S}_{+1}^{\mathbb{C}^*, \sigma}(\mathbf{\hat{\Sigma}}) \xrightarrow[\cong]{} \mathcal{O}[\mathcal{X}_{\mathbb{C}^*}^{\sigma}(\mathbf{\hat{\Sigma}})]. $$
 One motivation for the study of $\mathcal{NA}$ is the following. When $\omega$ is a root of unity of odd order, a simple module  for $\CF$ has both an abelian classical shadow $[\rho^{ab}]\in \mathcal{X}_{\mathbb{C}^*}^{\sigma}(\mathbf{\hat{\Sigma}})$ and a non-abelian classical shadow $[\rho]\in \mathcal{X}_{\SL_2}(\mathbf{\Sigma})$ which are related by $\mathcal{NA}( [\rho^{ab}])=[\rho]$.

A common feature between our non-abelianization map and the construction in \cite{GMN13, HollandsNeitzke} is the following. For each inner edge $e$ of the triangulation $\Delta$, we define an oriented closed curve $\gamma_e$ in $\mathbf{\hat{\Sigma}}$ such that if  $\mathcal{NA}( [\rho^{ab}])=[\rho]$, then the holonomy of $\rho^{ab}$ along $\gamma_e$ is equal to the shear-bend coordinate of $[\rho]$ associated to $e$ (see Proposition \ref{prop:sb1}). Hence the non-abelianization can be thought of as a "balanced" shear bend parametrization. 

\vspace{2mm}

 Suppose that $\mathbf{\Sigma}$ is closed with at least one puncture; let $\mathcal{P}$ be the set of punctures.  
 For each  $p\in \mathcal{P}$, consider a small curve $\gamma_p$ of $\mathbf{\Sigma}$ around $p$. 
 Let $c: \mathcal{P} \rightarrow \mathbb{C}$ be a map and call \textit{relative} character variety the sub-variety $\mathcal{X}_{\SL_2}(\mathbf{\Sigma}, c) \subset \mathcal{X}_{\SL_2}(\mathbf{\Sigma})$ of classes of representations $\rho$ such that $\tr(\rho(\gamma_{p}))=c(p)$ for all puncture $p$. 
 Each puncture $p$ lifts to two punctures $\hat{p}_1$ and $\hat{p}_2$ in the double branched covering $\mathbf{\hat{\Sigma}}$. To each map  $\hat{c} : \mathcal{\hat{P}} \rightarrow \mathbb{C}^*$ such that $c(p)= \hat{c}(\hat{p}_1) + \hat{c}(\hat{p}_2)$ for any $p\in \mathcal{P}$ one can consider the  relative equivariant character variety $\mathcal{X}_{\mathbb{C}^*}^{\sigma}(\hat{\mathbf{\Sigma}}, \hat{c})$. 
 The non-abelianization map induces, by restriction, a regular morphism $\mathcal{NA}  : \mathcal{X}_{\mathbb{C}^*}^{\sigma}(\mathbf{\hat{\Sigma}}, \hat{c}) \rightarrow \mathcal{X}_{\SL_2}(\mathbf{\Sigma}, c)$.

\begin{theorem}\label{theorem3}
If $c(p)\neq \pm 2$ for every $p\in \mathcal{P}$, then the colored non-abelianization map $\mathcal{NA}  : \mathcal{X}_{\mathbb{C}^*}^{\sigma}(\mathbf{\hat{\Sigma}}, \hat{c}) \rightarrow \mathcal{X}_{\SL_2}(\mathbf{\Sigma}, c)$ is a birational symplectomorphism.
\end{theorem}

\par Note that the geometric non-abelianization map $\mathcal{NA}^{HN} : \mathcal{M}_{\mathbb{C}^*}^{\sigma}(\hat{\mathbf{\Sigma}}) \xrightarrow{\cong} \mathcal{M}_{\SL_2}^{fr}(\mathbf{\Sigma})$ in \cite{GMN13, HollandsNeitzke} is a bijection with value in the moduli space of framed $\SL_2$ flat connections. Here the framing refers to a generalization of Penner's notion of enhancement of the Teichm\"uller space and were already considered by various authors including the ones in \cite{BonahonLiu, BonahonWong3, FockGoncharovClusterVariety}. There is a natural map ${forget} : \mathcal{M}_{\SL_2}^{fr}(\mathbf{\Sigma}) \rightarrow \mathcal{M}_{\SL_2}(\mathbf{\Sigma})$ which consists in forgetting the framing, so by composition we get a map $\underline{\mathcal{NA}}^{HN}= {forget} \circ \mathcal{NA}^{HN} : \mathcal{M}_{\mathbb{C}^*}^{\sigma}(\hat{\mathbf{\Sigma}}) \rightarrow \mathcal{M}_{\SL_2}(\mathbf{\Sigma})$  which can be compared to our algebraic non-abelianization map. 
The forget map ${forget} : \mathcal{M}_{\SL_2}^{fr}(\mathbf{\Sigma}) \rightarrow \mathcal{M}_{\SL_2}(\mathbf{\Sigma})$ is neither injective nor surjective, so is the map $\underline{\mathcal{NA}}^{HN}$. Said differently, not every $\SL_2$-flat connection admit a framing and, when it does, the framing is not unique and it is not known how far is $\underline{\mathcal{NA}}^{HN}$ from being bijective. Concerning the image of ${forget}$ (and thus of $\underline{\mathcal{NA}}^{HN}$), Banahon-Liu (\cite{BonahonLiu}) and Bonahon-Wong (\cite{BonahonWong3}) gave some partial results that are summarized in Section 3. Theorem \ref{theorem3} gives a different kind of answer in the algebraic context: there exist open Zariski dense subsets $\mathcal{U} \subset \mathcal{X}_{\mathbb{C}^*}^{\sigma}(\mathbf{\hat{\Sigma}}, \hat{c})$ and $\mathcal{V} \subset \mathcal{X}_{\SL_2}(\mathbf{\Sigma}, c)$ such that $\mathcal{NA} : \mathcal{U} \rightarrow \mathcal{V}$ is an isomorphism.

\vspace{2mm}
\par We conclude the introduction with two comments. First, the variety $\mathcal{X}_{\mathbb{C}^*}^{\sigma}(\mathbf{\hat{\Sigma}}, \hat{c})$ is a smooth torus. The authors believe (but were not able to prove) that the colored non-abelianization map in Theorem \ref{theorem3} is a resolution of singularities (so is proper). Such a resolution would be a powerful tool to compute the E-polynomial of  $\mathcal{X}_{\SL_2}(\mathbf{\Sigma}, c)$ (\textit{i.e.} the Hodge numbers of the variety). The computation of these E-polynomials is only known when $\mathcal{P}$ is empty or has cardinality $1$ or $2$ (see \cite{LogaresMunoz, MartinezMunoz}). A second possible outcome of this paper is the following. The balanced Chekhov-Fock algebra, and thus the quantum trace map, only exist when $\mathcal{P}$ is non-empty. However given a closed surface $\Sigma$ without puncture, Hollands and Neitzke considered in \cite{HollandsNeitzke} a non-abelianization map associated to a pants decomposition $P$.  In Section $4$, we formulate a conjecture concerning the existence of a quantum non-abelianization map that would replace the quantum trace for unpunctured surfaces.

\vspace{5mm}
\par \textbf{Organisation of the paper}
\vspace{3mm}
\par 
In Section 2, we define the involutive punctured surfaces and show that they decompose into basic involutive punctured surfaces. We recall the definition of the $2$-fold branched covering associated to a triangulated punctured surface; it is the main example of involutive punctured surfaces that we consider.    
	 We then introduce the equivariant $\mathbb{C}^*$ skein algebras and equivariant $\mathbb{C}^*$ character varieties and prove Theorem \ref{theorem1}. 
We then investigate how the $\mathbb{C}^*$ skein algebra behaves under the decomposition of involutive punctured surfaces; we use this to prove Theorem \ref{theorem2}. 
We then prove Corollary \ref{corollary1}. 
\\
In Section 3, after a brief review  on the  stated skein algebra and the quantum trace map, we define the non-abelianization map and relate it with the shear-bend coordinates. 
The proof of Theorem \ref{theorem3} will follow from this relation.

 \begin{notation}  	We reserve the notation $q:=\omega^{-4}$. 
 \end{notation}

 \section{Equivariant $\mathbb{C}^*$ skein algebras and character varieties}

 \subsection{Involutive punctured surfaces}
 
 \begin{definition}
 A \textit{punctured surface} is a pair $\mathbf{\Sigma}=(\Sigma,\mathcal{P})$ where $\Sigma$ is a compact oriented surface (possibly with boundary) and $\mathcal{P}\subset \Sigma$ is a finite subset which intersects non-trivially each boundary component. The elements of $\mathcal{P}$ are called \textit{punctures}. We write $\Sigma_{\mathcal{P}}:=\Sigma \setminus \mathcal{P}$. A \emph{boundary arc} is a connected component of $\partial\Sigma_{\mathcal{P}}$. 
  An \emph{isomorphism} $f: (\Sigma,\mathcal{P}) \cong (\Sigma',\mathcal{P}')$  of punctured surfaces is an orientation-preserving homeomorphism of the underlying surfaces such that $f(\mathcal{P})= \mathcal{P}'$.
\end{definition}

 \begin{definition}\label{def_involutive_surface}
	 An \textit{involutive punctured surface} is a triple $(\Sigma, \mathcal{P}, \sigma)$ where $(\Sigma, \mathcal{P})$ is a punctured surface and $\sigma : \Sigma \rightarrow \Sigma$ is an 
  orientation-preserving homeomorphism such that: 
	\begin{enumerate}
		\item $\sigma^2=\id$ and $\sigma(\mathcal{P})=\mathcal{P}$. 
		\item The set $B$ of fixed-points of $\sigma$ is finite, disjoint from $\mathcal{P}$, lies in the interior of $\Sigma$ and intersects non-trivially each connected component of $\Sigma$.  Moreover, each fixed-point $b$ is contained in a small disc disjoint of $\mathcal{P}$, that is globally preserved by $\sigma$ and that intersects the set of fixed-point only at $b$.
		\item If $\partial$ is a boundary component of $\Sigma$ such that $\sigma(\partial)=\partial$,  then: 
		\begin{enumerate}
			\item there exists a simple closed curve $\gamma_{\partial}$ parallel to $\partial$ preserved by the involution, and
			\item $\partial$ has $2(2k+1)$ punctures, for a $k\geq 0$. 
		\end{enumerate}
		\item If $\partial$ is a boundary component of $\Sigma$ which is not stable under $\sigma$,  then it contains a (non null) even number  of punctures.
	\end{enumerate}

An \emph{isomorphism} of involutive punctured surfaces is an isomorphism of punctured surfaces that is equivariant for the involutions. 

\end{definition}

\begin{remark}\label{ex1}
	Every involutive punctured surface is a $2$--fold branched covering. 
	Indeed, let $(\Sigma, \mathcal{P}, \sigma)$  be an involutive punctured surface, and let $B$ be the set of fixed-points of $\sigma$.    
	By identifying the points $x$ and $\sigma(x)$, one obtains a $2$--fold covering $\pi\co \Sigma \to \Sigma/{\sigma}$ branched at the points $\pi(B)$. 
	Note that such a covering is characterized by an obstruction in $\mathrm{H}^1((\Sigma / {\sigma})\setminus \pi(B); \mathbb{Z}/2\mathbb{Z})$. 

\end{remark}

\begin{definition}
	  The \textit{combinatorial data} of an involutive punctured surface $(\Sigma,\mathcal{P}, \sigma)$ is the collection $(g,n_{\partial}, \{s_i \}_i, n_b, \mathring{s})$ consisting  of: the genus $g$ of $\Sigma$; the number $n_{\partial}$ of its boundary components; for each component $\partial_i$, the number $s_i$ of punctures it contains; the number $n_b$  of fixed points of the involution and the number $\mathring{s}$ which is  half of the number of punctures in the interior of $\Sigma$.
\end{definition}

\begin{lemma}\label{lemmatrucmachin} Any two connected involutive punctured surfaces have the same combinatorial data if and only if they are isomorphic.
\end{lemma}

\begin{proof} 
One implication is tautological. For the other one, let  $({\Sigma}_1, \mathcal{P}_1, \sigma_1)$ and $({\Sigma}_2, \mathcal{P}_2, \sigma_2)$ be two involutive punctured surfaces with the same combinatorial data. 
		There exists  a homeomorphism $\phi$ between $S_1:=\Sigma_1/{\sigma_1}$ and $S_2:=\Sigma_2/{\sigma_2}$ which: is orientation-preserving; sends the branched points to the branched points; sends the inner punctures to the inner punctures; and, sends a boundary component with $s_i$ punctures to a boundary component with $s_i$ punctures by preserving the punctures. 
For $i=1,2$, denote by  $o_i \in \mathrm{H}^1(S_i \setminus \pi_i(B_i); \mathbb{Z}/2\mathbb{Z})$ their obstructions as in Remark \ref{ex1}. 
For each branched point $b$, let $[\gamma_b]$ in $ \in \mathrm{H}_1(S_1 \setminus \pi_1(B_1); \mathbb{Z}/2\mathbb{Z})$ be the class of a simple closed curve encircling $b$. 
By definition of the coverings, both $o_1$ and $\phi^* o_2$ send the classes $[\gamma_b]$ to $1\in \mathbb{Z}/2\mathbb{Z}=\{0,1\}$. 
In particular both classes are not null since the set of branched points is non-empty. For each inner puncture $p$ and for each boundary component $\partial$, choose simple closed curves $\gamma_p$ encircling $p$ and $\gamma_{\partial}$ parallel to $\partial$ respectively and note that both $o_1$ and $\phi^* o_2$ send the classes $[\gamma_p]$ and $[\gamma_{\partial}]$ to $0\in \mathbb{Z}/2\mathbb{Z}=\{0,1\}$. Let $H \subset \mathrm{H}^1(S_1\setminus \pi_1(B_1); \mathbb{Z}/2\mathbb{Z})$ be the subset of classes sending the $\gamma_b$'s to $1$ and the $\gamma_p$'s and $\gamma_{\partial}$ to $0$.
The mapping class group $G$ of mapping classes of $S_1$ that  preserve the branched points and the punctures and that are equal to the identity on the boundary acts on  $H$. 
\vspace{2mm} \par  \textbf{Claim}:  the group $G$ acts transitively on $H$.

Indeed, suppose that $S_1$ has genus $g$ and write $\left(\cdot, \cdot \right)$ the intersection form modulo $2$. Let $\{\alpha_1, \beta_1, \ldots, \alpha_g, \beta_g, \gamma_1, \ldots \gamma_n\}$ be a basis of $\mathrm{H}_1(S_1\setminus \pi_1(B_1); \mathbb{Z}/2\mathbb{Z})$ where the $\gamma_i$ are homology classes of simple closed curves encircling the branched points, inner punctures and boundary components and $\alpha_i, \beta_i$ are such that $(\alpha_i, \alpha_j)= (\beta_i, \beta_j)=0, (\alpha_i,\beta_j)=\delta_{ij}$. Write $\mathrm{H}_1(S_1\setminus \pi_1(B_1); \mathbb{Z}/2\mathbb{Z})=X\oplus Y$ where $X$ is the $\mathbb{Z}/2\mathbb{Z}$-subspace generated by the $\alpha_i,\beta_i$ and $Y$ is generated by the $\gamma_j$. The group $G$ preserves both $X$ and $Y$,  acts as the identity on $Y$ and we need to show that its action on $X\setminus \{0\}$ is transitive. The $\mathbb{Z}/2\mathbb{Z}$ vector space $X$ with the intersection form is a $2g$-dimensional symplectic space and mapping classes preserves the intersection form so $\rho$ factors as $\rho : G \xrightarrow{\rho_1} \mathrm{Sp}_{2g}(\mathbb{Z}/2\mathbb{Z}) \xrightarrow{ \rho_2} \mathrm{GL}(X)$. The map $\rho_1$ is surjective (this follows from \cite[Theorem $6.4$]{FM}) and $\rho_2$ is the standard representation of the symplectic group $\mathrm{Sp}_{2g}(\mathbb{Z}/2\mathbb{Z})$ on $(\mathbb{Z}/2\mathbb{Z})^{2g}$ which acts transitively on the set of non zero vectors. Indeed, given $v, w \in (\mathbb{Z}/2\mathbb{Z})^{2g} \setminus \{0\}$, using the symplectic Gram-Schmidt process we can extend both $v$ and $w$ to a symplectic bases $b_v=(v, v_2, \ldots, v_{2g})$ and $b_w=(w, w_2, \ldots, w_{2g})$ and find $A \in \mathrm{Sp}_{2g}(\mathbb{Z}/2\mathbb{Z})$ sending $b_v$ to $b_w$ so sending $v$ to $w$. This proves the claim.

In particular, there exists an orientation-preserving homeomorphism $\psi\co S_1 \rightarrow S_1$ which preserves: the branched points, the punctures,  which is equal to the identity on the boundary  and which is such that $\psi^*\phi^*o_2=o_1$. The homeomorphism $\phi\circ \psi\co S_1\rightarrow S_2$ lifts to an isomorphism $ \Sigma_1\rightarrow \Sigma_2$ of punctured surfaces commuting with the involutions, hence gives an isomorphism between the involutive punctured surfaces.

\end{proof}

Let $\D$ be the unit disc in $\R^2$ endowed with the central symmetry as involution. 

\begin{definition}\label{def: sewing VEE}
	Let $(\mathbf{\Sigma}, \sigma)$ be an involutive punctured surface with set of fixed-points $B$ of cardinality $\geq 2$.   
	For $b_1,b_2\in B$, let $\veebar_{b_1,b_2}(\mathbf{\Sigma})$  be the following involutive punctured surface  
	with set of fixed-points $\veebar_{b_1,b_2}(B)=B\setminus \{b_1,b_2\}$ and set of punctures $\mathcal{P}$.  
	For $i=1,2$, choose an embedding $\phi_i\co \D\to \Sigma_{\mathcal{P}}$ that is equivariant for the involutions and that intersects $B$ only at $\phi_i(0)=b_i$; also we require that $\phi_1(\D)\cap \phi_2(\D)=\emptyset$. 
	In particular, one has  $\sigma(\partial \phi_i(\D))=\partial \phi_i(\D)$. 
	The surface $\veebar_{b_1,b_2}(\mathbf{\Sigma})$ is given by sewing  $\Sigma\setminus \left( \mathring{\phi_1(\D)} \sqcup  \mathring{\phi_2(\D)}\right)$ along the maps ${\phi_1}_{|\partial \D}$ and ${\phi_2}_{|\partial\D}$. 
\end{definition}

Let us consider the following $5$ \emph{basic} involutive punctured surfaces.  Figure \ref{figelementarycob} represents them.
\begin{enumerate}
	\item $\mathbb{S}= (\mathbb{S}^2, \{p, p'\}, \sigma)$, where $\mathbb{S}^2$  is the unit sphere of $ \mathbb{R}^3$, the punctures are  $p=(0,0,1)$ and $p'=(0,0,-1)$, and the involution is $\sigma((x,y,z))=(x,-y,-z)$. 
	\item $\mathbb{E}_1 = (\Sigma_1, \emptyset, \sigma)$, where $\Sigma_1= S^1\times S^1$ is the torus, there are no puncture, and $\sigma$ is the elliptic involution. 
	\item $\mathbb{E}_2 = (\Sigma_2, \emptyset, \sigma)$, where $\Sigma_2= \Sigma_1 \# \Sigma_1$ is a genus two surface, there are no puncture, and $\sigma$ is the involution exchanging the two copies of $\Sigma_1$. 
	\item For an odd integer $n\geq 1$, denote by $\mathbb{P}_n=(\D, \{p_1, \ldots, p_{2n} \}, \sigma)$, where $\D$ is the unit disc of $\mathbb{R}^2$,   the punctures are  $p_k = (\cos(\frac{\pi k}{n}),-\sin(\frac{\pi k}{n})) $ for $1\leq k \leq 2n$, and the involution is  $\sigma((x,y))=-(x,y)$. 
	\item For an even integer  $n\geq 2$, denote by $\mathbb{P}_n=(\mathbb{S}^1\times [-1,1], \{p_1, \ldots, p_n, p_1', \ldots, p_n'\}, \sigma)$, where the punctures are $p_k=(\cos(\frac{2\pi k}{n}),\sin(\frac{2\pi k}{n}), 1)$ and their images $p_k':=\sigma(p_k)$ for $1\leq k \leq n$, and the involution is defined by $\sigma((x,y),t)=((x,-y), -t)$.  
\end{enumerate}

   \begin{figure}[!h] 
\centerline{\includegraphics[width=8cm]{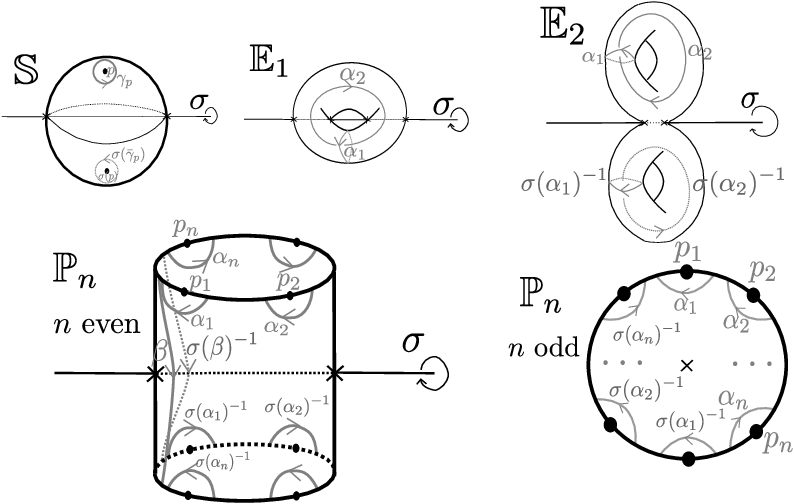} }
\caption{The basic involutive punctured surfaces. Involution is represented as a $180$-degree rotation around an axis. The dots represent punctures and the crosses are branched points.}
 \label{figelementarycob} 
\end{figure} 

\begin{lemma}\label{lemma_integers}
If $(\Sigma, \mathcal{P}, \sigma)$ is a connected  involutive punctured surface with combinatorial data $(g, n_{\partial}, \{s_i\}_{i\in I},n_b, \mathring{s})$, then  $n_1:= (n_b+n_{\partial}^{odd}-2)/2$  and $n_2:=(2g-n_b-n_{\partial}^{odd}+2)/4$ are non negative integers, where $n_{\partial}^{odd}$ denotes the number of boundary components preserved by $\sigma$.
\end{lemma}

\begin{proof}
Consider the surface $\widetilde{\Sigma}$ obtained from $\Sigma$ by removing one disc preserved by $\sigma$ around each inner puncture and each branched point. The involution $\sigma$ restricts to an involution on $\widetilde{\Sigma}$ without fixed point. The surface $\widetilde{\Sigma}$ has $l:=n_b+n_{\partial}+2\mathring{s}$ boundary components and $l_0:=n_b+n_{\partial}^{odd}\leq l$ of them are preserved by the free involution. By Definition \ref{def_involutive_surface}, one has $n_b\geq 1$, so $l_0\geq 1$. By \cite[Theorem $1.3$]{Asoh}, a genus $g$ oriented surface has  an orientation preserving free involution with $l\geq 0$ boundary components such that $1\leq  l_0\leq l$ of them are preserved by the involution if and only if $l$ and $l_0$ are even and $g+2 l_0-\frac{l_0}{2} \geq 1$ is odd. These conditions imply that $n_1$ and $n_2$ are non negative integers.
\end{proof}

  We will show that every  involutive punctured surface  decomposes into basic ones via $\veebar$. There are various ways to sew the  basic surfaces that lead to the same result. 
We focus on one way, which is ``linear'' in the sense that it will not  produce additional genera (\ie other than those of $\mathbb{E}_1$ and $\mathbb{E}_2$) when sewing these surfaces. 

\vspace{2mm}
\paragraph{\textit{Recipe for sewing the basic surfaces.}} 
Let us consider  a connected involutive punctured surface $(\Sigma,\mathcal{P},\sigma)$ with combinatorial data $(g,n_{\partial}, \{s_i \}_{i\in I}, n_b, \mathring{s})$.  
Let us split the set $I$ of  boundary components in two types: the "odd" ones which are preserved by the involution (so  their image in ${\Sigma}/{\sigma}$ have an odd number of punctures) and the "even" ones which not preserved by $\sigma$ (whose image in ${\Sigma}/{\sigma}$ has an even number of punctures).
This gives a partition of both the indexing set $I=I^{odd}\sqcup I^{even}$ and of its cardinality $n_{\partial}=n_{\partial}^{odd}+n_{\partial}^{even}$. 

We denote by $\veebar_{Rec}\left( \mathbb{S}^{\sqcup \mathring{s}} \sqcup \mathbb{E}_1^{\sqcup n_1}\sqcup \mathbb{E}_2^{\sqcup n_2}\sqcup \bigsqcup_{i\in I} \mathbb{P}_{s_i}\right)$ the following surface. 
\begin{enumerate}
	\item We sew the $\mathring{s}$ copies of $\mathbb{S}$ two-by-two; one obtains a sphere $S_1$ with  $\mathring{s}$ inner punctures and $2$ fixed-points. 
	\item We sew the $n_1+n_2$ copies of $\mathbb{E}_1$ and $\mathbb{E}_2$ two-by-two (the order does not matter); one obtains a surface $S_2$ of genus $n_1+2n_2=g$ with $2n_1+2$ fixed-points. 
	\item We sew the $n_{\partial}^{even}$ surfaces $\mathbb{P}_{s_i}$, for $s_i\in I^{even}$,  two-by-two; one obtains a surface $S_3$ with $2$ fixed-points and $\sum_{I^{even}} s_i$ boundary puntures. 
	\item We sew the three surfaces $S_1$, $S_2$ and $S_3$ two-by-two; the resulting surface $S_4$ has $2n_1+2=n_b+n_{\partial}^{odd}$ fixed-points. 
	\item We sew each of the $n_{\partial}^{odd}$ copies of $\mathbb{P}_{s_i}$ for $s_i\in I^{odd}$ with  $S_4$ along $n_{\partial}^{odd}$ of its fixed-points. One obtains an involutive  surface with combinatorial data $(g,n_{\partial}, \{s_i \}_{i\in I}, n_b, \mathring{s})$. 
\end{enumerate}

In virtue of Lemma \ref{lemmatrucmachin} we have shown the following.    

\begin{lemma}\label{lemma_decomp_invol_surface}
	Any connected involutive punctured surface with combinatorial data $(g,n_{\partial}, \{s_i \}_i, n_b, \mathring{s})$ is isomorphic to the involutive punctured surface 
	$\veebar_{Rec}\left( \mathbb{S}^{\sqcup \mathring{s}} \sqcup \mathbb{E}_1^{\sqcup n_1}\sqcup \mathbb{E}_2^{\sqcup n_2}\sqcup\bigsqcup_{i\in I} \mathbb{P}_{s_i}\right)$. 
\end{lemma}

\subsection{Double branched covering associated to topological triangulations}\label{sec: branch cov}

A particular class of involutive punctured surfaces that is of importance for us is that of  $2$--fold branched coverings of punctured surfaces equipped with a \emph{topological triangulation}.

 \begin{definition}\label{def_triangulation} 
 \begin{enumerate}
\item 	A \emph{small} punctured surface is one of the following four connected punctured surfaces: the sphere with one or two punctures; the disc with one or two  punctures on its boundary.
\item 	A punctured surface is said to \emph{admit a triangulation} if each of its connected components has at least one puncture and is not small. 
\item 	Suppose that $\mathbf{\Sigma}=(\Sigma, \mathcal{P})$ admits a triangulation.  
 	A \textit{topological triangulation} $\Delta$ of $\mathbf{\Sigma}$ is a collection $\mathcal{E}(\Delta)$ of embedded arcs in $\Sigma$ (named edges) which satisfy the following conditions: 
 		the endpoints of the edges belong to $\mathcal{P}$;  
 		the interior of the edges are pairwise disjoint and do not intersect $\mathcal{P}$; 
 		the edges are not contractible and are pairwise non isotopic in $\Sigma_{\mathcal{P}}$ relatively to their endpoints; 
 		the boundary arcs of $\mathbf{\Sigma}$ belong to $\mathcal{E}(\Delta)$. 
 		Moreover, the collection $\mathcal{E}(\Delta)$ is required to be maximal for these properties. 
 \end{enumerate}
 \end{definition}
Said differently, a triangulable surface $\mathbf{\Sigma}$ is obtained from a disjoint union $\mathbf{\Sigma}_{\Delta}=\bigsqcup_i \mathbb{T}_i$ of triangles by gluing some pairs of boundary arcs. The connected component $\mathbb{T}_i$ in $\mathbf{\Sigma}_{\Delta}$ are called \textit{faces} of the triangulation and their set is denoted by $F(\Delta)$.

For a triangulated punctured surface $(\Sigma, \mathcal{P},\Delta)$, let  $\Gamma^{\dagger}\subset \Sigma_{\mathcal{P}}$ be the dual of the $1$--skeleton of the triangulation. The graph $\Gamma^{\dagger}$ has one trivalent vertex inside each triangle, one univalent vertex inside each boundary arc and one edge $e^*$ intersecting once transversally each edge $e$ of the triangulation. Denote by $B$ the set of its vertices. Let $[\Gamma^{\dagger}]$ denotes its Borel-Moore relative homology class (see \cite{BorelMoore}) in $\mathrm{H}_1^c(\Sigma\setminus B, \partial\Sigma; \mathbb{Z}/2\mathbb{Z})$ and $\phi\in \mathrm{H}^1(\Sigma\setminus B; \mathbb{Z}/2\mathbb{Z})$ the Poincar\'e-Lefschetz dual of $[\Gamma^{\dagger}]$ sending a class $[\alpha]$ to its algebraic intersection with $[\Gamma^{\dagger}]$ modulo $2$. 

\begin{definition}\label{de: cover} 
\begin{enumerate}
\item The covering $ \pi : \hat{\Sigma}(\Delta) \rightarrow \Sigma$ is the double covering of $\Sigma$ branched along $B$ defined by $\phi$. 
\item  Write $\hat{\mathcal{P}}$ the lift in $\hat{\Sigma}$ of $\mathcal{P}$, $\hat{B}$ the lift of the set of branched points $B$ and $\sigma: \hat{\Sigma}\rightarrow \hat{\Sigma}$ the covering involution. 
The involutive punctured surface associated to the triangulated punctured surface $(\mathbf{\Sigma}, \Delta)$ is $(\hat{\mathbf{\Sigma}}, \sigma):=(\hat{\Sigma}, \hat{\mathcal{P}}, \sigma)$. 
\item Let $\hat{\Gamma}^{\dagger}$ be the lift of $\Gamma^\dagger$.  A \emph{leaf} of $\mathbf{\hat{\Sigma}}$ is a connected component of $\hat{\Sigma}\setminus \hat{\Gamma}^{\dagger}$. 
	A \emph{leaf labeling} of $\mathbf{\hat{\Sigma}}$ is a labeling of its leaves by $1$ or $2$ such that any two leaves that are separated by an edge of $\hat{\Gamma}^{\dagger}$  have different labels.
\end{enumerate}
\end{definition}
In Figure \ref{figtrianglecov} are represented some double branched coverings. 

\begin{figure}[!h] 
	\centerline{\includegraphics[width=12cm]{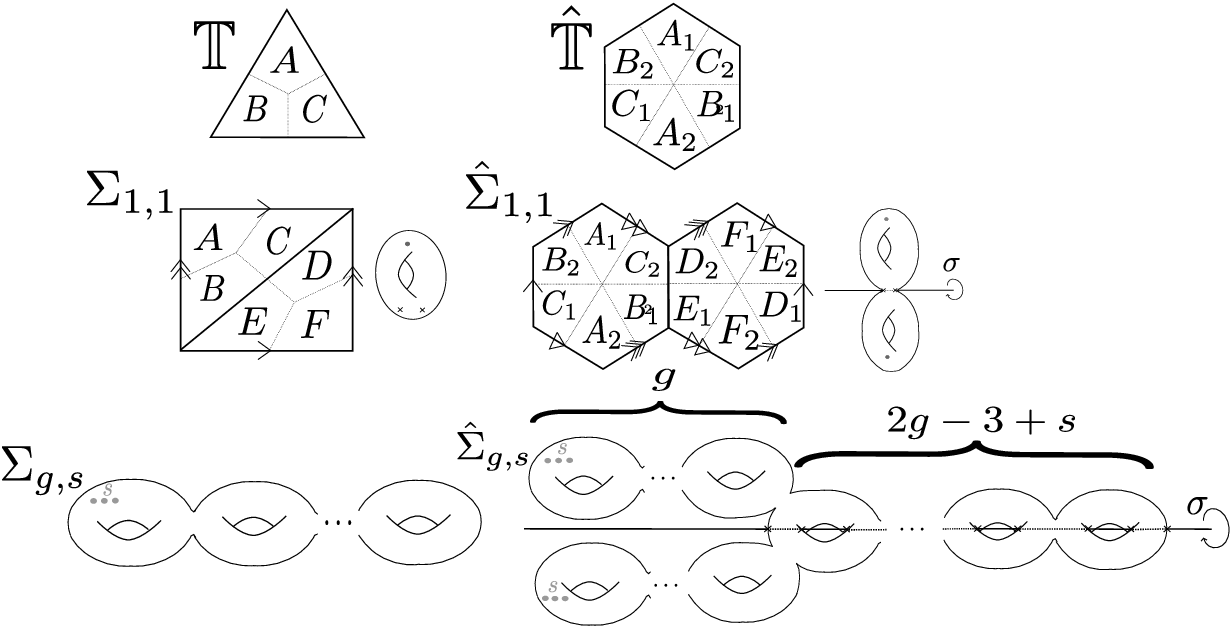} }
	\caption{On the top, the triangle and its double branched covering. The dashed lines represent the dual graph $\Gamma^{\dagger}$, its edges are the branched cuts and its vertices the branched point of the covering. The covering involution is the central symmetry along the branched point and we draw a leaf labelling.  In the middle, a triangulated once-punctured torus and the corresponding covering, a twice-punctured genus two surface. At the bottom, the double covering of a genus $g$ surface with $s$ punctures. } 
	\label{figtrianglecov} 
\end{figure}

\begin{remark} \begin{enumerate}
\item	 These double branched coverings were considered independently by Bonahon-Wong in \cite{BonahonWong2} and Gaiotto-Hollands-Moore-Neitzke in \cite{GMN13,HollandsNeitzke}. 
	 
\item If $\mathbf{\hat{\Sigma}}$ is connected, it has exactly two different leaf labelings. 
\item Punctured surfaces can be glued along their boundary arcs. In particular, a triangulated punctured surface is the result of such a gluing, starting from a disjoint union of  triangles (\emph{i.e.} discs with three punctures on their boundary). 
	The $2$--fold branched coverings of these triangulated surfaces are obtained by gluing hexagons.  
	Locally, for each face $\T$ of $\Delta$ with branching point $b$, one has the covering $\pi_{\T}\co\Hex\to \T$ as given in Figure \ref{figtrianglecov}: the covering involution $\sigma_{\T}$ is the central symmetry with fixed-point $\pi^{-1}(b)$. 
\end{enumerate}
\end{remark}

	For later use, let us express the genus $\hat{g}$ of  $\hat{\Sigma}$ in terms of the following data on $(\Sigma,\mathcal{P})$:  
	the genus $g$ of $\Sigma$; 
	the number $n^{ev}_{\partial}$ of boundary components with an even number of punctures; 
	the number $n^{odd}_{\partial}$ of boundary components with an odd number of punctures; 
	the number $\mathring{s}$ of inner punctures;  and, 
	the number  $s_{\partial}$ of boundary punctures. 

\begin{lemma}\label{lemmagenus}
	The genus of the surface $\hat{\Sigma}$ is 
	$$ \hat{g}= 4g-3+n_{\partial}^{ev} + \frac{3}{2}n_{\partial}^{odd}+ \mathring{s}+\frac{1}{2}s_{\partial}. $$
\end{lemma}

\begin{proof}This is a mere application of the  Riemann-Hurwitz formula (see \textit{e.g.} \cite[Section IV.2]{Hart}) which implies that the Euler characteristic of $\hat{\Sigma}$ is $\chi(\hat{\Sigma})=2\chi(\Sigma)-|B|. \quad$ 
\end{proof}

\subsection{Equivariant $\mathbb{C}^*$ character varieties}

In this section we define  the $\mathbb{C}^*$ character varieties of  punctured surfaces and their equivariant versions, for involutive punctured surfaces. 
	These definitions essentially rely on a relative version of the intersection product, that we now describe.

	Let  $\mathbf{\Sigma}=(\Sigma, \mathcal{P})$ be a punctured surface. 
	Let $c_1, c_2 : [0,1] \rightarrow \Sigma_{\mathcal{P}}$ be two cycles in $\mathrm{Z}_1(\Sigma_{\mathcal{P}}, \partial \Sigma_{\mathcal{P}} ;  \mathbb{Z})$. 
	Suppose they intersect transversally in the interior of $\Sigma_{\mathcal{P}}$ along simple crossings; in particular, for each intersection point $v$, one has a basis $(e_1,e_2)$ of the tangent space of $\Sigma_{\mathcal{P}}$ at $v$ that is formed by the tangent vectors $e_1$ and $e_2$ of $c_1$ and $c_2$ respectively. If the orientation of this basis agrees with that of $\Sigma$, then we set $\varepsilon(v):= +1$, and we set $\varepsilon(v):=-1$ otherwise. 
	
	For  a boundary arc $b$ of $\Sigma_{\mathcal{P}}$, let $S(b)$ be the set of pairs $(v_1, v_2)$ such that $v_i \in c_i \cap b$. 
	Note that $c_1$ and $c_2$ do not intersect in $b$ by definition. For a pair $(v_1, v_2) \in S(b)$, we define  $\varepsilon(v_1, v_2) \in \{ \pm 1 \}$ as follows. 
	Isotope $c_1$ around $b$ to bring $v_1$ at $v_2$. 
	The isotopy must preserve the transversality condition and must not create any new inner intersection point. 
	The resulting tangent vectors of the curves at $v_2$ form a basis $(e_1,e_2)$ of the tangent space; as before, the orientation of this basis gives an element $\varepsilon(v_1, v_2) \in \{ \pm 1 \}$ according to the orientation of $\Sigma$:   $\varepsilon(v_1, v_2) = +1$ if the orientation of $(e_1, e_2)$ agrees with that of $\Sigma$, and $\varepsilon(v_1, v_2)=-1$ otherwise. 
	For every $c_1$ and $c_2$ as above, we set 
	\begin{equation}\label{eq: rel inters form}
	(c_1, c_2) := \sum_b \sum_{(v_1, v_2)\in S(b)} \frac{1}{2} \varepsilon(v_1, v_2)  + \sum_{v\in c_1\cap c_2} \varepsilon(v). 
	\end{equation}

	\begin{definition}\label{de: rel inters form} The skew-symmetric form  $(\cdot, \cdot) : \mathrm{H}_1(\Sigma_{\mathcal{P}}, \partial \Sigma_{\mathcal{P}} ;  \mathbb{Z})^{\otimes 2} \rightarrow \frac{1}{2} \mathbb{Z}$ is the map induced by the linear extension of \eqref{eq: rel inters form}. 
	\end{definition}
By  \cite[Lemma $3.15$]{KojuTriangularCharVar}, the pairing $(c_1, c_2)$ only depends on the relative homology classes $[c_1], [c_2] \in  \mathrm{H}_1(\Sigma_{\mathcal{P}}, \partial \Sigma_{\mathcal{P}} ;  \mathbb{Z})$, so that the above form is well-defined.

\begin{definition} The $\mathbb{C}^*$ \textit{character variety} is the Poisson affine variety $\mathcal{X}_{\mathbb{C}^*}(\mathbf{\Sigma})$ whose algebra of regular functions is the group algebra $\mathcal{O}[\mathcal{X}_{\mathbb{C}^*}(\mathbf{\Sigma})]:= \mathbb{C}[\mathrm{H}_1(\Sigma_{\mathcal{P}}, \partial \Sigma_{\mathcal{P}} ;  \mathbb{Z})]$ with the Poisson bracket defined by: 
$$ \{ [c_1], [c_2] \} := -\frac{1}{2}([c_1], [c_2] ) [c_1 + c_2] \mbox{, for any }[c_1], [c_2] \in \mathrm{H}_1(\Sigma_{\mathcal{P}}, \partial \Sigma_{\mathcal{P}} ;  \mathbb{Z}). $$
\end{definition}

The following lemma implies that  $\mathcal{X}_{\mathbb{C}^*}(\mathbf{\Sigma})$ is a torus.

\begin{lemma}\label{lemma_free} The $\mathbb{Z}$-module $\mathrm{H}_1(\Sigma_{\mathcal{P}}, \partial \Sigma_{\mathcal{P}} ;  \mathbb{Z})$ is free.
\end{lemma}

\begin{proof}
Without lost of generality, we can assume that $\Sigma$ is connected. If $\mathbf{\Sigma}$ does not contain any boundary arc, then the fact that  $\mathrm{H}_1(\Sigma_{\mathcal{P}}, \partial \Sigma_{\mathcal{P}} ;  \mathbb{Z})=\mathrm{H}_1(\Sigma_{\mathcal{P}} ;  \mathbb{Z})$ is free is well-known. Else, for each boundary arc $a$ choose one point $v_a\in a$ and set $V:= \{v_a\}_a$. Consider a spine $\Gamma \subset \Sigma_{\mathcal{P}}$, that is an embedded graph whose set of vertices is $V$ and such that $\Sigma_{\mathcal{P}}$ retracts on $\Gamma$ (see e.g. \cite{KojuPresentationSSkein} for a correspondence between ciliated graphs and open punctured surfaces). Then the embedding of pairs $\iota: (\Gamma, V) \hookrightarrow (\Sigma_{\mathcal{P}}, \partial \Sigma_{\mathcal{P}})$ is a homotopy retract, so induces an isomorphism 
$$  \mathrm{H}_1(\Gamma, V; \mathbb{Z}) \cong \mathbb{Z}^{\mathcal{E}(\Gamma)}\xrightarrow[\cong]{\iota_*}\mathrm{H}_1(\Sigma_{\mathcal{P}}, \partial \Sigma_{\mathcal{P}} ;  \mathbb{Z}), $$
where $\mathcal{E}(\Gamma)$ is the set of edges of $\Gamma$.

\end{proof}

 Consider an involutive punctured surface $(\mathbf{\Sigma}, \sigma)$ with set of fixed points $B$. Denote by $\mathrm{H}^{\sigma}_1(\Sigma_{\mathcal{P}}\cup B, \partial \Sigma_{\mathcal{P}}\cup B ;  \mathbb{Z})\subset \mathrm{H}_1(\Sigma_{\mathcal{P}}, \partial \Sigma_{\mathcal{P}}\cup B ;  \mathbb{Z})$ the sub-group of elements $[c]$ such that $\sigma_*([c])= - [c]$. 
\begin{definition}
The \textit{equivariant} $\mathbb{C}^*$ {character variety} is the Poisson affine variety $\mathcal{X}^{\sigma}_{\mathbb{C}^*}(\mathbf{\Sigma})$ whose algebra of regular functions is the group algebra $\mathcal{O}[\mathcal{X}^{\sigma}_{\mathbb{C}^*}(\mathbf{\Sigma})]:= \mathbb{C}[\mathrm{H}^{\sigma}_1(\Sigma_{\mathcal{P}\cup B}, \partial \Sigma_{\mathcal{P}\cup B} ;  \mathbb{Z})]$
 with Poisson bracket obtained by restricting the one of  $\mathbb{C}[\mathrm{H}_1(\Sigma_{\mathcal{P}\cup B}, \partial \Sigma_{\mathcal{P}\cup B} ;  \mathbb{Z})]$.
 \end{definition}


 \subsection{Equivariant $\mathbb{C}^*$ skein algebras}

 	In this section we define  a $\mathbb{C}^*$  skein algebra, as well as an equivariant version for involutive punctured surfaces. For a closed punctured surface, the  $\mathbb{C}^*$  skein algebra recovers Gelca-Uribe's skein algebra in \cite{GU}. 
 The extension to surfaces with boundary is inspired from L\^e's stated skein algebra, which is concerned with Kauffman-bracket relations, related to $\SL_2(\mathbb{C})$,  while here, we are concerned with the self-linking number  related to $\mathbb{C}^*$.   
 We show that the (equivariant)  $\mathbb{C}^*$   skein algebra is a deformation quantization of the (equivariant)  $\mathbb{C}^*$  character variety.

 \subsubsection{The definition}

  Let $\mathbf{\Sigma}$ be a punctured surface. A \textit{tangle} is an oriented, compact framed, properly embedded $1$-dimensional manifold $T\subset \Sigma_{\mathcal{P}}\times (0,1)$ such that for every point of $\partial T \subset \partial \Sigma_{\mathcal{P}} \times (0,1)$ the framing is parallel to the $(0,1)$ factor and points to the direction of $1$.   The \textit{height} of a point $(v,h)\in \Sigma_{\mathcal{P}} \times (0,1)$ is $h$. 
  For a boundary arc $b$ and a tangle $T$,  the points of $\partial_b T := \partial T \cap b\times(0,1)$ are ordered by their heights. 
  The tangle is said \textit{Weyl-ordered} if for every boundary arc $b$ the points of $\partial_b T$ have the same height. 
  A tangle has \textit{vertical framing} if for each of its points, the framing is parallel to the $(0,1)$ factor and points in the direction of $1$. Two tangles are isotopic if they are isotopic through the class of tangles that preserve the boundary height orders. The empty set is by definition a tangle only isotopic to itself.

\vspace{2mm}
\par 
Every tangle is isotopic to a tangle with vertical framing and such that it is in general position with the first factor projection  $\pi : \Sigma_{\mathcal{P}}\times (0,1)\rightarrow \Sigma_{\mathcal{P}}$, that is such that $\pi_{\big| T} : T \rightarrow \Sigma_{\mathcal{P}}$ is an immersion with at most transversal double points only in the interior of $\Sigma_{\mathcal{P}}$. We call \textit{diagram} of $T$ the image $D=\pi(T)$ together with  its orientation and the over/undercrossing information at each double point. An isotopy class of diagram $D$ together with an order of $\partial_b D=\partial D\cap b$ for each boundary arc $b$, determines uniquely an isotopy class of tangle. 
Diagrams are used to depict tangles. Locally, if two boundary points of a tangle are consecutive for the height order, then we represent this order by drawing an arrow on the boundary, such that the order increases with the direction of the arrow. If the two points have same order, then we do not draw any arrow. 

The data of an isotopy class of diagrams and  boundary arc orientations uniquely defines an isotopy class of tangles. Two such classes of diagrams represent the same class of tangles if and only if we can pass from one to the other by a succession of elementary moves of Figure \ref{figmoves}.

\begin{figure}[!h] 
\centerline{\includegraphics[width=10cm]{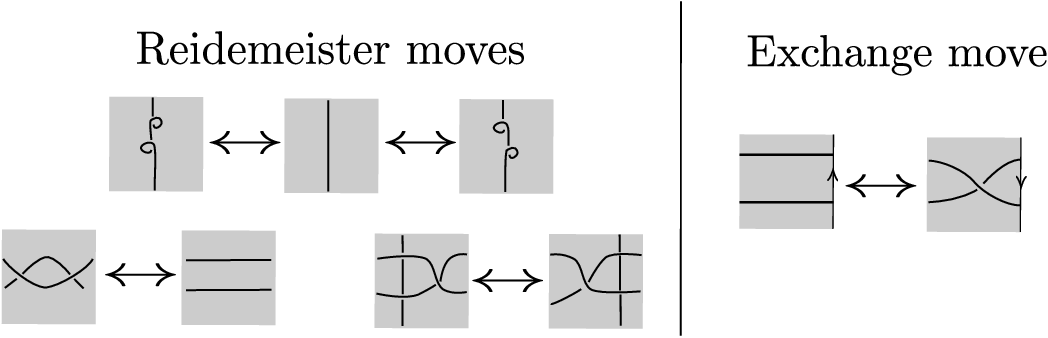} }
\caption{Elementary moves relating the diagrams of two isotopic tangles.} 
\label{figmoves} 
\end{figure}

 \par  Let $\mathcal{R}$ be a commutative unital ring and $\omega^{1/2}\in \mathcal{R}^{\times}$ an invertible element.
 \begin{definition}
 The $\mathbb{C}^*$ \textit{skein algebra}  $\mathcal{S}^{\mathbb{C}^*}_{\omega^{1/2}}(\mathbf{\Sigma})$ is  the free $\mathcal{R}$-module generated by isotopy classes of tangles, modulo the skein relations \eqref{eq: C skein 1}, \eqref{eq: C skein 2} and \eqref{eq: C skein 3}. 
 	 The product is given by stacking tangles:  $[T_1][T_2]$ is the class of a tangle $T_1\cup T_2$ where $T_2$ has been isotoped in $(0,\frac{1}{2})$ and $T_1$ in $(\frac{1}{2}, 1)$.
  \end{definition}
\begin{align}
\label{eq: C skein 1}
\omega
\begin{tikzpicture}[>=stealth,thick, arrow/.style={->},baseline=-0.4ex,scale=0.5]	
\draw [fill=gray!60,gray!45] (-.6,-.6)  rectangle (.6,.6)   ;
\draw[line width=1.1,-] (-0.4,-0.52) -- (-.1,-.12);
\draw[line width=1.1,->] (.1,0.12) -- (0.4,.52);
\draw[line width=1.1,->] (0.4,-0.52) -- (-.4,.53);
\end{tikzpicture} 
=\omega^{-1}
\begin{tikzpicture}[>=stealth,thick, arrow/.style={->},baseline=-0.4ex,scale=0.5]	
\draw [fill=gray!60,gray!45] (-.6,-.6)  rectangle (.6,.6)   ;
\draw[line width=1.1,->] (-0.4,-0.52) -- (.4,.53);
\draw[line width=1.1,-] (0.4,-0.52) -- (0.1,-0.12);
\draw[line width=1.1,->] (-0.1,0.12) -- (-.4,.53);
\end{tikzpicture}
=
\begin{tikzpicture}[>=stealth,thick, arrow/.style={->},baseline=-0.4ex,scale=0.5] 
\draw [fill=gray!60,gray!45] (-.6,-.6)  rectangle (.6,.6)   ;
\draw[line width=1.1,->] (-0.4,-0.52) ..controls +(.3,.5).. (-.4,.53);
\draw[line width=1.1,arrow] (0.4,-0.52) ..controls +(-.3,.5).. (.4,.53);
\end{tikzpicture}
\hspace{.5cm} &\text{ and }\hspace{.5cm}
\begin{tikzpicture}[baseline=-0.4ex,scale=0.5,Circ/.style={circle, fill, minimum width=12pt},
InLineArrow/.style={
	draw,thick, postaction={decorate,decoration={markings,mark=at position #1 with {\fill(0:2.5pt) -- (120:2.5pt) -- (240:2.5pt);}}}}
]
\node[](P) at (-30:0.4){};
\draw [fill=gray!60,gray!45] (-.6,-.6)  rectangle (.6,.6)   ;  
\draw[InLineArrow=0.41] (P) arc (150-180:150-180-360:0.4) node [pos=0.41, anchor=180-180,inner sep=7pt]{};
\end{tikzpicture}
= 
\begin{tikzpicture}[baseline=-0.4ex,scale=0.5,Circ/.style={circle, fill, minimum width=12pt},
InLineArrow/.style={
	draw,thick, postaction={decorate,decoration={markings,mark=at position #1 with {\fill(0:2.5pt) -- (120:2.5pt) -- (240:2.5pt);}}}}
]
   \node[](P) at (30:0.4){};
\draw [fill=gray!60,gray!45] (-.6,-.6)  rectangle (.6,.6)   ;  
\draw[InLineArrow=0.41] (P) arc (-150+180:-150+180+360:0.4) node [pos=0.41, anchor=-180+180,inner sep=7pt]{};
\end{tikzpicture}
=\emptyset. 
\\
\label{eq: C skein 2}
\begin{tikzpicture}[>=stealth,arrow/.style={->},baseline=-0.4ex,scale=0.5] 
\draw [fill=gray!60,gray!45] (-.75,-.75)  rectangle (.4,.75)   ;
\draw[line width=1.1,->] (-.75,0.3) to (-.25,.3);
\draw[-] (0.4,-0.75) to (.4,.75);
\draw[line width=1.1] (-.75,-0.3) to (-.3,-.3);
\draw[line width=1.05] (-.3,0) ++(-90:.3) arc (-90:90:.3);
\end{tikzpicture}
=
\heightexchor{-}{->}{<-}
= \omega^{-1/2}
\heightexchor{<-}{->}{<-}
= \omega^{1/2}
\heightexchor{->}{->}{<-} 
\hspace{.5cm}&\text{ and }\hspace{.5cm}
\begin{tikzpicture}[>=stealth,arrow/.style={->},baseline=-0.4ex,scale=0.5] 
\draw [fill=gray!60,gray!45] (-.75,-.75)  rectangle (.4,.75)   ;
\draw[line width=1.1,<-] (-.75,0.3) to (-.25,.3);
\draw[-] (0.4,-0.75) to (.4,.75);
\draw[line width=1.1] (-.75,-0.3) to (-.3,-.3);
\draw[line width=1.05] (-.3,0) ++(-90:.3) arc (-90:90:.3);
\end{tikzpicture}
=
\heightexchor{-}{<-}{->}
= \omega^{-1/2}
\heightexchor{<-}{<-}{->}
= \omega^{1/2}
\heightexchor{->}{<-}{->}
\\
\label{eq: C skein 3}
\heightexchor{-}{<-}{<-}
= \omega^{1/2}
\heightexchor{<-}{<-}{<-}
= \omega^{-1/2}
\heightexchor{->}{<-}{<-}	
\hspace{.5cm} &\text{ and } \hspace{.5cm}
\heightexchor{-}{->}{->}
= \omega^{1/2}
\heightexchor{<-}{->}{->}
= \omega^{-1/2}
\heightexchor{->}{->}{->}	
\end{align}

\begin{remark}
	As a straightforward consequence of the relations  \eqref{eq: C skein 1}, \eqref{eq: C skein 2} and \eqref{eq: C skein 3}, one has  the following relations. 
	\begin{equation}\label{eq:skeinconsequences}
	\begin{tikzpicture}[>=stealth,thick, arrow/.style={->},baseline=-0.4ex,scale=0.5] 
	\draw [fill=gray!60,gray!45] (-.6,-.6)  rectangle (.6,.6)   ;
	\draw[line width=1.1,->] (-0.4,-0.52) ..controls +(.3,.5).. (-.4,.53);
	\draw[line width=1.1,arrow] (.4,.53) ..controls +(-.3,-.5).. (0.4,-0.52);
	\end{tikzpicture}
	=
	\begin{tikzpicture}[>=stealth,thick, arrow/.style={->},baseline=-0.4ex,scale=0.5] 
	\begin{scope}[shift={(0,0)},rotate=90]
	\draw [fill=gray!60,gray!45] (-.6,-.6)  rectangle (.6,.6)   ;
	\draw[line width=1.1,->] (-.4,.53)  ..controls +(.3,-.5)..  (-0.4,-0.52);
	\draw[line width=1.1,arrow] (0.4,-0.52)..controls +(-.3,.5).. (.4,.53) ;
	\end{scope}
	\end{tikzpicture}
	~	 \emph{(homology relation) }  	 
	\text{ and }~ 
	\begin{tikzpicture}[baseline=-0.4ex,scale=0.5] 
	\draw [fill=gray!60,gray!45] (-.75,-.75)  rectangle (.4,.75)   ;
	\draw[-] (.4,-0.75) to (.4,.75);
	\end{tikzpicture}
	=
	\begin{tikzpicture}[>=stealth,arrow/.style={->},baseline=-0.4ex,scale=0.5] 
	\begin{scope}[shift={(0,0)},rotate=180]
	\draw [fill=gray!60,gray!45] (-.75,-.75)  rectangle (.4,.75)   ;
	\draw[line width=1.1,<-] (-.75,0.3) to (-.25,.3);
	\draw[-] (-.75,-0.75) to (-.75,.75);
	\draw[line width=1.1] (-.75,-0.3) to (-.3,-.3);
	\draw[line width=1.05] (-.3,0) ++(-90:.3) arc (-90:90:.3);
	\end{scope}
	\end{tikzpicture}
	=
	\begin{tikzpicture}[>=stealth,arrow/.style={->},baseline=-0.4ex,scale=0.5] 
	\begin{scope}[shift={(0,0)},rotate=180]
	\draw [fill=gray!60,gray!45] (-.75,-.75)  rectangle (.4,.75)   ;
	\draw[line width=1.1,->] (-.75,0.3) to (-.25,.3);
	\draw[-] (-.75,-0.75) to (-.75,.75);
	\draw[line width=1.1] (-.75,-0.3) to (-.3,-.3);
	\draw[line width=1.05] (-.3,0) ++(-90:.3) arc (-90:90:.3);
	\end{scope}
	\end{tikzpicture}
	~ \emph{(trivial relation)}.
	\end{equation}
\end{remark}

\subsubsection{Interpretation through relative homology}

	Let us show that the $\mathbb{C}^*$ skein algebra  deforms the $\mathbb{C}^*$  character variety. 
	This essentially extends Gelca-Uribe's result \cite[Theorem 4.5]{GU} to surfaces with boundary. 
   	Recall the relative intersection form of Definition \ref{de: rel inters form}. 
   	Let $\mathcal{R}_{\omega^{1/2}}[\mathrm{H}_1(\Sigma_{\mathcal{P}}, \partial\Sigma_{\mathcal{P}}; \mathbb{Z})]$ be the $\mathcal{R}$-module $\mathcal{R}[\mathrm{H}_1(\Sigma_{\mathcal{P}}, \partial\Sigma_{\mathcal{P}}; \mathbb{Z})]$ with the product 
   	\begin{equation}\label{eq: prod relat hom}
   	[\alpha]\cdot [\beta] = \omega^{([\alpha], [\beta])} [\alpha + \beta] \mbox{, for }[\alpha], [\beta]\in \mathrm{H}_1(\Sigma_{\mathcal{P}}, \partial\Sigma_{\mathcal{P}}; \mathbb{Z}).    
   	\end{equation}	
   	We now describe an isomorphism of algebras  
   	\begin{equation}\label{eq: h}
   	h: \mathcal{R}_{\omega^{1/2}}[\mathrm{H}_1(\Sigma_{\mathcal{P}}, \partial \Sigma_{\mathcal{P}}; \mathbb{Z})] \rightarrow  \mathcal{S}^{\mathbb{C}^*}_{\omega^{1/2}}(\mathbf{\Sigma}). 
   	\end{equation}
   	
   	First let us show that,  as a set, $\mathrm{H}_1(\Sigma_{\mathcal{P}}, \partial \Sigma_{\mathcal{P}}; \mathbb{Z})$ is a basis of the $\mathcal{R}$--module $\mathcal{S}^{\mathbb{C}^*}_{\omega^{1/2}}(\mathbf{\Sigma})$, which provides the linear map $h$.  
   	Let us identify relevant diagrams: %
   	a diagram is called \textit{reduced} if it has neither crossings nor contractible components. Contractibility is understood relatively to the boundary of $\Sigma_{\mathcal{P}}$.   
   	In particular, reduced diagrams give a basis for the group of cycles $\mathrm{Z}_1(\Sigma_{\mathcal{P}}, \partial \Sigma_{\mathcal{P}}; \mathbb{Z})$. 
   	To any reduced diagram $D$ corresponds a Weyl-ordered tangle;  we denote its class by $[D]\in \mathcal{S}^{\mathbb{C}^*}_{\omega^{1/2}}(\mathbf{\Sigma})$. 
   	Note that if $T$ is a tangle, there is a unique pair of the class $[D]$ of a reduced diagram $D$  and an integer $n$ such that $[T]=\omega^{n/2} [D]$ in $\mathcal{S}^{\mathbb{C}^*}_{\omega^{1/2}}(\mathbf{\Sigma})$. 
   	Therefore, the Weyl-ordered classes associated to the reduced diagrams generate the $\mathbb{C}^*$ skein algebra.  
   	In other words, we have a surjective $\mathcal{R}$--linear map 
   	\begin{equation*}
   	\wid{h}\co \mathcal{R}[Z_1(\Sigma_{\mathcal{P}}, \partial\Sigma_{\mathcal{P}}; \mathbb{Z})] \to \mathcal{S}^{\mathbb{C}^*}_{\omega^{1/2}}(\mathbf{\Sigma}),   
   	\end{equation*}
   	given by $\wid{h}(D)=[D]$ for every reduced diagram $D$. 
   	\begin{lemma}\label{lemmaskeinspan}
   	The linear map $\wid{h}$ induces an isomorphism of $\mathcal{R}$--modules on the homology; this is $h$.   
   	\end{lemma}
   		 	\begin{proof} 
   		First we show that if $D$ is a reduced diagram with trivial relative homology, that is  $[D]=1$ in $\mathrm{H}_1(\Sigma_{\mathcal{P}}, \partial \Sigma_{\mathcal{P}}; \mathbb{Z})$, then $\wid{h}(D)=1$.  
   		If  $\partial D =\emptyset$, then this is the proof of \cite[Theorem 4.5]{GU} which we briefly sketch here.
		    		Let $S\subset \Sigma_{\mathcal{P}}$ be a compact oriented sub-surface with oriented boundary $D$. When $S$ is homeomorphic to a cylinder $D^2\times [0,1]$ or a pair of pants, \textit{i.e.} a disc with three open sub-discs removed, it follows from the homological relation of \eqref{eq:skeinconsequences} that $\wid{h}(D)=1$ (see \cite[Figure $6$]{GU}). Since every surface $S$ is obtained by gluing along boundary components such elementary cobordisms, we obtain that $\wid{h}(D)=1$ in the general case by induction on the number of elementary cobordisms in a pants decomposition of $S$.
   		Next suppose that $\partial D \neq \emptyset$. 
   		Since $D$ has trivial relative homology, there exists a finite collection $D_{\partial}$ of oriented arcs in the boundary and an oriented surface $S\subset \Sigma_{\mathcal{P}}$ such that $\partial S=D\cup D_{\partial}$. 
   		Moreover, because of the relations of \eqref{eq:skeinconsequences}, one has $[D]=[D\cup D_{\partial}]\in  \mathcal{S}^{\mathbb{C}^*}_{\omega^{1/2}}(\mathbf{\Sigma})$. 
   		We conclude by applying the preceding case to $D\cup D_{\partial}$. 
   		
   		Now let us show that if $D$ and $D'$ are  two reduced diagrams such that $\wid{h}(D)=\wid{h}(D')$, then they have same relative homology class. Two such diagrams can be obtained one from each other by means of elementary moves (isotopy, Reidemeister moves followed by resolution of the crossings and the defining relations \eqref{eq: C skein 1}, \eqref{eq: C skein 2}). 
   		Since these moves preserve the relative homology class, the result follows. 
   	\end{proof}

     \begin{proposition}\label{skeinalgstructure}
  The $\mathcal{R}$--linear isomorphism	$h$ is a morphism of algebras. 
   \end{proposition}
 \begin{proof}
	This is a direct consequence of the defining skein relations \eqref{eq: C skein 1}, \eqref{eq: C skein 2} and of the definition of the relative  intersection form.
\end{proof}
\begin{remark}
If $\mathcal{R}=\mathbb{C}$ and $\omega^{1/2}=+1$, then the $\mathbb{C}^*$ skein algebra is isomorphic to the algebra of regular functions of the $\mathbb{C}^*$ character variety $\mathcal{X}_{\mathbb{C}^*}(\mathbf{\Sigma})$. 	
\end{remark}
     
    \subsubsection{The equivariant version}\label{sec: equiv skein}
    
    Consider an involutive punctured surface $(\mathbf{\Sigma}, \sigma)$, where $\sigma$ has the set of fixed points $B$,  and a commutative unital ring $\mathcal{R}$ with an invertible element $\omega \in \mathcal{R}^{\times}$. 
    In the definition of $\mathbb{C}^*$ skein algebra, we need a square root of $\omega$, however the equivariant skein algebra we are about to define does not depend on this square root. We thus introduce it artificially as follows.
    Denote by $\mathcal{R}' := \quotient{\mathcal{R}[\omega^{1/2}]}{\left( (\omega^{1/2})^2 - \omega\right)}$. We have a natural embedding $\mathcal{R}\hookrightarrow \mathcal{R}'$. 

    \begin{definition}
    The \textit{equivariant} $\mathbb{C}^*$  {skein algebra} is the $\mathcal{R}$-submodule $\mathcal{S}_{\omega}^{\mathbb{C}^*, \sigma}(\mathbf{\Sigma}) \subset \mathcal{S}_{\omega^{1/2}}^{\mathbb{C}^*}(\Sigma, \mathcal{P}\cup B) $ spanned by classes $[D]$ of Weyl-ordered diagrams such that $\sigma(D)= D^{-1}$, where $D^{-1}$ denotes the diagram $D$ with opposite orientation. Here $\mathcal{S}_{\omega^{1/2}}^{\mathbb{C}^*}(\Sigma, \mathcal{P}\cup B) $ is seen as an $\mathcal{R}'$-module.
        \end{definition}
      
\begin{remark}\label{rk: h equiv}
Using the isomorphism $h$,  the set of $\sigma_*$ anti-invariant classes $\mathrm{H}^{\sigma}_1(\Sigma_{\mathcal{P}\cup B}, \partial \Sigma_{\mathcal{P}\cup B}; \mathbb{Z})\subset \mathrm{H}_1(\Sigma_{\mathcal{P}\cup B}, \partial \Sigma_{\mathcal{P}\cup B}; \mathbb{Z})$ is an $\mathcal{R}$--basis of  $\mathcal{S}_{\omega}^{\mathbb{C}^*, \sigma}(\mathbf{\Sigma})$. Moreover, the isomorphism of algebras $h$ restricts to an isomorphism of algebras, for the restricted structures.  
\end{remark}

     By Proposition \ref{skeinalgstructure}, $\mathcal{S}_{\omega}^{\mathbb{C}^*, \sigma}(\mathbf{\Sigma})$  is an $\mathcal{R}$-algebra isomorphic to the algebra whose underlying $\mathcal{R}$-modules is  $\mathcal{R}[\mathrm{H}^{\sigma}_1(\Sigma_{\mathcal{P}}, \partial \Sigma_{\mathcal{P}}; \mathbb{Z})]$ and product $ [\alpha]\cdot [\beta] = \omega^{([\alpha], [\beta])} [\alpha + \beta]$. 
     
     In particular, one has the following. 
\begin{corollary}\label{cor: skein equ iso to reg. func. of char.}
	For $\mathcal{R}=\C$ and $\omega=+1$, the algebra $\mathcal{S}_{+1}^{\mathbb{C}^*, \sigma}(\mathbf{\Sigma})$  is isomorphic to the algebra of regular functions of $\mathcal{X}_{\mathbb{C}^*}^{\sigma}(\mathbf{\Sigma})$. 
\end{corollary}

\subsubsection{Triangular decompositions} \label{sec: decompo skein eq}

Let $(\mathbf{\Sigma},\Delta)$  be  a triangulated punctured surface. 
For a reduced diagram $D$ of $\Sigma_{\mathcal{P}}$ that is transversed to the edges of the triangulation, denote by $D^{\T}$ its intersection  with the triangle $\T$. 
The assignment $i([D]):= \sum_{\mathbb{T}\in F(\Delta)} [D^{\mathbb{T}}]$ extends to an injective morphism of groups 
\begin{equation}\label{eq: i gp}
i: \mathrm{H}_1(\Sigma_{\mathcal{P}}, \partial \Sigma_{\mathcal{P}}; \mathbb{Z}) \hookrightarrow \oplus_{\mathbb{T} \in F(\Delta)} \mathrm{H}_1(\mathbb{T}, \partial \mathbb{T}; \mathbb{Z}). 
\end{equation} 

 \begin{lemma}
 	The group morphism $i$ preserves the relative intersection form. 
 \end{lemma}
\begin{proof} This is a straightforward consequence of the definitions. \end{proof}

Via the isomorphism $h$ of \eqref{eq: h}, one deduces the following injective morphism of algebras: 
\begin{equation}
i^{\Delta}\co \mathcal{S}^{\mathbb{C}^*}_{\omega^{1/2}}(\mathbf{\Sigma}) \hookrightarrow \otimes_{\mathbb{T}\in F(\Delta)} \mathcal{S}^{\mathbb{C}^*}_{\omega^{1/2}}(\T). 
\end{equation}  

Let $\pi\co \hat{\Sigma}(\Delta) \to \Sigma $ be the $2$--fold branched covering of $(\mathbf{\Sigma},\Delta)$ as in Definition \ref{de: cover}.  As mentioned in  Section \ref{sec: branch cov}, $\hat{\Sigma}(\Delta)$ decomposes into hexagons $\Hex$,  which are indexed by the triangles of $\Delta$; one has an injection similar to \eqref{eq: i gp}, call it $\hat{i}$. The morphism $\hat{i}$ commutes with the involutions (of $ \hat{\Sigma}(\Delta)$ and of the hexagons $\Hex$) and preserves the relative intersection form. 
Therefore, one has  an injective morphism of algebras 
\begin{equation}\label{eq: i decomp skein equiv}
 i^{\Delta} \co \mathcal{S}_{\omega}^{\mathbb{C}^*, \sigma} (\mathbf{\hat{\Sigma}}) \hookrightarrow \otimes_{\mathbb{T}\in F(\Delta)} \mathcal{S}_{\omega}^{\mathbb{C}^*, \sigma}(\hat{\mathbb{T}}).
\end{equation}

    \subsection{Tensor decomposition of equivariant $\mathbb{C}^*$ skein algebras}

 In Section \ref{sec: decompo skein eq}  we described the behaviour of the equivariant $\mathbb{C}^*$ skein algebra under gluing of punctured surfaces along boundary arcs. Here we describe the behaviour of the equivariant $\mathbb{C}^*$ skein algebra under another type of gluing: the sewing $\veebar$ of involutive punctured surfaces introduced in Definition \ref{def: sewing VEE}. 	This will be particularly  useful in the study of the representations of the equivariant $\mathbb{C}^*$ skein algebra. 
\vspace{2mm}
\par 

	Throughout this section $\mathcal{R}=\mathbb{C}$ and $\omega$ is a root of unity of odd order $N>1$. 
	Let  $(\mathbf{\Sigma}_1, \sigma_1)$ and $(\mathbf{\Sigma}_2, \sigma_2)$ be two involutive punctured surfaces. Let $\veebar_{\textbf{b}}(\Sigma_1\sqcup \Sigma_2)$ be the involutive punctured surface obtained from $(\mathbf{\Sigma}_1, \sigma_1)$ and $(\mathbf{\Sigma}_2, \sigma_2)$ by sewing them at a $\textbf{b}=\{b_1,b_2\}$ for $b_i$ a fixed-point of $\sigma_i$; see Definition \ref{def: sewing VEE}.    
	
	 \begin{proposition}\label{propgluinginvolutive} 
	 The algebras $\mathcal{S}_{\omega}^{\mathbb{C}^*,\sigma}(\veebar_{\textbf{b}}(\Sigma_1\sqcup \Sigma_2))$ and $\mathcal{S}_{\omega}^{\mathbb{C}^*,\sigma}(\mathbf{\Sigma}_1) \ot  \mathcal{S}_{\omega}^{\mathbb{C}^*,\sigma}(\mathbf{\Sigma}_2)$ are Morita equivalent. 
	 Moreover, if $\Sigma_1\sqcup \Sigma_2$ is either  closed, or contains only one boundary component or has exactly two boundary components that are exchanged by its involution, then the algebras $\mathcal{S}_{\omega}^{\mathbb{C}^*,\sigma}(\veebar_{\textbf{b}}(\Sigma_1\sqcup \Sigma_2))$ and $\mathcal{S}_{\omega}^{\mathbb{C}^*,\sigma}(\mathbf{\Sigma}_1) \ot  \mathcal{S}_{\omega}^{\mathbb{C}^*,\sigma}(\mathbf{\Sigma}_2)$ are isomorphic.
	\end{proposition}

     To prove Proposition \ref{propgluinginvolutive}, we first state a technical lemma. Consider a pair $\mathbb{E}=(E,\left(\cdot, \cdot\right)_E)$ where $E$ is a free $\mathbb{Z}$-module of finite rank and $\left(\cdot, \cdot\right)_E : E\times E \rightarrow \mathbb{Z}$ is a skew-symmetric bilinear form. The quantum torus    $\mathcal{T}_{\mathbb{E}}$ is the complex algebra with underlying vector space $\mathbb{C}[E]$ and product given by $[x]\cdot [y] := \omega^{(x,y)_E}[x+y]$. Given $e=(e_1,\ldots, e_n)$ a basis of $E$, the quantum torus $\mathcal{T}_{\mathbb{E}}$ is isomorphic to the complex  algebra generated by invertible elements $Z_{e_i}^{\pm 1}$ with relations $Z_{e_i}Z_{e_j}=\omega^{2(e_i,e_j)_E}Z_{e_j}Z_{e_i}$.  In particular, the equivariant $\mathbb{C}^*$ skein algebra $\mathcal{S}_{\omega}^{\mathbb{C}^*,\sigma}(\mathbf{\Sigma})$ is isomorphic to the quantum torus associated to the pair $(\mathrm{H}_1^{\sigma}(\Sigma_{\mathcal{P}\cup B}, \partial \Sigma_{\mathcal{P}\cup B} ; \mathbb{Z}), \left(\cdot, \cdot\right))$.

\begin{lemma}\label{lemma_preliminary} Consider two pairs $\mathbb{E}_1=(E_1,\left(\cdot, \cdot\right)_1)$  and $\mathbb{E}_2=(E_2,\left(\cdot, \cdot\right)_2)$ and suppose that we have a short exact sequence of $\mathbb{Z}$ modules
$$ 0 \rightarrow E_1 \xrightarrow{\iota}  E_2  \xrightarrow{o} \mathbb{Z}/2\mathbb{Z} \rightarrow 0$$
such that $(\iota(x), \iota(y))_2=(x,y)_1$ for all $x,y\in E_1$. Then the quantum tori $\mathcal{T}_{\mathbb{E}_1}$ and   $\mathcal{T}_{\mathbb{E}_2}$ are Morita equivalent.
\end{lemma}

\begin{proof}
Denote by $n$ the common rank of $E_1$ and $E_2$. Fix some bases $e=(e_1, \ldots, e_n)$ and $f=(f_1, \ldots, f_{n})$ of $E_1$ and $E_2$  such that $\iota(e_a) = \left\{ \begin{array}{ll} f_a & \mbox{, if } a<n; \\ 2f_n & \mbox{, if }a=n.\end{array}\right.$  and $o(\sum_a \alpha_a f_a) = \alpha_n \pmod{2}$.  Since $N$ is prime to $2$, we can find integers $2'$ and $k$ such that $2\cdot 2' + N\cdot k = 1$. Define two embeddings $i: \mathcal{T}_{\mathbb{E}_1}\hookrightarrow \mathcal{T}_{\mathbb{E}_2}$ and $j: \mathcal{T}_{\mathbb{E}_2}\hookrightarrow \mathcal{T}_{\mathbb{E}_1}$ by the formulas:
\begin{eqnarray*}
i(Z_{e_a}^{\pm 1}) := \left\{ 
\begin{array}{ll} Z_{f_a}^{\pm 1} & \mbox{, if }a< n; \\ Z_{f_n}^{\pm 2} & \mbox{, if }a=n. \end{array}\right.
& j(Z_{f_a}^{\pm 1}) := \left\{ 
\begin{array}{ll} Z_{e_a}^{\pm 1} & \mbox{, if }a< n; \\ Z_{e_n}^{\pm 2'} & \mbox{, if }a=n. \end{array}\right.
\end{eqnarray*}
 Consider the two morphisms of algebras $\phi_1 : \mathcal{T}_{\mathbb{E}_1}\xrightarrow{\cong}\mathcal{T}_{\mathbb{E}_1}$ and $\phi_2 : \mathcal{T}_{\mathbb{E}_2}\xrightarrow{\cong}\mathcal{T}_{\mathbb{E}_2}$ defined by $\phi_1(Z_{e_a}) = \left\{ \begin{array}{ll} Z_{e_a} & \mbox{, if }a< n; \\ Z_{e_n}^{1-kN} & \mbox{, if }a=n. \end{array}\right. $ and $\phi_2(Z_{f_a}) = \left\{ \begin{array}{ll} Z_{f_a} & \mbox{, if }a< n; \\ Z_{f_n}^{1-kN} & \mbox{, if }a=n. \end{array}\right. $. We have the identities $i\circ j = \phi_2$ and $j\circ i = \phi_1$. The induced functors $i^* : \mathcal{T}_{\mathbb{E}_2}- \Mod \rightarrow \mathcal{T}_{\mathbb{E}_1}-\Mod$ and $j^* : \mathcal{T}_{\mathbb{E}_1}- \Mod \rightarrow \mathcal{T}_{\mathbb{E}_2}-\Mod$ satisfy $i^*\circ j^* = \phi_1^*$ and $j^*\circ i^* = \phi_2^*$ and we need to prove that $\phi_1^*$ and $\phi_2^*$ are isomorphic to the identity functors. Let us define an inverse functor $\psi_1: \mathcal{T}_{\mathbb{E}_1}-\Mod \rightarrow  \mathcal{T}_{\mathbb{E}_1}-\Mod $ for $\phi_1^*$. As any quantum torus, the algebra $\mathcal{T}_{\mathbb{E}_1}$ is semi-simple and for an irreducible representation  $\rho : \mathcal{T}_{\mathbb{E}_1} \rightarrow \End(V)$, the central element $Z_{e_n}^{N}$ is sent to a scalar $\chi_{\rho}(Z_{e_n}^N)\in \mathbb{C}^*$. For an irreducible representation $(\rho,V)$, we define $\psi_1(\rho) :  \mathcal{T}_{\mathbb{E}_1} \rightarrow \End(V)$ by $\psi_1(\rho) (Z_{e_a}^{\pm 1}):=  \left\{ 
\begin{array}{ll} \rho(Z_{e_a})^{\pm 1} & \mbox{, if }a< n; \\ \chi_{\rho}(Z_{e_n}^N)^{\pm \frac{k}{2\cdot 2'}} \rho(Z_{e_n})^{\pm 1} & \mbox{, if }a=n. \end{array}\right.$
 Here $ \chi_{\rho}(Z_{e_n}^N)^{ \frac{k}{2\cdot 2'}}$ is the $2\cdot 2'$-th root of $\chi_{\rho}(Z_{e_n}^N)^k$ with argument in $[0,\frac{\pi}{2'})$. We define $\psi_1(\rho)$ for an arbitrary finite dimensional representation of $\mathcal{T}_{\mathbb{E}_1}$ by imposing $\psi_1(\rho_1\oplus \rho_2) = \psi_1(\rho_1)\oplus \psi_1(\rho_2)$ and obtain an endofunctor $\psi_1: \mathcal{T}_{\mathbb{E}_1}-\Mod \rightarrow  \mathcal{T}_{\mathbb{E}_1}-\Mod $. A straightforward computation shows that $\psi_1$ and $\phi_1^*$ are inverse to each other so $i^*\circ j^* = \phi_1^*$ is isomorphic to the identity functor.
 
 We prove that $j^* \circ i^*= \phi_2^*\cong \id$  similarly. 
   Hence $i^*$ and $j^*$ are equivalence of categories and the algebras $\mathcal{T}_{\mathbb{E}_1}$ and $\mathcal{T}_{\mathbb{E}_2}$ are Morita equivalent. 

\end{proof}

 \begin{proof}[Proof of Proposition \ref{propgluinginvolutive}]
 
 		Let   $\veebar_{\textbf{b}}$ denote the underlying surface of  $\veebar_{\textbf{b}}(\Sigma_1\sqcup \Sigma_2)$ with punctures and branched points removed.  
 		
		By definition of $\veebar_{\textbf{b}}(\Sigma_1\sqcup \Sigma_2)$, for $i=1,2$, there is an embedding $\iota_i\co \Sigma_i \setminus (\mathcal{P}_i \cup B_i \cup \mathring{\phi_i(\D)}) \hookrightarrow \veebar_{\textbf{b}}$. Their images intersect as a curve $\mathcal{C}\subset \veebar_{\textbf{b}}$.  Let us slightly enlarge  the image of $\iota_i$  into a surface $\Sigma_i'$ so that  $\Sigma_1'\cap \Sigma_2'$ is an open tubular neighborhood of $\mathcal{C}$. Consider the inclusion maps in $\mathrm{Top}^2$:
$$  (\mathcal{C}, \emptyset) \xrightarrow{(i,j)} (\Sigma_1', \partial \Sigma_1') \bigsqcup ( \Sigma_2',  \partial \Sigma_2') \xrightarrow{(k,l)} (\veebar_{\textbf{b}}, \partial \veebar_{\textbf{b}}).$$
		
The associated Mayer-Vietoris long exact sequence writes
\begin{equation*}
 \begin{tikzcd}
\mathrm{H}_1(\mathcal{C}) \arrow[r, "a"] & \mathrm{H}_1(\Sigma_1', \partial \Sigma_1')\oplus \mathrm{H}_1(\Sigma_2', \partial \Sigma_2') \arrow[r, "\overline{\iota}"] & \mathrm{H}_1(\veebar_{\textbf{b}}, \partial \veebar_{\textbf{b}}) \arrow[r, "b"] & \mathrm{H}_0(\mathcal{C}),
\end{tikzcd}
\end{equation*}
where the homology is taken with coefficients in $\mathbb{Z}$. More precisely:
\begin{enumerate}
\item By definition $a:= (i_*, j_*)$ and $\overline{\iota}:= k_* - l_*$, in particular they are both 
$\sigma$--equivariant. Since  $\sigma_*([\mathcal{C}])= [\mathcal{C}]$ we have $\mathrm{H}_1^{\sigma}(\mathcal{C})=0$, so the morphism $\overline{\iota}$ induces an injective morphism 
\begin{equation*}
\iota: \mathrm{H}_1^{\sigma}(\Sigma_1', \partial \Sigma_1')\oplus \mathrm{H}_1^{\sigma}(\Sigma_2', \partial \Sigma_2')\hookrightarrow \mathrm{H}_1^{\sigma}(\veebar_{\textbf{b}}, \partial \veebar_{\textbf{b}}). 
\end{equation*}
\item If $[x] \in  \mathrm{H}_1(\veebar_{\textbf{b}}, \partial \veebar_{\textbf{b}})$, we can decompose it as $x=x_1+x_2$ with $x_i \in \mathrm{C}_1(\Sigma_i')$ and we can decompose $\partial x = \partial_1 +\partial_2$ with $\partial_i \in \mathrm{C}_1(\partial \Sigma_i')$. By definition $b([x]):= [\partial x_1 -\partial_1]=[-\partial x_2 + \partial_2]$. So if $p\in \mathcal{C}$ is a point so that $\mathrm{H}_0(\mathcal{C})=\mathbb{Z}[p]$, then $b([x])=i([x], [\mathcal{C}])[p]$ where 
$$i: \mathrm{H}_1(\veebar_{\textbf{b}}, \partial \veebar_{\textbf{b}}) \times \mathrm{H}_1(\veebar_{\textbf{b}}) \to \mathbb{Z}$$ is the algebraic intersection pairing as defined for instance in \cite[Section IV.11]{Bredon_TopoGeom}. Note that the intersection pairing is preserved by the action of an oriented mapping class, in particular one has $i(\sigma_*([x]), \sigma_*([\mathcal{C}]))= i([x], [\mathcal{C}])$.  Therefore, if $\sigma_*([x])=-[x]$, then $b([x])=i([x], [\mathcal{C}])[p]= i(\sigma_*([x]), \sigma_*([\mathcal{C}]))[p]=-i([x], [\mathcal{C}])[p]=-b([x])[p]$ so $b([x])=0$ and  $\mathrm{H}_1^{\sigma}(\veebar_{\textbf{b}}, \partial \veebar_{\textbf{b}}) $ is included in $Ker(b)$.
\end{enumerate}

Let us investigate the lack for $\iota$ for being surjective; we show that there is an exact sequence 
\begin{equation*}
\begin{tikzcd}
0 \arrow[r, ""] &  \mathrm{H}_1^{\sigma}(\Sigma_1', \partial \Sigma_1')\oplus \mathrm{H}_1^{\sigma}(\Sigma_2', \partial \Sigma_2') \arrow[r, "\iota"] & \mathrm{H}_1^{\sigma}(\veebar_{\textbf{b}}, \partial \veebar_{\textbf{b}}) \arrow[r, "o"] & 
\mathbb{Z}/2\mathbb{Z}.  
\end{tikzcd}
\end{equation*}  
First since $\mathrm{H}_1^{\sigma}(\veebar_{\textbf{b}}, \partial \veebar_{\textbf{b}}) $ is included in $Ker(b)=Im(\overline{\iota})$, then  for each $y\in \mathrm{H}_1^{\sigma}(\veebar_{\textbf{b}}, \partial \veebar_{\textbf{b}})$ there is an $x\in  \mathrm{H}_1(\Sigma_1', \partial \Sigma_1')\oplus \mathrm{H}_1(\Sigma_2', \partial \Sigma_2')$ such that $\overline{\iota}(x)=y$. 
Since $y+\sigma_*(y)=0$, one has $x+\sigma_*(x)\in Ker(\overline{\iota})=Im(a)$. In other words, there is an integer $n$ such that $a(n[\mathcal{C}])=x+\sigma_*(x)$. 
We define $o(y)$ to be $n$ modulo $2$. 
\\
It is clear that $Im(\iota)$ is included in $Ker(o)$.  For the converse, let $y\in \mathrm{H}_1^{\sigma}(\veebar_{\textbf{b}}, \partial \veebar_{\textbf{b}})$ such that $o(y)=0 \pmod{2}$. 
There exist $x\in  \mathrm{H}_1^{\sigma}(\Sigma_1', \partial \Sigma_1')\oplus \mathrm{H}_1^{\sigma}(\Sigma_2', \partial \Sigma_2')$ and an integer $n'$ such that $\overline{\iota}(x)= y$ and $x+\sigma_*(x)=2n' a([\mathcal{C}])$. Define $z:=x-a(n'[\mathcal{C}])$. 
One has $\sigma_*(z)= \sigma_*(x)-a(n'[\mathcal{C}]) = -x+a(n'[\mathcal{C}])= -z$, so $z\in  \mathrm{H}_1^{\sigma}(\Sigma_1', \partial \Sigma_1')\oplus \mathrm{H}_1^{\sigma}(\Sigma_2', \partial \Sigma_2')$.  
On the other hand, $\iota(z)= \iota(x)=y$ which proves that $y\in \mathrm{Im}(\iota)$.
\\

The first part of the proposition follows from Lemma \ref{lemma_preliminary}. 
For the second part, it is enough to prove that, if one of the two surfaces $\Sigma_1$ or $\Sigma_2$ has no boundary, then $o$ is null. 
Suppose $\Sigma_1$ has no boundary. 
We proceed by contradiction: suppose that there exists $[\alpha] \in \mathrm{H}_1^{\sigma}(\veebar_{\textbf{b}}, \partial \veebar_{\textbf{b}})$ such that $o([\alpha])=1\pmod{2}$. 
In this case, there exist $[\gamma]$ and $[\beta]$ such that  $[\gamma]+[\beta] \in  \mathrm{H}_1(\Sigma_1')\oplus \mathrm{H}_1(\Sigma_2', \partial \Sigma_2')$ satisfies $\overline{\iota}([\gamma]+[\beta])=[\alpha]$ and $\sigma_*([\gamma]+[\beta])+ [\gamma]+[\beta]= a([\mathcal{C}])$. One can suppose that $[\gamma]$ is the homology class of a simple closed curve $\gamma$. 
Let $[\mathcal{C}]_i$ be the $i$--th component of $a([\mathcal{C}])=[\mathcal{C}]_1+[\mathcal{C}]_2 \in  \mathrm{H}_1(\Sigma_1')\oplus \mathrm{H}_1(\Sigma_2', \partial \Sigma_2')$. 
Note that $[\gamma]+\sigma_*([\gamma])-[\mathcal{C}]_1$ is zero in $ \mathrm{H}_1(\Sigma_1')$.  
Therefore $\gamma \cup \sigma(\gamma)\cup \mathcal{C}$ bounds a closed embedded surface, say $S\subset \Sigma_1'$, that can be supposed to be stable under the involution $\sigma$; the latter restricts to $S$ as an involution without fixed point (because $\sigma$ has no fixed point in $\Sigma_1'$). 
Moreover, since $S$ has three boundary components, two of them are exchanged by $\sigma_{|S}$ and one is fixed. 
By \cite[Theorem 1.3]{Asoh}, such a free involutive surface does not exist since it should have an even number of boundary components preserved by the involution. This contradicts the assumption and concludes the proof. 
 	 
 \end{proof}

	Recall the decomposition of involutive surfaces into basic ones given in Lemma \ref{lemma_decomp_invol_surface}. In virtue of Proposition \ref{propgluinginvolutive}, one has a Morita equivalence between the equivariant $\mathbb{C}^*$ skein algebra of an involutive surface and the tensor product of the equivariant  $\mathbb{C}^*$ skein algebras of its basic surfaces. 
 Let us compute the equivariant  $\mathbb{C}^*$ skein algebras for these surfaces: 
\begin{enumerate}
	\item $\eqSkpre(\mathbb{S})$ is isomorphic to  $\mathcal{R}[H_p^{\pm 1}]$, where $H_p$ is made of two simple closed curves encircling the punctures $p$ and $p'$ which are anti-invariant under the involution.  
	\item $\eqSkpre(\mathbb{E}_1)$ is isomorphic to  
	$\mathcal{W}_{q^2} = \mathcal{R}\left< Z_1^{\pm 1}, Z_2^{\pm 1} | Z_1 Z_2 = q^2 Z_2 Z_1 \right>$, 
	(recall $q=\omega^{-4}$) and  $Z_1$ and $Z_2$ are the classes of $\alpha_1\cup \sigma({\alpha_1})^{-1}$ and $\alpha_2\cup \sigma({\alpha_2})^{-1}$ respectively. Here $\alpha_1$ is the meridian $S^1\times \{0\}$ and $\alpha_2$ is the longitude $\{0\}\times S^1$; they are oriented such that the intersection $( \alpha_1, \alpha_2)$ is $-1$. 
	\item $\eqSkpre(\mathbb{E}_2)$  is isomorphic to $\mathcal{W}_{q} = \mathcal{R}\left< Z_1^{\pm 1}, Z_2^{\pm 1} | Z_1 Z_2 = q Z_2 Z_1 \right>$, where  $Z_1$ and $Z_2$ are the classes of $\alpha_1\cup \sigma({\alpha_1})^{-1}$ and $\alpha_2\cup \sigma({\alpha_2})^{-1}$ respectively drawn in Figure \ref{figelementarycob}.
	\item For an odd integer $n\geq 1$, $\eqSkpre(\mathbb{P}_n)$ is    isomorphic to 
	$$ \mathcal{Y}^{(n)}_q = \mathcal{R}\left< Z_i^{\pm 1}, i \in \mathbb{Z}/n\mathbb{Z} | Z_i Z_{i+1} =q Z_{i+1} Z_i, Z_iZ_j=Z_jZ_i, j\neq i-1, i, i+1 \right>,$$ where for $1\leq k \leq n$, we set $Z_k=[\alpha_k \cup \sigma({\alpha_k})^{-1}]$ in which  $\alpha_k$ denotes  the arc encircling $p_k$ oriented in the clockwise direction, as depicted in Figure \ref{figelementarycob}. 
	\item For an even integer $n\geq 2$, $\eqSkpre(\mathbb{P}_n)$ is isomorphic to  the $\mathcal{R}$-module generated by $Z_b^{\pm 1}, Z_i^{\pm 1}, i \in \mathbb{Z}/n\mathbb{Z}$ with relations:
	\begin{eqnarray*}
		Z_i Z_{i+1}=q Z_{i+1}Z_i & \mbox{ for }i \in \mathbb{Z}/n\mathbb{Z}
		\\ Z_b Z_1= q^2 Z_1 Z_b, ~~~ Z_b Z_n = q^{-2}Z_n Z_1 &
		\\ Z_b Z_i=Z_i Z_b &\mbox{ for }i\neq 1,n 
		\\ Z_i Z_j =Z_j Z_i & \mbox{ for } j \neq i-1, i, i+1.
	\end{eqnarray*}
	Here $Z_k=[\alpha_k \cup \sigma({\alpha_k})^{-1}]$ and $Z_b = [\beta \cup \sigma(\beta)^{-1}]$.
	We denote by $ \mathcal{Y}^{(n)}_q$ this algebra.
\end{enumerate}

    \begin{proposition}\label{propdecomposition}
  If $\mathbf{\Sigma}=(\Sigma, \mathcal{P}, \sigma)$ is a connected  involutive punctured surface with combinatorial data $(g, n_{\partial}, \{s_i\}_{i\in I},n_b, \mathring{s})$, then the algebra $\mathcal{S}_{\omega}^{\mathbb{C}^*, \sigma}(\mathbf{\Sigma})$ is Morita equivalent to the algebra 
  \begin{equation*}
  	\mathcal{A}:= \mathcal{R}[H_p^{\pm 1}]^{\otimes \mathring{s}} \otimes \mathcal{W}_{q^2}^{\ot n_1} 
  	\otimes \mathcal{W}_q^{\ot n_2} \otimes  \otimes_{i\in I} \mathcal{Y}^{(s_i)}_q, 
  	\end{equation*}
  	where $n_1=(n_b+n_{\partial}^{odd}-2)/2$ and $n_2=(2g-n_b-n_{\partial}^{odd}+2)/4$. 
  
  \par Moreover if $\Sigma$ is either closed, or has one boundary component, or has two boundary components which are exchanged by the involution $\sigma$, then the algebras $\mathcal{S}_{\omega}^{\mathbb{C}^*, \sigma}(\mathbf{\Sigma})$ and  $\mathcal{A}$ are isomorphic.
  \end{proposition}
  
  \begin{proof} This is an immediate consequence of Lemma \ref{lemma_decomp_invol_surface} and Proposition \ref{propgluinginvolutive}.
 \end{proof}
 
 \subsection{Irreducible representations of elementary algebras}
 
 In this subsection, we suppose that $\omega$ is a root of unity of odd order $N>1$ and classify the irreducible representations of the equivariant abelian skein algebras associated to elementary involutive surfaces.
 
 \begin{remark}
  De Concini and Procesi proved in \cite[Proposition $7.2$]{DeConciniProcesiBook} that any quantum torus at root of unity (not necessarily odd) is Azumaya of constant rank. Their theorem thus implies that quantum tori are semi-simple, that their simple modules are in $1$-to-$1$ correspondance, modulo isomorphism, with their induced character over their center and that they all have the same dimension $d$ such that $d^2$ is the rank of the quantum tori over its center. Let us emphasize why the De Concini-Procesi theorem is insufficient to prove Theorem \ref{theorem2}. Consider a quantum torus $\mathcal{T}_{\mathbb{E}}$ associated to a pair $\mathbb{E}=(E,(\cdot, \cdot)_E)$, where $E$ is a free finitely generated abelian group and $(\cdot,\cdot)_E : E\times E \rightarrow \mathbb{Z}$ a skew-symmetric form as in Section $2.5$. Let $E_0\subset E$ be the kernel of the bilinear form $E\times E \rightarrow \mathbb{Z}/N\mathbb{Z}$ obtained by reduction modulo $N$ of $(\cdot, \cdot)_E$. The center of $\mathcal{T}_{\mathbb{E}}$ is easily seen to be spanned by the elements $Z_{e_0}$ for $e_0\in E_0$, hence its simple modules have dimension $\sqrt{\lvert E/E_0 \rvert}$. By definition, the balanced Chekhov-Fock algebra is the quantum torus associated to the abelian group $K_{\Delta}$ of balanced monomials equipped with the Weil-Petersson form $\left< \cdot, \cdot\right>^{WP}$ (see below). Our strategy to compute the dimension  $\sqrt{\lvert E/E_0 \rvert}$ for this pair, which is similar to the approach of Bonahon, Liu and Wong in \cite{BonahonLiu,BonahonWong2} in the closed case, is to identify the pair $(K_{\Delta}, \left< \cdot, \cdot\right>^{WP})$ with the pair $(\mathrm{H}_1^{\sigma}(\widehat{\Sigma}_{\widehat{\mathcal{P}}\cup \widehat{B}}, \partial \widehat{\Sigma}_{\widehat{\mathcal{P}}\cup \widehat{B}}; \mathbb{Z}), \left(\cdot, \cdot\right))$ (Theorem \ref{theorem1}) and to use the sewing operation and Proposition \ref{propgluinginvolutive} to decompose this pair in direct summands, up to an extension by $(\mathbb{Z}/2\mathbb{Z})^n$ which does not change $\sqrt{\lvert E/E_0 \rvert}$ as long as $N$ is odd. This subsection is the only moment of the paper where the De Conicini-Procesi theorem could simplify our study by classifying the simple modules of the elementary quantum tori in the tensor decomposition. However, the proof of  Lemma \ref{lemmairrepelementary} below, is quite elementary (in comparison to the Artin-Procesi theorem on which \cite[Proposition $7.2$]{DeConciniProcesiBook} is based) and has the advantage of describing the simple modules explicitly. 
 \end{remark}

  It is well-known that the algebras $\mathcal{W}_q$ and $\mathcal{W}_{q^2}$ are semi-simple, that their simple modules have dimension $N$ and that the set of isomorphism classes of simple modules is in bijection with $(\mathbb{C}^*)^{2}$ by the map sending a simple module to the scalars associated  to the central elements  $Z_1^N$ and $Z_2^N$ (see \textit{e.g.} \cite[Lemma $17$]{BonahonLiu} for an explicit description of their simple modules). It remains to study the simple modules of the algebras $\mathcal{Y}^{(n)}_q$. Define $d(n):=
 \left\{ \begin{array}{ll} \frac{n-1}{2} & \mbox{, if } n \mbox{ is odd ;}
\\ \frac{n}{2} & \mbox{, if }n\mbox{ is even.}
\end{array} \right.$
  
  \begin{lemma}\label{lemmairrepelementary}
  Let $n\geq 1$ be an integer. The algebra $\mathcal{Y}^{(n)}_q$ is semi-simple. Its simple modules have dimension $N^{d(n)}$ and the set of isomorphism classes of simple modules is in bijection with the the set of  characters  on the center of the algebra. Moreover this center is generated by the elements $Z_i^N$ for $i\in \mathbb{Z}/n\mathbb{Z}$, by the element $H_{\partial}:= Z_1\ldots Z_n$ and their inverses.
  \end{lemma}
  \begin{proof} First note that $\mathcal{Y}^{(1)}_q\cong \mathbb{C}[H_{\partial}^{\pm 1}]$ so the result is immediate if $n=1$. Note also that  $\mathcal{W}_{q^2}\otimes \mathbb{C}[H_{\partial}^{\pm 1}] \cong \mathcal{Y}^{(2)}_q$  through the isomorphism sending the generators $Z_1$ and $Z_2$ of $\mathcal{W}_{q^2}$ to the elements $Z_b$ and $Z_1$ of $ \mathcal{Y}^{(2)}_q$ respectively; so the case $n=2$ follows from the preceding remark (so from \textit{e.g.} \cite[Lemma $17$]{BonahonLiu}). It follows from the definition of $\mathcal{Y}^{(n)}_q$ that the elements $Z_i^N$ and $H_{\partial}$ are central.
  \vspace{2mm}
  \par
  Next suppose $n\geq 3$ is odd. Consider the maximal torus $\mathcal{T}\subset \mathcal{Y}^{(n)}_q$ generated by the pairwise commuting elements $Z_1^{\pm1}, Z_3^{\pm 1}, \ldots , Z_{n-2}^{\pm 1}$ with odd indices. Let $V$ be a module of $ \mathcal{Y}_q^{(n)}$. Choose $v\in V$ a common eigenvector of the elements of $\mathcal{T}$ and of the central elements $Z_i^N$ and $H_{\partial}$ such that $Z_{2k+1}v=\lambda_{2k+1}v$ for $1\leq k \leq (n-1)/2$ and $Z_i^Nv={z}_i v$ and $H_{\partial}v=h_{\partial}v$. Define $w_{(i_0, i_2, \ldots, i_{n-3})}:= Z_0^{i_0}Z_2^{i_2}\ldots Z_{n-3}^{i_{n-3}}v$ and let $W\subset V$ be the subspace spanned by the $w_i$ for $i=(i_0, i_2, \ldots, i_{n-3})\in \{ 0, \ldots, N-1\}^{(n-1)/2}$. The defining relations of $\mathcal{Y}^{(n)}_q$ imply that: 
  \begin{align*}
  Z_{2k} w_i = w_{(i_0, i_2, \ldots, i_{k}+1, \ldots,  i_{n-3})} &\mbox{, if }0\leq k \leq (n-3)/2\mbox{ and }i_k<N-1, 
  \\  Z_{2k} w_i = z_{2k} w_{(i_0, i_2, \ldots, 0, \ldots,  i_{n-3})} &\mbox{, if }0\leq k \leq (n-3)/2\mbox{ and }i_k=N-1, 
  \\ Z_{2k+1}w_i = q^{i_{2k+2}-i_{2k}}\lambda_{2k+1} w_i &\mbox{, if }0\leq k \leq (n-5)/2,
  \\ Z_{n-2}w_i= q^{i_{n-3}}\lambda_{n-2} w_i.
  \end{align*}
  We deduce from these formulas that 
  \begin{enumerate}
  \item $W$ is $\mathcal{Y}^{(n)}_q$-stable, \item the vector $v$ is cyclic in $W$,  \item for each $i\in \{ 0, \ldots, N-1\}^{(n-1)/2}$, there is a character $\chi_i : \mathcal{T} \rightarrow \mathbb{C}^*$ such that $t\cdot w_i = \chi_i (t) w_i$ for all $t \in \mathcal{T}$ and \item the $\chi_i$ are pairwise distinct and their set $\{ \chi_i \}_i$ is exactly the set of characters $\chi: \mathcal{T}\to \mathbb{C}^*$ such that $\chi(Z_{2k+1})^N = z_{2k+1}$.
  
   \end{enumerate}
  
  So $\mathbb{C} \cdot w_i = \{ w \in W | t\cdot w = \chi_i(t) w, \mbox{ for all }t\in \mathcal{T} \}$.
  It follows that the set $\{w_i\}_i$ is free, so forms a basis of $W$ and $W$ has dimension $N^{d(n)}$. 
  Moreover, if $\Theta\in \End(W)$ is an operator commuting with the image of $\rho_{\big| W}$ then $\Theta$ preserves each axis $\mathbb{C}w_i$. In particular $\Theta v =\lambda v$, for some scalar $\lambda$,  thus $\Theta=\lambda \id$  by cyclicity. This proves that $W$ is simple.

   If $V\neq W$, since $\mathcal{T}$ is abelian so semi-simple, we can find a common eigenvector of the elements of $\mathcal{T}$ which is not in $W$ and repeat the construction to obtain a second simple module $W'$ such that $W\oplus W'\subset V$. By Zorn's lemma, there exists a maximal submodule $W^{max}\subset V$ which is direct sum of simple modules. By contradiction, if $W^{max}$ is proper, then we can find a common eigenvector of the elements of $\mathcal{T}$ in $V\setminus W^{max}$ and thus construct a simple module $W$ such that $W\oplus W^{max}\subset V$. This would contradict the maximality of $W^{max}$, hence $\mathcal{Y}^{(n)}_q$ is semi-simple.
\vspace{2mm}
 \par   If $W$ and $W'$ are two simple modules such that $\rho(Z_i^N)=\rho'(Z_i^N)=z_i \id$ and $\rho(H_{\partial})=\rho'(H_{\partial})= h_{\partial}\id$ then the eigenvalues $\lambda_{2k+1}$ and $\lambda'_{2k+1}$ of the elements of the maximal torus $\mathcal{T}$ differ by a $N$-{th} root of unity $\omega^{n_{2k+1}}$. Moreover for each index $i$ there exists an index $i'$ such that $w_i\in W$ and $w'_{i'}\in W'$ are associated to the same character $\chi_i$ of $\mathcal{T}$.
  
  Using the above relations we see that the vector space isomorphism between $W$ and $W'$ sending $w_i$ to $(\prod_k \omega^{n_{2k+1}i_k})w'_{i'}$ is equivariant for the action of $\mathcal{Y}^{(n)}_q$, thus the representation depends only, up to isomorphism, on its central character evaluated on the elements $Z_i^{\pm N}$ and $H_{\partial}^{\pm 1}$. 
 Since the above explicit construction of $W$ works with any parameters $z_i$ and $h_{\partial}$, every such set of  parameters induces a simple module.

 It remains to show that the center is generated by the elements $Z_i^{\pm N}$ and $H_{\partial}^{\pm 1}$. A Laurent monomial is an element of the form $M= Z_1^{k_1} \ldots Z_n^{k_n}$. Note that a linear combination $\sum_i \alpha_i M_i$, with $\alpha_i \in \mathbb{C}^*$, is central if and only if each Laurent monomial $M_i$ is central. Consider a central Laurent monomial $M=Z_1^{k_1} \ldots Z_n^{k_n}$ and let us prove that $M$ is a product of elements $Z_i^{\pm N}$ and $H_{\partial}^{\pm 1}$. Multiplying $M$ by some elements $Z_i^{N m_i}$, we can suppose that $0 \leq k_i \leq N-1$ for $0 \leq i \leq n-1$. For $i \in \mathbb{Z}/n \mathbb{Z}$, the fact that $M$ commutes with $Z_i$ implies that $k_{i-1} \equiv k_{i+1} \pmod{N}$, so $k_{i-1}=k_{i+1}$. Since $n$ is odd, this implies that all $k_i$ are equal, thus $M$ is a power of $H_{\partial}$.
 Thus we have proved the lemma when $n$ is odd.
  
  \vspace{2mm}
  \par The proof when $n$ is even is quite similar. Consider the maximal torus $\mathcal{T}$ generated by the pairwise commuting elements $Z_{2k+1}^{\pm 1}$ for $0\leq k \leq (n-2)/2$. Consider a  $\mathcal{Y}_q^{(n)}$-module $V$ and choose $v\in V$ a common eigenvector of the element of $\mathcal{T}$ and of the central elements $Z_i^N$ and $H_{\partial}$ such that $Z_{2k+1}v=\lambda_{2k+1}v$ for $1\leq k \leq (n-2)/2$ and $Z_i^Nv={z}_i v$ for $i\in \{0, \ldots, n\}$ and $H_{\partial}v=h_{\partial}v$. Define $w_{( i_b, i_2,, i_4,  \ldots, i_{n-2})}:= Z_b^{i_b}Z_2^{i_2}\ldots Z_{n-2}^{i_{n-2}}v$ and $W\subset V$ be the subspace spanned by the $w_i$ for $i=(i_b, i_2, \ldots, i_{n-2})\in \{ 0, \ldots, N-1\}^{n/2}$. The defining relations of $\mathcal{Y}^{(n)}_q$ imply that: 
    \begin{align*}
  Z_{2k} w_i = w_{(i_b, i_2, \ldots, i_{k}+1, \ldots,  i_{n-2})} &\mbox{, if }1\leq k \leq (n-2)/2\mbox{ and }i_k<N-1, 
  \\ Z_b w_i = w_{(i_b+1, i_2, \ldots, i_{k}, \ldots,  i_{n-2})} &\mbox{, if }i_b<N-1, 
  \\  Z_{2k} w_i = z_{2k} w_{(i_b, i_2, \ldots, 0, \ldots,  i_{n-2})} &\mbox{, if }1\leq k \leq (n-2)/2\mbox{ and }i_k=N-1, 
  \\ Z_b w_i = z_b w_{(0, i_2, \ldots, i_{n-2})}&\mbox{, if }i_b=N-1,  
  \\ Z_{2k+1}w_i = q^{i_{2k+2}-i_{2k}}\lambda_{2k+1} w_i &\mbox{, if }1\leq k \leq (n-4)/2,
  \\ Z_1w_i= q^{2(i_2-i_b)} \lambda_1 w_i,&
\\   Z_{n-1}w_i=q^{-i_{n-2}}\lambda_{n-1} w_i.&
  \end{align*}
  We conclude in the same manner than in the previous case.
  
  \end{proof}

 \subsection{The balanced Chekhov-Fock algebra}

 	Fix $\mathcal{R}$  a unital commutative ring and $\omega\in \mathcal{R}^{\times}$ an invertible element. 
 	Throughout this section,  $(\mathbf{\Sigma}, \Delta)$ is a triangulated punctured surface.

 	Let us recall the definition of the balanced Chekhov-Fock algebra introduced in \cite[Section 2.1 and 3.1]{BonahonWong2} to which we refer for more details.  
 	\\

 	Denote by  $\mathcal{E}(\Delta)=\{e_1,...,e_n\}$ the set of edges of $\Delta$ and $\mathring{\mathcal{E}}$ the subset of the inner edges.
 	Given two edges $e$ and $e'$, we denote by $a_{e,e'}$ the number of faces $\mathbb{T}$ of $\Delta$ for which $e$ and $e'$ are edges of $\mathbb{T}$ such that we pass from $e$ to $e'$ in the counter-clockwise direction in $\mathbb{T}$. The \textit{Weil-Petersson} skew-symmetric form $\left<\cdot, \cdot \right>: \mathcal{E}(\Delta)\times \mathcal{E}(\Delta)\rightarrow \mathbb{Z}$ is defined by $\left< e, e'\right>:= a_{e,e'}-a_{e',e}$.
 	The quantum torus $\mathcal{T}_{\omega}(\mathbf{\Sigma}, \Delta)$ is the 
 	non-commutative unital Laurent polynomial ring $\mathcal{R}\left\{Z_e^{\pm 1}, e\in \mathcal{E}(\Delta)\right\}$ modded out by the following relation: 
 	\begin{equation}
 	Z_e Z_{e'} = \omega^{2\left<e, e'\right>}Z_{e'}Z_e \mbox{ for } e,e'\in \mathcal{E}(\Delta).
 	\end{equation}   
 	A  convenient $\mathcal{R}$-basis is given by the 
 	\textit{Weyl ordered} monomials \ie the monomials 
 	\begin{equation*}
 	[Z_{e_1}^{k_1}\ldots Z_{e_n}^{k_n}]:= \omega^{-\sum_{i<j} \left<e_i, e_j\right>} Z_{e_1}^{k_1}\ldots Z_{e_n}^{k_n}, 
 	\end{equation*}
 	by letting the $k_i$'s running through $\Z$.

 	A monomial   $Z_{e_1}^{k_1}\ldots Z_{e_n}^{k_n}$ is called  \textit{balanced} if for each  	triangle of $\Delta$ with edges $e_a, e_b, e_c$ the sum $k_a+k_b+k_c$ is even. Here, when the triangle is self-folded, for instance if $e_b=e_c$ we put $k_b=k_c$ so the condition is just $k_a$ is even. 
 	
 \begin{remark}
 	Geometrically, this is means that there exists a collection of closed curves and arcs (with boundary in $\partial \Sigma_{\mathcal{P}}$)  in $ \Sigma_{\mathcal{P}}$ such that, for each edge $e$, the  intersection of all these curves and arcs with $e$ has the same parity as $k_e$. 
 \end{remark}	
 
 	\begin{definition}
 		The \textit{balanced Chekhov-Fock algebra} $\CF$ is the sub-algebra of $\mathcal{T}_{\omega}(\mathbf{\Sigma}, \Delta)$ generated by the balanced monomials. 
 	\end{definition}
	
 	Let us define a morphism 
 	\begin{equation}\label{eq: decompo morph i for CF}
 	i^{\Delta} : \mathcal{Z}_{\omega}(\mathbf{\Sigma},\Delta) \hookrightarrow \otimes_{\mathbb{T} \in F(\Delta)} \mathcal{Z}_{\omega}(\mathbb{T}). 
 	\end{equation}
 	For an inner edge $e$ of $\Delta$, let us denote by $e'$ and $e''$  its two lifts in $\sqcup_{\T\in F(\Delta)} \T$. 
 	For each balanced monomial $[Z_{e_1}^{k_1} \cdots Z_{e_n}^{k_n}]$, let 
 	$$i^{\Delta}([Z_{e_1}^{k_1} \cdots Z_{e_n}^{k_n}]) := \left[\prod_{e_i\in \mathring{\mathcal{E}}} Z_{e_i'}^{k_i}Z_{e''_i}^{k_i} \cdot \prod_{e_j\in \mathcal{E}\setminus \mathring{\mathcal{E}}} Z_{e_j}^{k_j}\right].$$
 	We extend it by linearity to $\mathcal{Z}_{\omega}(\mathbf{\Sigma},\Delta)$.
 	 \begin{lemma} $i^{\Delta}$ is an injective morphism of algebras. 
 	\end{lemma}
 	\begin{proof}
 		It is a straightforward consequence of the definitions and of the fact that $i^{\Delta}$ sends the basis of balanced monomials of $\mathcal{Z}_{\omega}(\mathbf{\Sigma},\Delta)$ to a subset of the basis of balanced monomials of $\mathcal{Z}_{\omega} (\bigsqcup \mathbb{T}) = \otimes_{\mathbb{T}}\mathcal{Z}_{\omega}(\mathbb{T})$.
 	\end{proof}
 	
 	Recall from Definition \ref{de: cover} the $2$--fold branched covering $\pi: \hat{\Sigma}(\Delta) \rightarrow \Sigma$ with involution $\sigma$ and branched points ${B}\subset \Sigma$. Let $\hat{\mathbf{\Sigma}}=(\hat{\Sigma}(\Delta), \hat{\mathcal{P}}\sqcup \hat{B})$ be the resulting punctured surface with involution $\sigma$.

For each leaf labelling $\ell$, let us define an isomorphism of $\mathcal{R}$--modules 	
\begin{equation}\label{eq:iso CF Sk}
\Phi_{\ell} \co \mathcal{Z}_{\omega}(\mathbf{\Sigma}, \Delta) \to \mathcal{S}_{\omega}^{\mathbb{C}^*, \sigma}(\mathbf{\hat{\Sigma}}). 
\end{equation}
  A very similar isomorphism was defined for closed punctured surfaces by Bonahon-Liu and Bonahon-Wong in \cite{BonahonLiu, BonahonWong2} and its extension to punctured surfaces with boundary is straightforward. 
  However, it will be useful to have an explicit description of this morphism, especially on triangles; we now spend some time in doing so.  
  
  The morphism  $\Phi_{\ell}$ is  the linear extension of  a sequence of three bijections 
  \begin{equation}\label{eq:bij Theta}
  \varphi: K_{\Delta}\cong \mathcal{W}(\tau_{\Delta}, \mathbb{Z})\cong \mathcal{W}^{\sigma}(\hat{\tau}_{\Delta}, \mathbb{Z})\cong \mathrm{H}_1^{\sigma}(\hat{\Sigma}_{\hat{\mathcal{P}}\cup\hat{B}}, \partial \hat{\Sigma}_{\hat{\mathcal{P}}\cup\hat{B}}; \mathbb{Z}),
  \end{equation}
  where the first set $K_{\Delta}$ is the basis of balanced monomials of $\mathcal{Z}_{\omega}(\mathbf{\Sigma}, \Delta)$ and the next two sets are defined as follows. 
  
  Let $\hat{\tau}_{\Delta} \subset \hat{\Sigma}_{\hat{\mathcal{P}}\cup \hat{B}}$ be the oriented train track such that, in each hexagon, 
  it looks as in Figure \ref{figtraintracks}; its orientation thus depends on the leaf labeling $\ell$. 
  The train track intersects each boundary edge at one point.  
  Note that $\hat{\tau}_{\Delta}$ is a deformation retract of $\hat{\Sigma}_{\hat{\mathcal{P}}\cup \hat{B}}$ relatively to its boundary.  
  The projection map $\pi: \hat{\Sigma}(\Delta) \rightarrow \Sigma$ projects $\hat{\tau}_{\Delta}$ on an (non-oriented) train track $\tau_{\Delta}\subset \Sigma_{\mathcal{P}}$ which, on each triangle, looks like in Figure \ref{figtraintracks}. 
 
  \begin{figure}[!h] 
\centerline{\includegraphics[width=10cm]{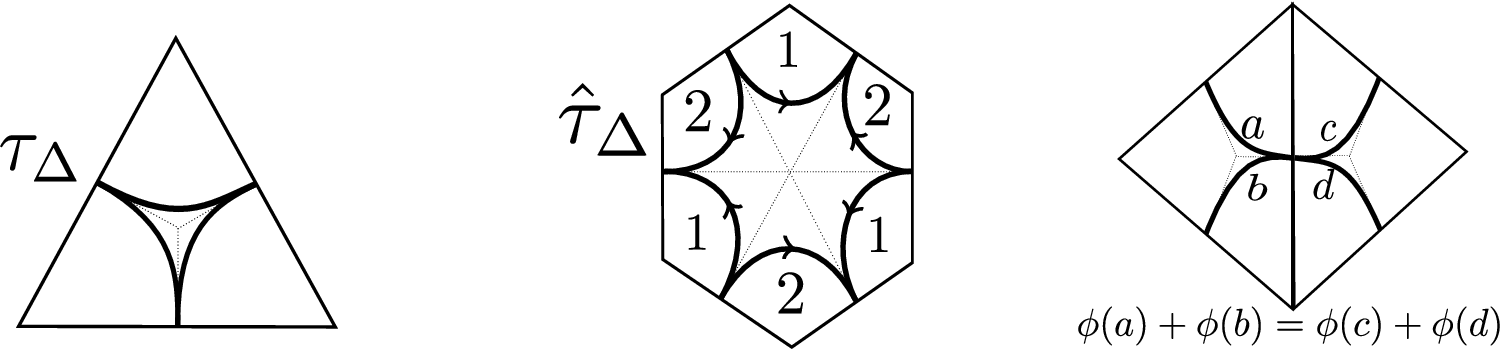} }
\caption{On the left: the train tracks associated to the triangle and its lift. We draw a leaf labeling and its corresponding orientation. On the right: the switch condition.} 
\label{figtraintracks} 
\end{figure}

 Let $\mathcal{W}(\tau_{\Delta}, \mathbb{Z})$ be the set of maps from the set of the edges of $\tau_{\Delta}$ to the set of integers that satisfy the \emph{switch-condition} illustrated in Figure \ref{figtraintracks}. 
 Likewise, let $\mathcal{W}^{\sigma}(\hat{\tau}_{\Delta}, \mathbb{Z})$ be the set of maps from the set of the edges of  $\hat{\tau}_{\Delta}$ to the set of integers that satisfy the switch condition and that are invariant under the covering involution.  
	
	\begin{enumerate}
		\item  The first bijection $K_{\Delta}\cong \mathcal{W}(\tau_{\Delta}, \mathbb{Z})$ 
   sends every monomial $[Z_{e_1}^{k_1}\ldots Z_{e_n}^{k_n}]$ to the following map $\phi$. 
   To an edge $\ec$ of $\tau_{\Delta}$ that connects two edges $e_a$ and $e_b$ of a triangle of $\Delta$ with third edge $e_c$, 
   one sets 
   \begin{equation*}
   \phi(\ec)=\frac{k_a+k_b-k_c}{2}. 
   \end{equation*}
   The balanced condition ensures that this is an integer. 
   \\
  The inverse map is defined by sending $\phi$ to the Weyl ordered balanced monomial with $k_e$ obtained by choosing an arbitrary face $\mathbb{T}$ containing $e$ and  setting $k_e=\phi(\ec)+\phi(\ec')$ where $\ec,\ec'$ are the two edges of the train track lying in $\mathbb{T}$ and intersecting $e$. The switch condition ensures that this integer does not depend on the choice of the triangle $\mathbb{T}$. 
   
   \item  The second bijection $\mathcal{W}(\tau_{\Delta}, \mathbb{Z})\cong \mathcal{W}^{\sigma}(\hat{\tau}_{\Delta}, \mathbb{Z})$ sends any map $\phi$ to the map $\hat{\phi}$  that sends the two lifts of an edge $\ec$ to $\phi(\ec)$. 
  	\item The third bijection   is as follows. 
    For $\hat{\phi} \in \mathcal{W}^{\sigma}(\hat{\tau}_{\Delta}, \mathbb{Z})$, one constructs a generator of $\mathrm{H}_1^{\sigma}(\hat{\Sigma}_{\hat{\mathcal{P}}\cup\hat{B}}, \partial \hat{\Sigma}_{\hat{\mathcal{P}}\cup\hat{B}}; \mathbb{Z})$ by taking, for each edge $\ec$ of $\hat{\tau}_{\Delta}$:
    \begin{itemize}
    	\item $\hat{\phi}(\ec)$ parallel copies of $\ec$ that are oriented as the train tracks, if $\hat{\phi}(\ec)\geq 0$; 
    	\item $-\hat{\phi}(\ec)$ parallel copies of $\ec$ with opposite orientation, if $\hat{\phi}(\ec)<0$.
    \end{itemize}
	Then one connects the resulting arcs following the train track in an arbitrary way. 
	The switch condition ensures that this is possible and the homological relation of Equation \eqref{eq:skeinconsequences} implies that the corresponding class in the equivariant $\C^*$ skein algebra does not depend on the way we connect them. 
	\\
	The inverse bijection is as follows. For each edge $\ec$ of $\hat{\tau}_{\Delta}$  that is inside a hexagon $\hat{\mathbb{T}}$, denote by $\ec^{\dagger}$ the only arc intersecting $\ec$ once transversally with end point the branched point of $\hat{\mathbb{T}}$ and a puncture of $\hat{\mathcal{P}}\cap \hat{\mathbb{T}}$ oriented from the puncture to the branched point if $c$ lies in a leaf labeled by $1$ and in the opposite direction if the leaf is labeled by $2$. 
	To a homology class $[\mathcal{C}]\in \mathrm{H}_1^{\sigma}(\hat{\Sigma}_{\hat{\mathcal{P}}\cup\hat{B}}, \partial \hat{\Sigma}_{\hat{\mathcal{P}}\cup\hat{B}}; \mathbb{Z})$ we associate the map $\hat{\phi}$ such that $\hat{\phi}(\ec)$ is the intersection number of the Borel-Moore homology class of $\ec^{\dagger}$ with $[\mathcal{C}]$. The facts that $\sigma_*([\mathcal{C}])=-[\mathcal{C}]$ and that $\ec$ and $\sigma(\ec)$ have different leaf labellings ensure that $\hat{\phi}(\ec)=\hat{\phi}(\sigma(\ec))$. To prove that $\hat{\phi}$ satisfies the switch condition, let $\hat{e}$ be an edge separating two (non necessarily distinct) hexagons $\hat{\mathbb{T}}_1, \hat{\mathbb{T}}_2$ as in Figure \ref{fig_square_switch}. Let $\ec_1, \ec_2$ (resp. $\ec_3, \ec_4$) the two edges of $\hat{\tau}_{\Delta}$ in $\hat{\mathbb{T}}_1$ (resp. $\hat{\mathbb{T}}_2$) adjacent to $\hat{e}$. Then $\ec_1^{\dagger} \cup \ec_2^{\dagger} \cup \ec_3^{\dagger} \cup \ec_4^{\dagger}$ bounds an embedded square in $\hat{\Sigma}$ with $\hat{e}$ a diagonal of this square illustrated in Figure \ref{fig_square_switch}. Denote by $Q \in \mathrm{C}^{BM}_2(\hat{\Sigma}_{\hat{\mathcal{P}}\cup \hat{B}}; \mathbb{Z})$ the class of this square oriented such that $\partial Q = \ec_1^{\dagger}+\ec_2^{\dagger}-\ec_3^{\dagger}-\ec_4^{\dagger}$. Then the switch condition follows from the following equivalences, where $i$ denotes the intersection:
		\end{enumerate}
	$$
	i(\partial Q, [\mathcal{C}])=0 \Leftrightarrow i(\ec_1^{\dagger}, [\mathcal{C}])+i(\ec_2^{\dagger}, [\mathcal{C}])=i(\ec_3^{\dagger}, [\mathcal{C}])+i(\ec_4^{\dagger}, [\mathcal{C}]) 
	\Leftrightarrow \hat{\phi}(\ec_1) +\hat{\phi}(\ec_2) =\hat{\phi}(\ec_3) +\hat{\phi}(\ec_4).$$
	
  \begin{figure}[!h] 
\centerline{\includegraphics[width=6cm]{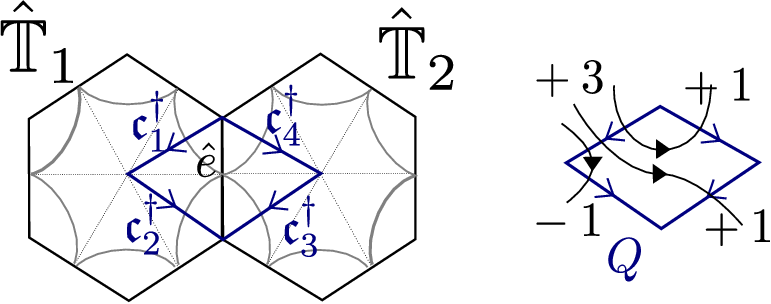} }
\caption{An illustration of the switch condition. On this exemple, it writes $\hat{\phi}(\ec_1) +\hat{\phi}(\ec_2) = 3-1 = 1 + 1 = \hat{\phi}(\ec_3) + \hat{\phi}(\ec_4)$.    } 
\label{fig_square_switch} 
\end{figure} 

The bijection $\varphi$ is illustrated in Figure \ref{fig_phi}.

	  \begin{figure}[!h] 
\centerline{\includegraphics[width=12cm]{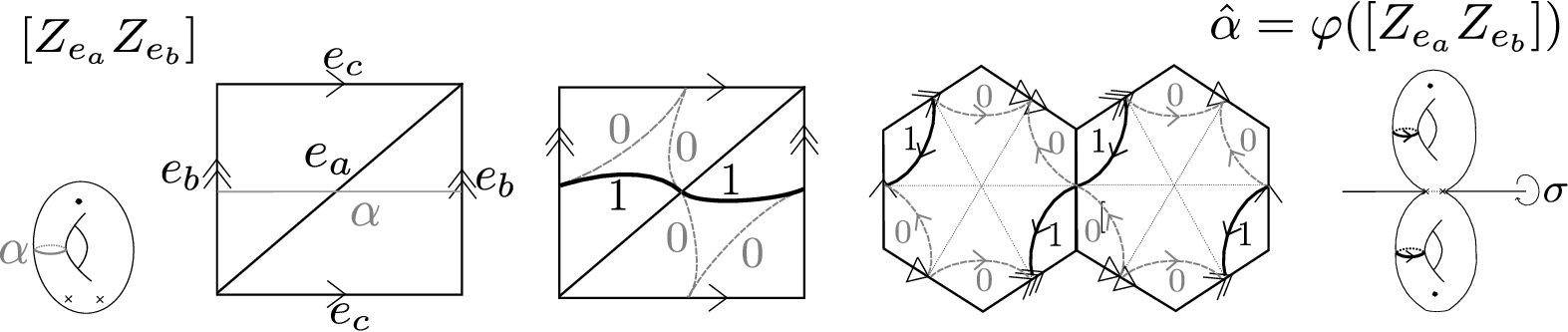} }
\caption{ The image $\varphi([Z_aZ_b])$ in the once punctured torus is the class of the union of two lifts of a simple curve $\alpha$. } 
\label{fig_phi} 
\end{figure}

  To simplify notation, let us denote by 
  $  \Theta_{\ell} \co \mathcal{S}_{\omega}^{\mathbb{C}^*, \sigma}(\mathbf{\hat{\Sigma}}) \to  \mathcal{Z}_{\omega}(\mathbf{\Sigma}, \Delta)$ 
  	the inverse of $\Phi_{\ell}$.  
    	
 	\begin{theorem}\label{th: iso Pso CF Sk}
 			The $\mathcal{R}$--linear isomorphism $\Theta_{\ell}$ is a morphism of algebras. 
 		\end{theorem}

 	\begin{proof}
  		First, remark that the following diagram commutes: 
  		$$\begin{tikzcd}
  		\mathcal{S}_{\omega}^{\mathbb{C}^*, \sigma}(\hat{\mathbf{\Sigma}})   \arrow[r,hook, "i^{\Delta}"] \arrow[d, ,"\Theta_{\ell}" ]& 
  		\otimes_{\mathbb{T}\in F(\Delta)} \mathcal{S}_{\omega}^{\mathbb{C}^*, \sigma}(\hat{\mathbb{T}})  \arrow[d, "\otimes_{\mathbb{T} \in F(\Delta)}\Theta_{\T,\ell}" ] \\
  		\mathcal{Z}_{\omega}(\mathbf{\Sigma}, \Delta)  \arrow[r,hook, "i^{\Delta}"]
  		&   \otimes_{\mathbb{T} \in F(\Delta)}\mathcal{Z}_{\omega}(\mathbb{T}).
  		\end{tikzcd}$$
  		Both the top and bottom horizontal morphisms, from \eqref{eq: i decomp skein equiv} and \eqref{eq: decompo morph i for CF},  are morphisms of algebras. We conclude the proof with the following Lemma \ref{lemmaisomorphism}. 
  	\end{proof}

 	\begin{lemma}\label{lemmaisomorphism} 
 		The map $\Theta_{\T,\ell}$ is a morphism of algebras.
 	\end{lemma}
 	
 	\begin{proof}
 		Label the edges of $\T$  by $e_1, e_2, e_3$  in the clockwise order and denote by $\ec_i$ be the edges of the train track as in Figure \ref{figtriangle}. 
 		The balanced Chekhov-Fock algebra is generated by the balanced monomials $\widetilde{Z}_{\ec_i}:= [Z_{e_{i-1}}Z_{e_{i+1}}]$ and their inverses with relations $\widetilde{Z}_{\ec_i}\widetilde{Z}_{\ec_{i+1}}=\omega^{2}\widetilde{Z}_{\ec_{i+1}}\widetilde{Z}_{\ec_i}$, for  $i\in \mathbb{Z}/3\mathbb{Z}$. 
		Let $z_{\ec_i}$ be such that  $\Theta_{\T,\ell}(z_{\ec_i})=\widetilde{Z}_{\ec_i}$. A simple skein computation drawn in Figure \ref{figtriangle} shows that  $z_{\ec_i}z_{\ec_{i+1}}= \omega^2 z_{\ec_{i+1}}z_{\ec_i}$. Therefore,  $\Theta_{\T,\ell}$ is a morphism of algebras. 
 		\begin{figure}[h!] 
 			\centerline{\includegraphics[width=12cm]{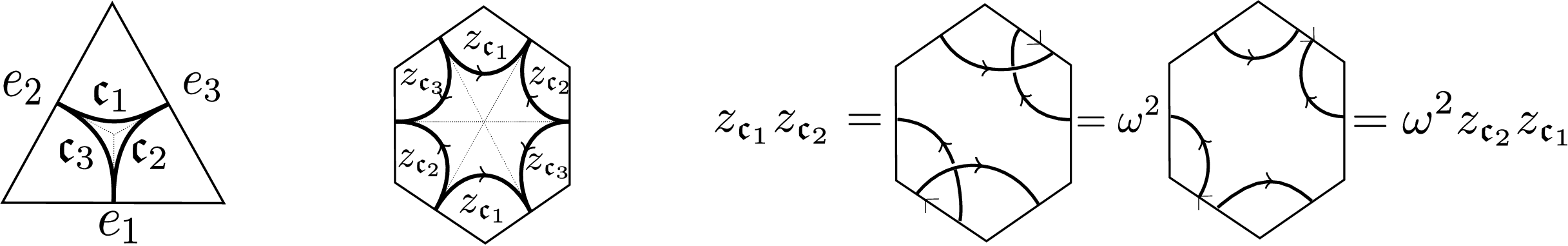} }
 			\caption{On the left, the triangle, its lift and the associated train tracks. On the right, a computation  in the equivariant skein algebra.}
 			\label{figtriangle} 
 		\end{figure} 
 	\end{proof}

\subsection{Irreducible and local representations of the balanced Chekhov-Fock algebras at odd roots of unity}

	Throughout this section, we suppose that $\omega$ is a root of unity of odd order $N>1$. 
	We classify the irreducible and the local representations of the balanced Chekhov-Fock algebra.  
	
	\subsubsection{Irreducible representations}

We first classify the irreducible representations of $\CF$.

For an inner puncture $p$ of $\mathbf{\Sigma}$ consider a peripheral curve $\gamma_p$ encircling $p$. For $e$ an edge of the triangulation, let $k_p(e)\in \{0,1,2\}$ be the number of endpoints of $e$ that are equal to the puncture $p$. Set   
\begin{equation}\label{eq: Hp}
H_p=[\prod_e Z_e^{k_p(e)}]\in \mathcal{Z}_{\omega}(\mathbf{\Sigma}, \Delta).
\end{equation}
The element $\hat{\gamma}_p:=\varphi(H_p)$ is the class of the union of two loops around the two lifts of $p$ that project to $\gamma_p$.

For a  boundary component $\partial$ of $\partial\Sigma$ and for each edge $e$, let $k_{\partial}(e)\in \{0, 1, 2\}$ be the number of endpoints of $e$ that lie in $\partial$.  
Let 
\begin{equation}\label{eq: Hpartial}
H_{\partial}=[\prod_e Z_e^{k_{\partial}(e)}]. 
\end{equation}
The element $\hat{\alpha}_{\partial}:= \varphi(H_{\partial})$ is the class of a union of arcs around the lifts of the punctures of $\partial \cap \mathcal{P}$. For instance, in $\mathbb{P}_n$, using the notations of Figure \ref{figelementarycob}, $\hat{\alpha}_{\partial}$ is the union $\alpha_1 \cup \sigma(\alpha_1)^{-1} \cup \ldots \cup \alpha_n \cup \sigma(\alpha_n)^{-1}$.

\begin{lemma}
	The elements $H_p$ and $H_{\partial}$ are central. 
\end{lemma}
\begin{proof} This is a straightforward consequence of the definition of the Weil-Petersson form. \end{proof}
Define an embedding $j_{(\mathbf{\Sigma}, \Delta)}$ of $\mathcal{Z}_{+1}(\mathbf{\Sigma}, \Delta)$ into the center of $\mathcal{Z}_{\omega}(\mathbf{\Sigma}, \Delta)$ by sending a balanced monomial $[\prod_e Z_e^{k_e}]$ to $j_{(\mathbf{\Sigma}, \Delta)}\left([\prod_e Z_e^{k_e}]\right) := [\prod_e Z_e^{Nk_e}]$. By Corollary \ref{cor: skein equ iso to reg. func. of char.} and Theorem \ref{th: iso Pso CF Sk}, one has an isomorphism between $\mathcal{Z}_{+1}(\mathbf{\Sigma}, \Delta)$ and the algebra of regular functions of  $\mathcal{X}_{\mathbb{C}^*}^{\sigma}(\hat{\mathbf{\Sigma}})$ which only depends on the choice of a leaf labeling.

An irreducible representation $\rho: \CF \to \mathrm{End}(V)$, induces a character on the center of $\CF$, hence a point in $\mathcal{X}_{\mathbb{C}^*}^{\sigma}(\hat{\mathbf{\Sigma}})$  that we denote by $[\rho^{ab}]$ and call \textit{the abelian classical shadow} of the representation. The images by this character of the central elements $\hat{\gamma}_p$ and $\hat{\alpha}_{\partial}$ are called, respectively, the \textit{puncture and boundary invariants} of the representation.

Before we state the classification result (Theorem \ref{theorem2}), let us define the following number. Suppose that $\mathbf{\Sigma}=(\Sigma, \mathcal{P})$ is a punctured surface with $\Sigma$ of genus $g$ with $n_{\partial}$ boundary components and $\mathcal{P}$ of cardinality $s$.
For each boundary component $\partial_i$ of $\Sigma$, let $s_{\partial_i}:=|\partial_i \cap \mathcal{P}|$. 
For $n\geq 1$, let $d(n)$ be  $(n-1)/2$ if $n$ is odd, and be $n/2$ if $n$ is even. 
We set $d_{\partial}:= \sum_{i} d(s_{\partial_i})$, where $i$ runs trough the number of boundary components of $\Sigma$. Recall that $\hat{g}$ denotes the genus of $\hat{\Sigma}$ computed in Lemma \ref{lemmagenus}.

	\begin{theorem}\label{prop_classif_irrep}
	The balanced Chekhov-Fock algebra $\CF$ is semi-simple. 
	Each simple module $\rho$ of $\CF$ has dimension $N^{\hat{g}-g+d_{\partial}}=N^{3g-3+s+n_{\partial}}$ and is determined, up to isomorphism, by:
		\begin{enumerate}
			\item an abelian classical shadow $[\rho^{ab}]\in \mathcal{X}_{\mathbb{C}^*}^{\sigma}(\hat{\mathbf{\Sigma}})$; 
			\item for each inner puncture $p$, an $N$--root of the holonomy of $[\rho^{ab}]$ around $\hat{\gamma}_p$ (puncture invariant); 
			\item for each boundary component $\partial$, an $N$--root of the holonomy of $[\rho^{ab}]$ along $\hat{\alpha}_{\partial}$ (boundary invariant). 
		\end{enumerate}
	\end{theorem}
\begin{proof}
	By Theorem \ref{th: iso Pso CF Sk} and Proposition \ref{propdecomposition}, up to Morita equivalence, $\CF$ decomposes into 
	 \begin{equation*}
	\mathcal{A}:= \mathcal{R}[H_p^{\pm 1}]^{\otimes \mathring{s}} \otimes \mathcal{W}_{q^2}^{\ot n_1} 
	\otimes \mathcal{W}_q^{\ot n_2} \otimes  \otimes_{i\in I} \mathcal{Y}^{(s_i)}_q, 
	\end{equation*}
	where $n_1=(n_b+n_{\partial}^{odd}-2)/2$ and $n_2=(2g-n_b-n_{\partial}^{odd}+2)/4$. 

	Using Lemma \ref{lemmairrepelementary} and the study of $\mathcal{W}_q$ and $\mathcal{W}_{q^2}$, all factors that appear in the tensor decomposition of $\mathcal{A}$ are semi-simple. Moreover, the sets of  isomorphism classes of their simple modules are in bijection with the sets of their induced characters on their centers. 
	Note that Morita equivalences preserve, up to isomorphism, centers; a direct look on $\Theta_{\ell}$ shows that: 
	the central elements $[\prod_i Z_i]^N$ that appear in the factors $\mathcal{Y}_{q}^{(s_i)}$, $\mathcal{W}_q$ and $\mathcal{W}_{q^2}$ are sent to  the central elements defining $[\rho^{ab}]$;  
	the central elements $H_p$ that appear in the factors $\mathbb{C}[H_p^{\pm 1}]$ are sent to $\hat{\gamma}_{p}$; and, 
	the central elements $H_{\partial}$ that appear in the factors $\mathcal{Y}_q^{(s_i)}$ are sent to $\hat{\alpha}_{\partial}$. 
	
	The dimension of the simple modules of $\mathcal{A}$ follows from a straightforward computation using Lemma \ref{lemmairrepelementary} and the fact that the simple modules of $\mathcal{W}_q$ and $\mathcal{W}_{q^2}$ have dimension $N$.
\end{proof}
 
 \subsubsection{Local representations}

 	\begin{definition}[\cite{BonahonBaiLiuLocalRep}]
 		A \emph{local representation} of $\CF$ is a representation of the form 
 		\begin{equation*}
 		r : \mathcal{Z}_{\omega}(\mathbf{\Sigma},\Delta) \xrightarrow{i^{\Delta}} \otimes_{\mathbb{T}\in F(\Delta)} \mathcal{Z}_{\omega}(\mathbb{T}) \xrightarrow{\otimes_{\mathbb{T}} r_{\mathbb{T}}} \End( \otimes_{\mathbb{T}\in F(\Delta)} W_{\mathbb{T}}),
 		\end{equation*}
 		where, for each triangle, $r_{\mathbb{T}}: \mathcal{Z}_{\omega}(\mathbb{T})\rightarrow \End(W_{\mathbb{T}})$ is an irreducible representation. 
 	\end{definition}

	Note that the elements of the form $i^{\Delta}(Z_{e_{i_1}}^N\cdots Z_{e_{i_k}}^N)$ are central in $\otimes_{\mathbb{T}\in F(\Delta)}\mathcal{Z}_{\omega}(\mathbb{T})$. Therefore, the local representations send them to scalar operators, which implies that any local representation $r$ has a well-defined abelian classical shadow $[\rho^{ab}]\in \mathcal{X}_{\mathbb{C}^*}^{\sigma}(\hat{\mathbf{\Sigma}})$. 
	Moreover, the image in $\otimes_{\mathbb{T}\in F(\Delta)}\mathcal{Z}_{\omega}(\mathbb{T})$ of the element 
	\begin{equation*}
	H_c:= [\prod_{e\in \mathcal{E}(\Delta)} Z_e^2]
	\end{equation*} 
	 is central and each local representation sends it to a scalar operator $h_c \id$. 
	The complex $h_c$ is called \textit{the central charge} of the local representation. 
	
	It is shown in  \cite[Proposition 4.6]{Toulisse16} that the isomorphism classes of local representations (for arbitrary punctured surfaces) are classified by their abelian classical shadow and their central charge. 
	
	The decomposition of local representations into simple modules is known for closed punctured surfaces; see \cite{Toulisse16}. 
	We now generalize this decomposition to open surfaces.
 \begin{corollary}\label{corollary2}
 Let $r:\CF \to \End(W)$ be a local representation with classical shadow $[\rho^{ab}]$ and central charge $h_C$. 
 The set of isomorphism classes of simple submodules of $W$ is the set of classes of those irreducible representations with classical shadow $[\rho^{ab}]$ such that the product of every boundary and inner puncture invariants is equal to $h_C$. 
 Moreover, each factor arises with multiplicity $N^g$, where $g$ is the genus of $\Sigma$.
 \end{corollary}

\par Consider $r:\mathcal{Z}_{\omega}(\mathbf{\Sigma}, \Delta )\rightarrow \End(W)$ a local representation and denote by $\overline{r}:= \otimes_{\mathbb{T}} r^{\mathbb{T}} : \otimes_{\mathbb{T}\in F(\Delta)} \mathcal{Z}_{\omega}(\mathbb{T})\rightarrow \End(W)$ the underlying representation such that $r=\overline{r}\circ i^{\Delta}$. A simple sub-module of $W$ has necessarily the same classical abelian shadow by definition. Moreover since $H_c$ is the product of every elements $H_p$ and $H_{\partial}$, the product of every puncture and boundary invariants of a simple sub-module must be equal to the central charge. Denote by $p_1, \ldots, p_s$ the inner punctures of $\mathbf{\Sigma}$ and $\partial_1, \ldots, \partial_m$ its boundary components and write $J= \{ p_1, \ldots, p_s, \partial_1, \ldots, \partial_m\}$. Let $\mathcal{D}$ denote the set of maps $h : J \rightarrow \mathbb{C}^*$ such that $h(k)^N = \rho^{ab}(H_k)$ for all $k\in J$ and such that $\prod_{k\in J} h(k)=h_C$. One has a decomposition
$$ W = \oplus_{h \in \mathcal{D}} W^{h} \quad \mbox{, where }W^h = \{ w \in W | r(H_k) w = h(k) w \mbox{, for all }k \in J \}.$$

\begin{lemma}\label{lemma_multiplicity}
The dimension of the subspace $W^h$ does not depend on $h\in \mathcal{D}$.
\end{lemma}

\begin{proof} When $J$ has cardinality one, the result is obvious so we suppose that $|J|\geq 2$. 
\\ Let $G$ denote the group of those maps $g : J \rightarrow \mathbb{Z}/N\mathbb{Z}$ such that $\sum_{k \in J} g(k) = 0$. The group $G$ acts freely and transitively on the set $\mathcal{D}$ by the action $g\cdot h (k) := \omega^{g(k)}h(k)$ for $k\in J, g\in G$ and $h\in \mathcal{D}$.
\\ Consider a tree $T\subset \Sigma_{\mathcal{P}}$ such that $(1)$ its set $V(T)$ of vertices is the union of the set of inner punctures of $\mathbf{\Sigma}$ together with a subset of $P$ which intersects each boundary component exactly once, so $V(T)$ is in natural bijection with $J$, and $(2)$ the set of edges $\mathcal{E}(T)$ is included in $\mathcal{E}(\Delta)$. For each edge $e\in \mathcal{E}(T)$, we associate an element $g_e \in G$ as follows. Let $p_1, p_2$ be the two endpoints of $e$. For $i=1,2$, let $k_i \in J$ be the element which is either $p_i$ if $p_i$ is an inner puncture or $\partial_i$ if $p_i$ is a puncture in the boundary component $\partial_i$. Note that by definition of $T$, one has $k_1 \neq k_2$. The element $g_e \in G$ is defined by $g_e(k_1) = 4, g_e(k_2)=-4$ and $g_e(k)=0$ for $k\neq k_1, k_2$. 

Let us prove that the elements $g_e$ generate $G$. Consider the CW chain complex $(\mathrm{C}_{\bullet}(T; \mathbb{Z}/N\mathbb{Z}), \partial_{\bullet})$  and
identify $G$ with the subset of chains $\sum_{v\in V(T)} n_v v \in \mathrm{C}_0(T; \mathbb{Z}/N\mathbb{Z})$ such that $\sum_v n_v=0$ by sending $g\in G$ to $\sum_v g(v) v \in \mathrm{C}_0(T; \mathbb{Z}/N\mathbb{Z})$. Clearly the image of $\partial_1 : \mathrm{C}_1(T; \mathbb{Z}/N\mathbb{Z}) \to  \mathrm{C}_0(T; \mathbb{Z}/N\mathbb{Z})$ is included in $G$ and 
since $T$ is connected, we have $\mathrm{H}_0(T; \mathbb{Z}/N\mathbb{Z})=\quotient{\mathrm{C}_0(T; \mathbb{Z}/N\mathbb{Z}) }{\mathrm{Im}(\partial_1)} \cong \mathbb{Z}/N\mathbb{Z}$, so $G=\mathrm{Im}(\partial_1)$.  Note that $g_e = 4 \partial_1 (\delta_e)$ where $\delta_e \in \mathrm{C}_1(T; \mathbb{Z}/N\mathbb{Z})\cong (\mathbb{Z}/N\mathbb{Z})^{\mathcal{E}(T)}$ is the generator corresponding to $e$, so the family $(\partial_1 (\delta_e))_{e\in \mathcal{E}(T)}$ generates $G$ and, since $4$ is prime to $N$, so does the family $(g_e)_{e\in \mathcal{E}(T)}$.

 Now fix $e\in \mathcal{E}(T)$ and orient it from $p_1$ to $p_2$ and let $\mathbb{T}_e$ be the face of $\Delta$ on the left of $e$ and $e^L$ be the corresponding lift of $e$ in $\mathcal{E}(\mathbb{T}_e)$ (note that when the triangle $\mathbb{T}_e$ is self-folded, $e$ admit two lifts in $\mathcal{E}(\mathbb{T}_E)$ and we look at the one on the left). Let $Z_{e_L}^2 \in \mathcal{Z}_{\omega}(\mathbb{T}_e) \subset \otimes_{\mathbb{T}} \mathcal{Z}_{\omega}(\mathbb{T})$ and write $X_e := \overline{r}(Z_{e_L}^2) \in \End(W)$. 
In the algebra $ \otimes_{\mathbb{T}} \mathcal{Z}_{\omega}(\mathbb{T})$, one has the relations
$$ Z_{e_L}^2 H_{k_1} = \omega^4 H_{k_1} Z_{e_L}^2  \quad,   Z_{e_L}^2 H_{k_2} = \omega^{-4} H_{k_2} Z_{e_L},  \quad Z_{e_L}^2 H_{k} =  H_{k} Z_{e_L}, \forall k\neq k_1, k_2$$
from which we deduce that $X_e$ induces an isomorphism between $W^{h}$ and $W^{g_e \cdot h}$, so they have the same dimension. We conclude using the facts that the $g_e$ generate $G$ and that $G$ acts transitively on $\mathcal{D}$.

\end{proof}

 \begin{proof}[Proof of Corollary \ref{corollary2}] 
 By Theorem \ref{prop_classif_irrep}, each summand $W^h$ is a direct sum of pairwise isomorphic simple $\mathcal{Z}_{\omega}(\mathbf{\Sigma}, \Delta)$-modules with classical shadow $\rho^{ab}$ and inner puncture and boundary invariants $h_k = h(k)$ for all $k\in J$. By Lemma \ref{lemma_multiplicity}, all such simple module arise with the same multiplicity $M$. It remains to prove that $M=N^g$ to conclude. On the one hand $\dim (W) = N^{|F(\Delta)|}$, where $|F(\Delta)|$ is the number of faces of the triangulation, so is equal to $|F(\Delta)|= 4g-4+2\mathring{s} +2n_{\partial} +s_{\partial}$. On the other hand, by Theorem \ref{prop_classif_irrep}, each summand $W^h$ has dimension $M\times N^{3g-3+s+n_{\partial}}$ and there are $|\mathcal{D}|= N^{|J| -1} = N^{\mathring{s}+n_{\partial}-1}$ such summand. Therefore, one has the equality
 $$ N^{4g-4+2\mathring{s}+ 2n_{\partial} +s_{\partial}} = M \times N^{3g-3+s+n_{\partial}} \times N^{\mathring{s}+n_{\partial} -1}, $$
 which implies $M=N^g$.

 \end{proof}

 \section{Non-abelianization maps}
 
In this section we define the non-abelianization map  
 	\begin{equation*}
 	\mathcal{NA}: \mathcal{X}_{\mathbb{C}^*}^{\sigma}(\mathbf{\hat{\Sigma}}) \rightarrow \mathcal{X}_{\SL_2}(\mathbf{\Sigma})
 	\end{equation*}
 	and give an explicit description of it. It is associated to a morphism of Poisson algebras  
 	\begin{equation*}
	\mathcal{NA}^*\co \mathcal{O}[\mathcal{X}_{\SL_2}(\mathbf{\Sigma})]\to \mathcal{O}[\mathcal{X}_{\mathbb{C}^*}^{\sigma}(\mathbf{\hat{\Sigma}})]; 
 	\end{equation*} 
 	the core of the latter is the quantum trace map  $\tr_{+1}\co \mathcal{S}_{+1}^{\SL_2}(\mathbf{\Sigma}) \to  \mathcal{Z}_{+1}(\mathbf{\Sigma}, \Delta)$, reinterpreted with values into $ \mathcal{S}_{+1}^{\mathbb{C}^*, \sigma}(\mathbf{\hat{\Sigma}})$ by means of the isomorphism of Theorem \ref{th: iso Pso CF Sk}.

  In the first section we give the main ingredients that compose the non-abelianization map. We define it in the second one. 
  We then restrict ourselves to closed punctured surfaces and prove Theorem \ref{theorem3}.  
  
   \subsection{Kauffman-bracket skein algebras, the quantum trace map and character varieties}
 \par In this subsection, we briefly recall the definitions of the Kauffman-bracket (stated) skein algebras,  of the quantum trace map and the character varieties and state some of their properties. Nothing original is claimed here except Lemma \ref{lemma_poisson_abelian}. 
 
 \subsubsection{The Kauffman-bracket stated skein algebra. }
 Consider a punctured surface $\mathbf{\Sigma}$, a commutative unital ring $\mathcal{R}$ and an invertible element $\omega\in \mathcal{R}^{\times}$. A \textit{stated tangle} $(T,s)$ is the data of a tangle $T$ and a map (a state) $s:\partial T \rightarrow \{-,+\}$. A stated diagram is defined similarly.
 \begin{definition}\cite{BonahonWongqTrace, LeStatedSkein}
  The Kauffman-bracket (stated) skein algebra $\mathcal{S}_{\omega}^{\SL_2}(\mathbf{\Sigma})$ is the quotient of the free $\mathcal{R}$-module generated by isotopy classes of stated unoriented tangles in $\mathbf{\Sigma}$ by the skein relations \eqref{eq: skein 1} and \eqref{eq: skein 2}, which are,    
\\ 
the Kauffman bracket relations:
	\begin{equation}\label{eq: skein 1} 
\begin{tikzpicture}[baseline=-0.4ex,scale=0.5,>=stealth]	
\draw [fill=gray!45,gray!45] (-.6,-.6)  rectangle (.6,.6)   ;
\draw[line width=1.2,-] (-0.4,-0.52) -- (.4,.53);
\draw[line width=1.2,-] (0.4,-0.52) -- (0.1,-0.12);
\draw[line width=1.2,-] (-0.1,0.12) -- (-.4,.53);
\end{tikzpicture}
=\omega^{-2}
\begin{tikzpicture}[baseline=-0.4ex,scale=0.5,>=stealth] 
\draw [fill=gray!45,gray!45] (-.6,-.6)  rectangle (.6,.6)   ;
\draw[line width=1.2] (-0.4,-0.52) ..controls +(.3,.5).. (-.4,.53);
\draw[line width=1.2] (0.4,-0.52) ..controls +(-.3,.5).. (.4,.53);
\end{tikzpicture}
+\omega^{2}
\begin{tikzpicture}[baseline=-0.4ex,scale=0.5,rotate=90]	
\draw [fill=gray!45,gray!45] (-.6,-.6)  rectangle (.6,.6)   ;
\draw[line width=1.2] (-0.4,-0.52) ..controls +(.3,.5).. (-.4,.53);
\draw[line width=1.2] (0.4,-0.52) ..controls +(-.3,.5).. (.4,.53);
\end{tikzpicture}
\hspace{.5cm}
\text{ and }\hspace{.5cm}
\begin{tikzpicture}[baseline=-0.4ex,scale=0.5,rotate=90] 
\draw [fill=gray!45,gray!45] (-.6,-.6)  rectangle (.6,.6)   ;
\draw[line width=1.2,black] (0,0)  circle (.4)   ;
\end{tikzpicture}
= -(\omega^{-4}+\omega^{4}) 
\begin{tikzpicture}[baseline=-0.4ex,scale=0.5,rotate=90] 
\draw [fill=gray!45,gray!45] (-.6,-.6)  rectangle (.6,.6)   ;
\end{tikzpicture}
;
\end{equation}
the boundary relations:
\begin{equation}\label{eq: skein 2} 
\begin{tikzpicture}[baseline=-0.4ex,scale=0.5,>=stealth]
\draw [fill=gray!45,gray!45] (-.7,-.75)  rectangle (.4,.75)   ;
\draw[->] (0.4,-0.75) to (.4,.75);
\draw[line width=1.2] (0.4,-0.3) to (0,-.3);
\draw[line width=1.2] (0.4,0.3) to (0,.3);
\draw[line width=1.1] (0,0) ++(90:.3) arc (90:270:.3);
\draw (0.65,0.3) node {\scriptsize{$+$}}; 
\draw (0.65,-0.3) node {\scriptsize{$+$}}; 
\end{tikzpicture}
=
\begin{tikzpicture}[baseline=-0.4ex,scale=0.5,>=stealth]
\draw [fill=gray!45,gray!45] (-.7,-.75)  rectangle (.4,.75)   ;
\draw[->] (0.4,-0.75) to (.4,.75);
\draw[line width=1.2] (0.4,-0.3) to (0,-.3);
\draw[line width=1.2] (0.4,0.3) to (0,.3);
\draw[line width=1.1] (0,0) ++(90:.3) arc (90:270:.3);
\draw (0.65,0.3) node {\scriptsize{$-$}}; 
\draw (0.65,-0.3) node {\scriptsize{$-$}}; 
\end{tikzpicture}
=0,
\hspace{.2cm}
\begin{tikzpicture}[baseline=-0.4ex,scale=0.5,>=stealth]
\draw [fill=gray!45,gray!45] (-.7,-.75)  rectangle (.4,.75)   ;
\draw[->] (0.4,-0.75) to (.4,.75);
\draw[line width=1.2] (0.4,-0.3) to (0,-.3);
\draw[line width=1.2] (0.4,0.3) to (0,.3);
\draw[line width=1.1] (0,0) ++(90:.3) arc (90:270:.3);
\draw (0.65,0.3) node {\scriptsize{$+$}}; 
\draw (0.65,-0.3) node {\scriptsize{$-$}}; 
\end{tikzpicture}
=\omega
\begin{tikzpicture}[baseline=-0.4ex,scale=0.5,>=stealth]
\draw [fill=gray!45,gray!45] (-.7,-.75)  rectangle (.4,.75)   ;
\draw[-] (0.4,-0.75) to (.4,.75);
\end{tikzpicture}
\hspace{.1cm} \text{ and }
\hspace{.1cm}
\omega^{-1}
\heightexch{->}{-}{+}
- \omega^{-5}
\heightexch{->}{+}{-}
=
\heightcurve.
\end{equation}
The product is given by stacking the stated tangles; $[T_1,s_1]\cdot [T_2,s_2]$ denotes  $[T_1, s_1]$ placed above $[T_2, s_2]$. 
\end{definition}
   In addition to the fact that we use different skein relations and states, we impose two main differences between the skein algebras $\mathcal{S}_{\omega}^{\SL_2}(\mathbf{\Sigma})$ and $\mathcal{S}_{\omega}^{\mathbb{C}^{*}}(\mathbf{\Sigma})$: 
 \begin{enumerate}
 \item In $\mathcal{S}_{\omega}^{\SL_2}(\mathbf{\Sigma})$, the tangles are unoriented.
 \item In $\mathcal{S}_{\omega}^{\SL_2}(\mathbf{\Sigma})$, given a boundary arc $b$ and a tangle $T$, we impose that the points of $\partial_b T:= T\cap b$ have pairwise distinct heights.
 \end{enumerate}

 \par A \textit{closed curve} is a closed connected reduced (unoriented) diagram and an \textit{arc} is an open connected reduced diagram. The algebra  $\mathcal{S}_{\omega}^{\SL_2}(\mathbf{\Sigma})$ is generated by classes of closed curves and stated arcs. 
 
 \subsubsection{The quantum trace map.}\label{sec: qtrace}

	For a  triangulated punctured surface $(\mathbf{\Sigma}, \Delta)$, we recall from \cite{BonahonWongqTrace, LeStatedSkein} the definition of the  \textit{quantum trace map} 
	\begin{equation}\label{eq: qtrace}
	\tr_{\omega}^{\Delta} : \mathcal{S}_{\omega}^{\SL_2}(\mathbf{\Sigma}) \rightarrow \mathcal{Z}_{\omega}(\mathbf{\Sigma}, \Delta), 
	\end{equation}
	in the particular case where $\omega=+1$. 
	The quantum trace is a morphism of algebras which is injective if and only if $\Sigma$ has no boundary (see \cite[Proposition $29$]{BonahonWongqTrace} for the "if" and \cite[Section $7$]{CostantinoLe19} for the "only if"). 
	\\ 
	 
	For each triangle of $\Delta$, consider the clockwise cyclic order on its edges. 
	Let $\alpha$ be either a closed curve or an arc between two boundary components; suppose it is placed in minimal and transversal position with the edges of $\Delta$. Denote by $Q_{\alpha}$ the set $\alpha\cap \mathcal{E}(\Delta)$.  
	For any two points $x$ and $y$ in $Q_{\alpha}$, we write $x\to y$ if: 
	\begin{itemize}
		\item 	$x$ and $y$ respectively belong to edges $e$ and $e'$ of a common triangle $\mathbb{T}$; 
		\item   $e$ is the immediate predecessor of $e'$ for the cyclic order; and, 
		\item $x$ and $y$ are on a same connected component of $\alpha\cap \mathbb{T}$.   
	\end{itemize}
	Let $\St^a(\alpha)$ be the set of maps $s: Q_{\alpha} \rightarrow \{-,+\}$ such that $(s(x),s(y))\neq (-,+)$ whenever $x\to y$. 
	For each state $s$, let $k_s: \mathcal{E}(\Delta)\to \mathbb{Z}$ be given by $k_s(e)=\sum_{x\in e} s(x)$ for $e\in  \mathcal{E}(\Delta)$, where we identify $-$ with $-1$ and $+$ with $+1$.     
	For a closed curve $\alpha$ as above, one sets	
	\begin{equation*}
	\tr_{+1}^{\Delta} (\alpha) = \sum_{s\in \St^a(\alpha)} \prod_{e\in \mathcal{E}(\Delta)}Z_{e}^{k_s(e)}.
	\end{equation*}
	For an arc  $\alpha$  between two boundary arcs with state  $s_0:\partial \alpha \to \{-,+\}$, the quantum trace is as follows.  
	If $\alpha$  bounds two arcs $e$ and $e'$ of a triangle, say at $x$ and $y$, and if $s_0$ is such that $(s_0(x),s_0(y))= (-,+)$, then $	\tr_{+1}^{\Delta} (\alpha,s_0) =0$. 
	Otherwise, let $\St^a(\alpha,s_0)$ be the subset of $\St^a(\alpha)$ of those states that agree with $s_0$ on $\partial \alpha$. 
	One sets
	\begin{equation}\label{eq_QTrace}
	\tr_{+1}^{\Delta} (\alpha,s_0) = \sum_{s\in \St^a(\alpha,s_0)} \prod_{e\in \mathcal{E}(\Delta)}Z_{e}^{k_s(e)}.
	\end{equation}
	The map $\tr_{+1}^{\Delta}$ is obtained by algebraic extension; we refer to \cite[Lemma $2.15$]{KojuQGroupsBraidings} for a proof that this definition coincides with the ones in \cite{BonahonWongqTrace, LeStatedSkein}.

 \subsubsection{Poisson structures on skein algebras.}

 	Let $G$ be either $\mathbb{C}^*$ or $\SL_2(\mathbb{C})$. 
 	Note that the skein algebra $\mathcal{S}_{\omega}^G(\mathbf{\Sigma})$ admits canonical linear basis. 
 	By canonical, we mean that it does not depend on the ring $\mathcal{R}$ nor on the parameter $\omega$. 
 	For $G=\mathbb{C}^*$, such a basis is given by $\mathrm{H}_1(\Sigma_{\mathcal{P}}, \partial \Sigma_{\mathcal{P}}; \mathbb{Z})$; for $G=\SL_2(\mathbb{C})$, such a basis was defined in \cite{LeStatedSkein}. 
 	
 	For $\mathcal{R}=\mathbb{C}$ and $\omega=+1$, let us equip the commutative algebra  $\mathcal{S}_{+1}^G(\mathbf{\Sigma})$ with the following Poisson bracket. 
 	Consider the ring $\mathbb{C}[[\hbar]]$ of formal power series and let  $\mathcal{S}_{+1}^G(\mathbf{\Sigma})[[\hbar]]:= \mathcal{S}_{+1}^G(\mathbf{\Sigma}) \otimes_{\mathbb{C}} \mathbb{C}[[\hbar]]$. 
 	We also consider the skein algebra $\mathcal{S}_{\omega_{\hbar}}^G(\mathbf{\Sigma})$  with parameter $\omega_{\hbar}:= \exp(-\hbar/4)\in \mathbb{C}[[\hbar]]=\mathcal{R}$. 
 	Each canonical basis $\mathcal{B}$ determines an isomorphism of  $\mathbb{C}[[\hbar]]$-modules 
 	$\phi^{\mathcal{B}}: \mathcal{S}_{+1}^G(\mathbf{\Sigma})[[\hbar]]\xrightarrow{\cong} \mathcal{S}_{\omega_{\hbar}}^G(\mathbf{\Sigma})$.  
 	Denote by $\star_{\mathcal{B}}$ the product on $\mathcal{S}_{+1}^G(\mathbf{\Sigma})[[\hbar]]$ obtained by pulling back along $\phi^{\mathcal{B}}$ the product of $ \mathcal{S}_{\omega_{\hbar}}^G(\mathbf{\Sigma})$. 
 	\begin{definition}
 		The Poisson bracket $\{\cdot, \cdot \}$ on $\mathcal{S}_{+1}^G(\mathbf{\Sigma})$ is defined by the formula
 		$$ f \star_{\mathcal{B}} g - g\star_{\mathcal{B}} f = \hbar \{f,g\} \pmod{\hbar^2} \mbox{, for }f,g \in \mathcal{S}_{+1}^G(\mathbf{\Sigma}).$$ 
 	\end{definition}
 The Poisson bracket is well-defined: if $\mathcal{B}$ and $\mathcal{B}'$ are two canonical basis, the algebra isomorphism $(\phi^{\mathcal{B}'})^{-1}\circ \phi^{\mathcal{B}} : (\mathcal{S}_{+1}^G(\mathbf{\Sigma})[[\hbar]], \star_{\mathcal{B}}) \xrightarrow{\cong} (\mathcal{S}_{+1}^G(\mathbf{\Sigma})[[\hbar]], \star_{\mathcal{B}'})$ is a \textit{gauge equivalence}, and it follows from classical properties of gauge equivalences (see \textit{e.g.}\cite{KontsevichQuantizationPoisson}, \cite[II.2]{GRS_QuantizationDeformation}) that $\{\cdot, \cdot\}$ does not depend on the choice of the canonical basis. 
 \begin{remark}
 The algebra $(\mathcal{S}_{+1}^G(\mathbf{\Sigma})[[\hbar]], \star_{\mathcal{B}})$ is a \textit{deformation quantization} of the commutative Poisson algebra $(\mathcal{S}_{+1}^G(\mathbf{\Sigma}), \{\cdot, \cdot \})$. 
 \end{remark}

 Given an involutive punctured surface $(\mathbf{\Sigma}, \sigma)$, we equip the algebra $\mathcal{S}_{+1}^{\mathbb{C}^*, \sigma}(\mathbf{\Sigma})$ with a Poisson bracket in the same manner.  
 
 Recall from  Corollary \ref{cor: skein equ iso to reg. func. of char.} that one has an isomorphism of algebras 
 $h^{\sigma}: \mathbb{C}[\mathcal{X}_{\mathbb{C}^*}^{\sigma}(\mathbf{\hat{\Sigma}})] \xrightarrow{\cong}  \mathcal{S}_{+1}^{\mathbb{C}^*, \sigma}(\mathbf{\hat{\Sigma}})$ which sends any element of $\mathrm{H}_1^{\sigma}(\Sigma_{\mathcal{P}}, \partial \Sigma_{\mathcal{P}}; \mathbb{Z})$ to the class of a Weyl-ordered reduced diagram that represents it. 
  \begin{lemma}\label{lemma_poisson_abelian}
 	The isomorphism $h^{\sigma}$ is Poisson.
 \end{lemma}
 \begin{proof} By Proposition \ref{skeinalgstructure}, given $[\alpha], [\beta] \in \mathrm{H}_1^{\sigma}(\Sigma_{\mathcal{P}}, \partial \Sigma_{\mathcal{P}}; \mathbb{Z})$, we have the equality $[\alpha]\star [\beta]= \omega_{\hbar}^{([\alpha], [\beta])} [\alpha + \beta]$. Hence one has:
 	\begin{eqnarray*}
 		[\alpha] \star [\beta] - [\beta] \star [\alpha] &=& (\omega_{\hbar}^{([\alpha], [\beta])} - \omega_{\hbar}^{([\beta], [\alpha])}) [\alpha +\beta] \\
 		& \equiv & -\frac{1}{2}([\alpha], [\beta]) [\alpha + \beta] \hbar \pmod{\hbar^2} 
 	\end{eqnarray*}
 	\par We deduce that $\{ [\alpha], [\beta] \} = -\frac{1}{2} ([\alpha], [\beta]) [\alpha+\beta]$, hence $h^{\sigma}$ is a Poisson \\ morphism.

 \end{proof}

 	We are now ready to define the \emph{quantum non-abelianization map}. For a leaf labeling $\ell$, we set   
 	\begin{equation}
 	 \nab_{\omega,\ell} : \mathcal{S}_{\omega}^{\SL_2}(\mathbf{\Sigma}) \xrightarrow{\tr_{\omega}^{\Delta}} \mathcal{Z}_{\omega}(\mathbf{\Sigma}, \Delta) \xrightarrow[\cong]{\Phi_{\ell}} \mathcal{S}_{\omega}^{\mathbb{C}^*, \sigma}(\mathbf{\hat{\Sigma}}),
 	\end{equation}
	where $\Phi_{\ell}$ is the isomorphism of Theorem \ref{th: iso Pso CF Sk}. 
 	Note that at $\omega=+1$, the morphism $\nab_{+1,\ell}$ is obtained from the algebra morphism $\nab_{\omega_{\hbar},\ell}$ by reduction modulo $\hbar$. Therefore  $\nab_{+1,\ell} $ is a Poisson morphism.

 \subsubsection{Character varieties}
Given $\mathbf{\Sigma}$ a closed connected punctured surface, the $\SL_2(\mathbb{C})$- character variety is defined as the GIT quotient 
 $$\mathcal{X}_{\SL_2}(\mathbf{\Sigma}):= \Hom(\pi_1(\Sigma_{\mathcal{P}}, v), \SL_2(\mathbb{C}) ) \sslash \SL_2(\mathbb{C})$$
 where $v\in \Sigma_{\mathcal{P}}$ is an arbitrary point. Given $\gamma \subset \Sigma_{\mathcal{P}}$ a loop represented by an element $[\gamma] \in \pi_1(\Sigma_{\mathcal{P}}, v)$, we define a regular function $\tau_{\gamma} \in \mathbb{C}[\mathcal{X}_{\SL_2}(\mathbf{\Sigma})]$ by sending the class $[\rho]$ of a representation $\rho : \pi_1(\Sigma_{\mathcal{P}}, v) \rightarrow \SL_2(\mathbb{C})$ to the complex $\tau_{\gamma}([\rho]):= \tr (\rho([\gamma]))$. It is proved in \cite{PS00} that the  regular functions $\tau_{\gamma}$ generate the algebra $\mathbb{C}[\mathcal{X}_{\SL_2}(\mathbf{\Sigma})]$. In  \cite{Goldman86}, Goldman defined a Poisson bracket on $\mathbb{C}[\mathcal{X}_{\SL_2}(\mathbf{\Sigma})]$. There exists an isomorphism of  Poisson algebras $\Psi': \mathcal{S}_{-1}^{\SL_2}(\mathbf{\Sigma}) \xrightarrow{\cong} \mathbb{C}[\mathcal{X}_{\SL_2}(\mathbf{\Sigma})]$, from the skein algebra with $A=-1$ (recall $A=\omega^{-2}$), characterized by the formula $\Psi'([\gamma]):= -\tau_{\gamma}$. It was proved by Bullock that $\Psi'$ is an isomorphism of algebras if $\mathcal{S}_{-1}^{\SL_2}(\mathbf{\Sigma}) $ is reduced. The latter fact was proved independently in \cite{PS00} and \cite{ChaMa}. The fact that $\Psi'$ is Poisson was proved by Turaev in \cite{Turaev91}. Consider a spin structure $S$ on $\Sigma_{\mathcal{P}}$ represented by quadratic form $w_S: \mathrm{H}_1(\Sigma_{\mathcal{P}}, \mathbb{Z}/2\mathbb{Z})\rightarrow \mathbb{Z}/2\mathbb{Z}$ (see \cite{Johnson_SpinStructures} for details on quadratic forms and spin structures). It follows from the main theorem of  \cite{Barett} that there is an algebra isomorphism $\mathcal{S}_{+1}^{\SL_2}(\mathbf{\Sigma}) \xrightarrow{\cong} \mathcal{S}_{-1}^{\SL_2}(\mathbf{\Sigma})$ sending the class of a loop $[\gamma]$ to $(-1)^{w_S([\gamma])} [\gamma]$. By composition, we obtain an isomorphism of Poisson algebras $\Psi^S : \mathcal{S}_{+1}^{\SL_2}(\mathbf{\Sigma}) \xrightarrow{\cong} \mathbb{C}[\mathcal{X}_{\SL_2}(\mathbf{\Sigma})]$ characterized by $\Psi^S ([\gamma]) := (-1)^{1+w_S([\gamma])} \tau_{\gamma}$.
 
 \vspace{2mm}
 \par When $\mathbf{\Sigma}$ has non-trivial boundary, the first author defined in \cite{KojuTriangularCharVar} an affine Poisson variety $\mathcal{X}_{\SL_2}(\mathbf{\Sigma})$ which generalizes the previous character variety. The Poisson structure however is not canonical and depends on the choice $\mathfrak{o}$ of an orientation of the boundary arcs of $\mathbf{\Sigma}$. In \cite{KojuQuesneyClassicalShadows}, the authors defined a family of Poisson isomorphisms $\Psi^{(\mathfrak{o}, S)} : \mathcal{S}_{+1}^{\SL_2}(\mathbf{\Sigma}) \xrightarrow{\cong} \mathbb{C}[\mathcal{X}_{\SL_2}(\mathbf{\Sigma})]$ which depend on $\mathfrak{o}$ and on the choice of a relative spin structure $S$. Note that for open punctured surfaces, the stated skein algebras at $A= -1$ are no longer commutative, hence we need to work with $A=+1$.

 \subsection{Algebraic non-abelianization maps}

 	\begin{definition}
 		Let $\mathbf{\Sigma}$ be a punctured surface. For a topological triangulation $\Delta$ of $\mathbf{\Sigma}$, an orientation $\mathfrak{o}_{\Delta}$ of its edges, and a leaf labelling $\ell$,  
 		the \textit{non-abelianization map} is the morphism of Poisson varieties 
 		$\mathcal{NA}: \mathcal{X}_{\mathbb{C}^*}^{\sigma}(\mathbf{\hat{\Sigma}}) \rightarrow \mathcal{X}_{\SL_2}(\mathbf{\Sigma})$ that is associated to the Poisson morphism:
 		\begin{equation*}
 		\mathcal{NA}^* : \mathcal{O}[\mathcal{X}_{\SL_2}(\mathbf{\Sigma})] \xrightarrow[\cong]{\Psi^{(\mathfrak{o}, S)}}^{-1} \mathcal{S}_{+1}^{\SL_2}(\mathbf{\Sigma}) \xrightarrow{\nab_{+1,\ell}}  \mathcal{S}_{+1}^{\mathbb{C}^*, \sigma}(\mathbf{\hat{\Sigma}}) \xrightarrow[\cong]{(h^{\sigma})^{-1}} \mathcal{O}[\mathcal{X}_{\mathbb{C}^*}^{\sigma}(\mathbf{\hat{\Sigma}})]. 
 		\end{equation*}
 	\end{definition}
 In view of Theorem \ref{theorem3}, we restrict our study of $\mathcal{NA}$ for closed punctured surfaces. 
 Note that, in this case, $\mathcal{NA}$ depends on a choice of a leaf labelling and on a spin structure $S$. 
 For simplicity, we assume that $S$ is a spin structure on $\Sigma_{\mathcal{P}}$ whose associated quadratic form $w_S$ satisfies $w_S([\gamma_p])=1$ for each peripheral curve $\gamma_p$ encircling a puncture $p\in \mathcal{P}$ (such a $w_S$ exists since $[\gamma_p]$ is in the kernel of the intersection form).
Since the quantum trace map is injective when $\Sigma$ is closed, and since both $\mathcal{X}_{\mathbb{C}^*}^{\sigma}(\mathbf{\hat{\Sigma}})$ and $\mathcal{X}_{\SL_2}(\mathbf{\Sigma})$ are irreducible, the map $\mathcal{NA}$ is dominant and its image is 
 an open Zariski subset in this case.

We now describe $\nab_{+1,\ell}$. Recall that it is composed of the quantum trace map $\Tr_{+1}$ and the isomorphism $\Phi_{\ell}$. We use the same notations than Section \ref{sec: qtrace}; for a curve $\gamma$, 
recall that  
	\begin{equation*}
\tr_{+1}^{\Delta} (\gamma) = \sum_{s\in \St^a(\gamma)} \prod_{e\in \mathcal{E}(\Delta)}Z_{e}^{k_s(e)}.
\end{equation*} 

Let $K(\gamma) \subset K_{\Delta}$ be the set of those balanced monomials that appear in the above expression. 
Denote by $\mathrm{Lt}^a(\gamma)$ the image of $K(\gamma)$ by the bijection $K_{\Delta}\cong  \mathrm{H}_1^{\sigma}(\hat{\Sigma}_{\hat{\mathcal{P}}\cup \hat{B}}; \mathbb{Z})$. 
In each hexagon, the elements of $\mathrm{Lt}^a(\gamma)$ are as follows. 
  For $\T$ a triangle, let $b$ be its branch point. 
  Suppose $\alpha$ is an arc of $\gamma\cap \T$ and $\varepsilon, \varepsilon'\in \{-,+\}$ are two states at its endpoints. 
  Geometrically, the bijection  $K_{\Delta} \cong \mathcal{W}(\tau_{\Delta}, \mathbb{Z})$   determines how to write $(\alpha,\varepsilon,\varepsilon')$ in the train track $\tau_{\Delta}$:  the states $(\varepsilon,\varepsilon')$ set the position of $\alpha$ relatively to the branch point $b$ together with an orientation; see Figure \ref{figliftarc}. 
  Finally, the last bijections $\mathcal{W}(\tau_{\Delta}, \mathbb{Z})\cong \mathrm{H}_1^{\sigma}(\hat{\Sigma}_{\hat{\mathcal{P}}\cup \hat{B}}; \mathbb{Z})$  correspond to taking the lift in $\Hex$ of $\alpha$ (which is placed relatively to $b$ according to $(\varepsilon,\varepsilon')$), together with an orientation (that depends on the leaf labelling) that makes it $\sigma$ anti-invariant.    
 
  \begin{figure}[!h] 
\centerline{\includegraphics[width=12cm]{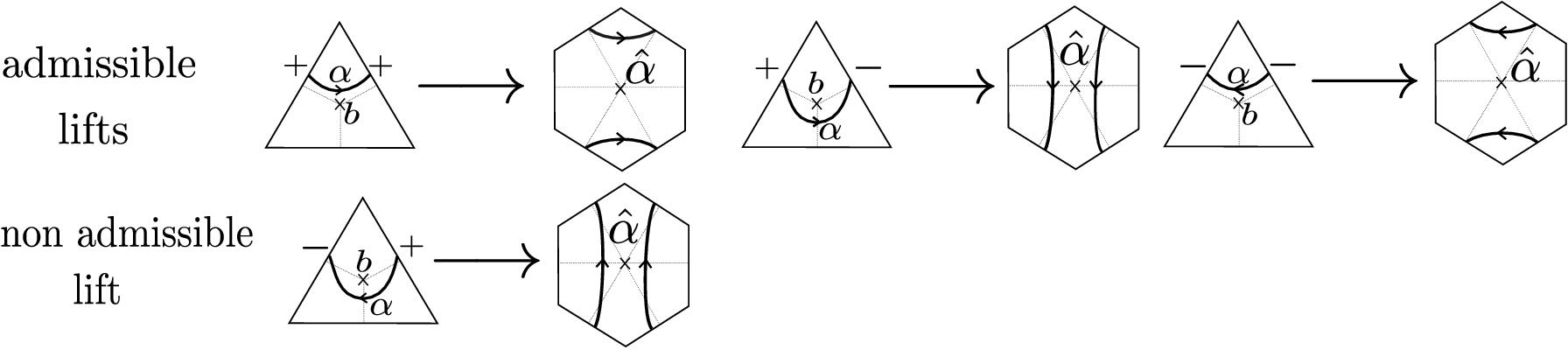} }
\caption{An arc in the triangle can be placed in two different manners relatively to the branched point and each has two different orientations. We draw the four possible arcs and their lifts. A lift is non admissible if it contains the pair of arcs on the bottom of the figure. } 
\label{figliftarc} 
\end{figure} 
 
 Given $\hat{\gamma} \in \mathrm{H}_1(\hat{\Sigma}_{\hat{\mathcal{P}} \cup \hat{B}} ; \mathbb{Z})$, denote by $f_{\hat{\gamma}}$ the associated regular function in $\mathcal{O}[\mathcal{X}_{\mathbb{C}^*}^{\sigma}(\hat{\mathbf{\Sigma}})]$. 
\begin{lemma}\label{lem:descr na}
	For each curve $\gamma$ in $\Sigma_{\mathcal{P}}$, one has 
	\begin{equation*}
 \mathcal{NA}^* (\tau_{\gamma}) = (-1)^{1+w_S([\gamma])} \sum_{[\hat{\gamma}] \in \mathrm{Lt}^a(\gamma)} f_{[\hat{\gamma}]}.	
	\end{equation*}  
For a representation $\rho$, if there exists  $\rho^{ab}: \mathrm{H}_1(\hat{\Sigma}_{\hat{\mathcal{P}} \cup \hat{B}} ; \mathbb{Z}) \rightarrow \mathbb{C}^*$ such that $[\rho]=\mathcal{NA}([\rho^{ab}])$,  then
\begin{equation*}
\tr (\rho( [\gamma])) = (-1)^{1+w_S([\gamma])} \sum _{[\hat{\gamma}] \in \mathrm{Lt}^a(\gamma)} \rho^{ab} ([\hat{\gamma}]).
\end{equation*} 
\end{lemma} 
\begin{proof} This is an immediate consequence of the definitions. \end{proof}

 \subsection{Colored non-abelianization} 

From now on, $\mathbf{\Sigma}=(\Sigma,\mathcal{P})$ is a closed connected punctured surface with at least one puncture (\ie $\mathcal{P}\neq \emptyset$). 

We define the relative version of character varieties of  $\mathbf{\Sigma}$ associated to a coloring $c\co\mathcal{P} \to \C$. 
We show that the non-abelianization map induces a birational map on these varieties whenever the coloring is \emph{generic}.  
\begin{definition}
	A map $c\co \mathcal{P}\rightarrow \mathbb{C}$ is called \textit{generic} if $c(p)\neq \pm 2$ for all $p\in \mathcal{P}$. 	
\end{definition} 
For each puncture $p\in \mathcal{P}$ denote by $\gamma_p \subset \Sigma_{\mathcal{P}}$ a peripheral curve around $p$, that is, a simple closed curve that bounds a disc in $\Sigma$ whose intersection with $\mathcal{P}$ is the singleton $\{p\}$. 
Fix a point $v\in \Sigma_{\mathcal{P}}$ and, for a map  $c: \mathcal{P}\rightarrow \mathbb{C}$, let 
\begin{align*}
\mathcal{R}_{\SL_2}(\mathbf{\Sigma}, c) := \{ \rho \in \Hom(\pi_1(\Sigma_{\mathcal{P}}, v), \SL_2(\mathbb{C})) | \tr (\rho(\gamma_p))=c(p)\mbox{ for all }p\in \mathcal{P}\} \\
\mathcal{R}_{\PSL_2}(\mathbf{\Sigma}, c) := \{ \rho \in \Hom(\pi_1(\Sigma_{\mathcal{P}}, v), \PSL_2(\mathbb{C})) | \left(\tr (\rho(\gamma_p))\right)^2=c(p)^2\mbox{ for all }p\in \mathcal{P}\}.
\end{align*}
For $G$ in $\{\SL_2(\mathbb{C}), \PSL_2(\mathbb{C})\}$, the variety $\mathcal{R}_{G}(\mathbf{\Sigma}, c)$ is acted on algebraically by $G$ by  conjugation.  Note that this action fails to be proper. 
\begin{definition}
	The spaces   
	\begin{equation*} 
	\mathcal{X}_{G}(\mathbf{\Sigma}, c):= {\mathcal{R}_{G}(\mathbf{\Sigma}, c)} \sslash G \quad
	\text{ and } \quad  \mathcal{M}_{G}(\mathbf{\Sigma}, c):= \mathcal{R}_{G}(\mathbf{\Sigma}, c)/ G
	\end{equation*}
	are called the \emph{relative $G$ character variety} and the \emph{relative $G$ moduli space} respectively. 
\end{definition} The former is a GIT quotient and is an affine (singular) sub-variety of $\mathcal{X}_G(\mathbf{\Sigma})$ while the latter is a classical quotient in canonical bijection with the set of isomorphic classes of flat $G$-structures on the surface through the holonomy map (Riemann-Hilbert correspondance).  
There is a surjective map: 
\begin{equation*}
 p : \mathcal{M}_G(\mathbf{\Sigma}, c) \twoheadrightarrow \mathcal{X}_G(\mathbf{\Sigma}, c).
\end{equation*}
The Zariski open subset $ \mathcal{X}^0_G(\mathbf{\Sigma}, c)\subset  \mathcal{X}_G(\mathbf{\Sigma}, c)$ of smooth points corresponds to the classes of irreducible representations in $\Hom(\pi_1(\Sigma_{\mathcal{P}}, v), G)$ (see \textit{e.g.} \cite[Section $3.2$]{MarcheCours09}). 
Writing $\mathcal{M}^0_G(\mathbf{\Sigma}, c):= p^{-1}( \mathcal{X}^0_G(\mathbf{\Sigma}, c))$, the restriction $p:\mathcal{M}^0_G(\mathbf{\Sigma}, c) \rightarrow   \mathcal{X}^0_G(\mathbf{\Sigma}, c)$ is a bijection.

Let $\Delta$ be a topological triangulation of $\mathbf{\Sigma}$ and fix a leaf labelling of $\hat{\mathbf{\Sigma}}$. 

\begin{notation}\label{not:lifts}
	For $p\in \mathcal{P}$, let $\hat{p}_1$ and $\hat{p}_2$ its two lifts, with the convention that $\hat{p}_i$ belongs to the leaf labeled by $i$. \\
	For $c: \mathcal{P}\rightarrow \mathbb{C}$, we let $\hat{c} : \hat{\mathcal{P}} \rightarrow \mathbb{C}^*$ be such that  $c(p)= \hat{c}(\hat{p}_1) + \hat{c}(\hat{p}_2)$ for any $p\in \mathcal{P}$; we call it a \emph{lift} of $c$.  \\
	For $p\in \mathcal{P}$, denote by $\hat{\gamma}_{p_1}, \hat{\gamma}_{p_2}\subset \hat{\Sigma}_{\hat{\mathcal{P}}\cup \hat{B}}$ the two peripheral curves around $\hat{p}_1$ and $\hat{p}_2$ such that $\hat{\gamma}_{p_1}$ is oriented in the counter-clockwise direction whereas $\hat{\gamma}_{p_2}$ is oriented in the clockwise direction. 
\end{notation}
\begin{definition}
The \emph{equivariant relative $\C^*$ character variety} is: 
\begin{equation*}
\mathcal{X}_{\mathbb{C}^*}^{\sigma}(\hat{\mathbf{\Sigma}}, \hat{c}) :=\{ \rho^{ab} \in \mathrm{H}^1(\hat{\Sigma}_{\hat{\mathcal{P}}\cup \hat{B}}, \mathbb{C}^*) | \rho^{ab}([\hat{\gamma}_{p_i}])= \hat{c}(\hat{p}_i)\mbox{, for all } \hat{p_i}\in \hat{\mathcal{P}} \}.
\end{equation*}	
\end{definition}
Note that $\mathcal{O}[\mathcal{X}_{\mathbb{C}^*}^{\sigma}(\hat{\mathbf{\Sigma}}, \hat{c})]$ is the quotient of the algebra $\mathcal{O}[\mathcal{X}_{\mathbb{C}^*}^{\sigma}(\hat{\mathbf{\Sigma}})]$ by the ideal $\mathcal{I}_2$ generated by the elements $[\hat{\gamma}_{p_1} + \hat{\gamma}_{p_2}] - (\hat{c}(\hat{p}_1)+\hat{c}(\hat{p}_2))$ 
for $p\in \mathcal{P}$. 
On the other hand, $\mathcal{O}[\mathcal{X}_{\SL_2}(\mathbf{\Sigma}, c)]$ is the quotient of the algebra $\mathcal{O}[\mathcal{X}_{\SL_2}(\mathbf{\Sigma})]$ by the ideal $\mathcal{I}_1$ generated by the elements $\tau_{\gamma_p} - c(p)$ for $p\in\mathcal{P}$.  
\begin{lemma}
	One has $\mathcal{NA}^*(\mathcal{I}_1)= \mathcal{I}_2$. 
\end{lemma}

\begin{proof}
	Recall that  $\Psi^S ([\gamma_p]) = (-1)^{1+w_S([\gamma_p])} \tau_{\gamma_p}$ and that  we have chosen a spin structure $S$ on $\Sigma_{\mathcal{P}}$ such that $w_S([\gamma_p])=1$.
	It follows from Equation \eqref{eq_QTrace} that the class $[\gamma_p] \in \mathcal{S}_{+1}^{\SL_2}(\mathbf{\Sigma})$ is sent by $\tr_{+1}^{\Delta}$ to the element $H_p+H_p^{-1}\in \mathcal{Z}_{+1}(\mathbf{\Sigma}, \Delta)$ defined in \eqref{eq: Hp}. In turn, $\Phi_{\ell}(H_p+H_p^{-1})\in \mathcal{S}_{+1}^{\mathbb{C}^*, \sigma}(\hat{\mathbf{\Sigma}})$ is sent by $(h^{\sigma})^{-1}$ to the class $[\hat{\gamma}_{\hat{p}_1} + \hat{\gamma}_{\hat{p}_2}] \in \mathrm{H}_1^{\sigma}(\hat{\Sigma}_{\hat{\mathcal{P}}\cup \hat{B}} ; \mathbb{Z})$. 
\end{proof}

 Thus, by passing to the quotient, we obtain an injective map (still denoted by the same symbol):
$$ \mathcal{NA}^* :  \mathcal{O}[\mathcal{X}_{\SL_2}(\mathbf{\Sigma}, c)] \hookrightarrow \mathcal{O}[\mathcal{X}_{\mathbb{C}^*}^{\sigma}(\hat{\mathbf{\Sigma}}, \hat{c})].  $$
\begin{definition} The \textit{colored non-abelianization} is the regular map $\mathcal{NA}: \mathcal{X}_{\mathbb{C}^*}^{\sigma}(\hat{\mathbf{\Sigma}}, \hat{c}) \rightarrow \mathcal{X}_{\SL_2}(\mathbf{\Sigma}, c)$ induced by $\mathcal{NA}^*$. 
\end{definition}
Since $\mathcal{NA}^*$ is injective, and both $\mathcal{X}_{\SL_2}(\mathbf{\Sigma}, c)$ and $\mathcal{X}_{\mathbb{C}^*}^{\sigma}(\hat{\mathbf{\Sigma}}, \hat{c})$ are irreducible, the image of $\mathcal{NA}$ is a Zariski open subset. 
Let $\mathcal{U}_2\subset \mathcal{X}_{\SL_2}(\mathbf{\Sigma}, c)$ be the intersection of the image of $\mathcal{NA}$ with the smooth part $ \mathcal{X}^0_{\SL_2}(\mathbf{\Sigma}, c)$ of $ \mathcal{X}_{\SL_2}(\mathbf{\Sigma}, c)$. 
Let $\mathcal{U}_1\subset \mathcal{X}_{\mathbb{C}^*}^{\sigma}(\hat{\mathbf{\Sigma}}, \hat{c}) $ be the Zariski open subset $\mathcal{U}_1:= \mathcal{NA}^{-1}(\mathcal{U}_2)$. 

\begin{theorem}\label{theorem3_V2}
	If $c$ is generic, the restriction $\mathcal{NA}_{| \mathcal{U}_1} : \mathcal{U}_1 \rightarrow \mathcal{U}_2$ is an isomorphism.
\end{theorem}

Theorem \ref{theorem3} follows from Theorem \ref{theorem3_V2} and \cite[Corollary $I.4.5$]{Hart}.

The strategy to prove Theorem \ref{theorem3_V2} is the following. The map $\mathcal{NA}_{| \mathcal{U}_1}$ is surjective by definition. 
In Section \ref{sec: action gp free}, we will show that the group $\mathrm{H}^1(\Sigma; \mathbb{Z}/2\mathbb{Z})$ acts freely on both $\mathcal{U}_1$ and $\mathcal{U}_2$ and that $\mathcal{NA}_{| \mathcal{U}_1}$ is equivariant for these actions. 
Therefore, $\mathcal{NA}_{| \mathcal{U}_1}$ induces a quotient map $\widetilde{\mathcal{NA}}_{\mathcal{V}_1} : \mathcal{V}_1 \rightarrow \mathcal{V}_2$, where $\mathcal{V}_i = \mathcal{U}_i/\mathrm{H}^1(\Sigma; \mathbb{Z}/2\mathbb{Z})$.   
The map $\mathcal{NA}_{| \mathcal{U}_1}$  is injective if and only if  $\widetilde{\mathcal{NA}}_{\mathcal{V}_1}$ so is. 
In Section \ref{sec: SB} we will identify $\widetilde{\mathcal{NA}}_{\mathcal{V}_1}$ with the \emph{shear-bend parametrization} (see Proposition \ref{prop_sb}) which is known to be injective (\cite{BonahonLiu}).

 \subsubsection{Actions of the group $\mathrm{H}^1(\Sigma_{\mathcal{P}}; \mathbb{Z}/2\mathbb{Z})$}\label{sec: action gp free}
 
  In this subsection, we define free actions of the group $\mathrm{H}^1(\Sigma_{\mathcal{P}}; \mathbb{Z}/2\mathbb{Z})$ on the varieties $ \mathcal{X}_{\mathbb{C}^*}^{\sigma}(\hat{\mathbf{\Sigma}})$ and $ \mathcal{X}_{\SL_2}(\mathbf{\Sigma})$ and show that the non-abelianization map intertwines  these actions.

 	Let 
  \begin{equation*} 
   \nabla_1 : \mathcal{X}_{\SL_2}(\mathbf{\Sigma})\times \mathrm{H}^1(\Sigma_{\mathcal{P}}; \mathbb{Z}/2\mathbb{Z})\rightarrow \mathcal{X}_{\SL_2}(\mathbf{\Sigma})
  \end{equation*}
 be defined by $\tau_{\gamma}([\rho]\cdot {\chi}) = (-1)^{\chi([\gamma])}\tau_{\gamma}([\rho])$, 
 for $[\rho]\in \mathcal{X}_{\SL_2}(\mathbf{\Sigma})$ and $\chi\in  \mathrm{H}^1(\Sigma_{\mathcal{P}}; \mathbb{Z}/2\mathbb{Z})$, and every curve $\gamma$. 
 Note that this action is free. It is proved in \cite{GoldmanTopoComponentsRepSpaces} that 
  \begin{equation*}
 \mathcal{X}_{\PSL_2}(\mathbf{\Sigma}) \cong \quotient{\mathcal{X}_{\SL_2}(\mathbf{\Sigma})}{\mathrm{H}^1(\Sigma_{\mathcal{P}}; \mathbb{Z}/2\mathbb{Z})}. 
 \end{equation*} 
  We denote by 
  $\pi_1: \mathcal{X}_{\SL_2}(\mathbf{\Sigma}) \rightarrow \mathcal{X}_{\PSL_2}(\mathbf{\Sigma})$ the quotient map. In algebraic terms, the group action $\nabla_1$ is induced by a co-action:

 \begin{equation*}
 \Delta_1 : \mathcal{O}[\mathcal{X}_{\SL_2}(\mathbf{\Sigma})] \rightarrow \mathcal{O}[\mathcal{X}_{\SL_2}(\mathbf{\Sigma})] \otimes \C[\mathrm{H}_1(\Sigma_{\mathcal{P}}; \mathbb{Z}/2\mathbb{Z})],
 \end{equation*} 
 given by the formula $\Delta_1(\tau_{\gamma}):= \tau_{\gamma} \otimes [\gamma]$ for every curve $\gamma$. So the algebra $\mathcal{O}[\mathcal{X}_{\PSL_2}(\mathbf{\Sigma})]$ is isomorphic to the subspace of co-invariant vectors for this co-action, that is to the kernel of
 
 \begin{equation*} 
 f_1:= \Delta_1 - \id \otimes \eta : \mathcal{O}[\mathcal{X}_{\SL_2}(\mathbf{\Sigma})] \to \mathcal{O}[\mathcal{X}_{\SL_2}(\mathbf{\Sigma})] \otimes \C[\mathrm{H}_1(\Sigma_{\mathcal{P}}; \mathbb{Z}/2\mathbb{Z})], 
 \end{equation*}
 where $\eta : \mathbb{C} \rightarrow \C[\mathrm{H}_1(\Sigma_{\mathcal{P}}; \mathbb{Z}/2\mathbb{Z})]$ is the unit map sending $1$ to the neutral element of $\mathrm{H}_1(\Sigma_{\mathcal{P}}; \mathbb{Z}/2\mathbb{Z})$.

\vspace{2mm}
\par Let 
 \begin{equation*}
  \nabla_2 : \mathcal{X}_{\mathbb{C}^*}^{\sigma}(\mathbf{\hat{\Sigma}})\times \mathrm{H}^1(\Sigma_{\mathcal{P}}; \mathbb{Z}/2\mathbb{Z}) \rightarrow \mathcal{X}_{\mathbb{C}^*}^{\sigma}(\mathbf{\hat{\Sigma}})
 \end{equation*}
 be the free action defined by $f_{[\hat{\gamma}]}(\rho^{ab}\cdot \chi) := (-1)^{\chi(\pi_*[\hat{\gamma}])} f_{\hat{\gamma}}(\rho^{ab})$, where $\pi: \hat{\Sigma}\setminus (\hat{\mathcal{P}}\cup \hat{B}) \to \Sigma_{\mathcal{P}\cup B} \hookrightarrow \Sigma_{\mathcal{P}}$. Algebraically, this group action is induced by the co-action 
 
  \begin{equation*}
 \Delta_2: \mathcal{O}[\mathcal{X}_{\mathbb{C}^*}^{\sigma}(\mathbf{\hat{\Sigma}})] \rightarrow \mathcal{O}[\mathcal{X}_{\mathbb{C}^*}^{\sigma}(\mathbf{\hat{\Sigma}})] 
  \otimes \C[\mathrm{H}_1(\Sigma_{\mathcal{P}}; \mathbb{Z}/2\mathbb{Z})],
 \end{equation*}
 defined by $\Delta_2(f_{[\hat{\gamma}]}) = f_{[\hat{\gamma}]} \otimes [\pi_*[\hat{\gamma}]]$. The algebra of regular functions of the (algebraic) quotient of $\mathcal{X}_{\mathbb{C}^*}^{\sigma}(\mathbf{\hat{\Sigma}})$ by $\mathrm{H}_1(\Sigma_{\mathcal{P}}; \mathbb{Z}/2\mathbb{Z})$ is given by the kernel of
 
 \begin{equation*}
  f_2:= \Delta_2 - \id \otimes \eta : \mathcal{O}[\mathcal{X}_{\mathbb{C}^*}^{\sigma}(\hat{\mathbf{\Sigma}})] \rightarrow 
  \mathcal{O}[\mathcal{X}_{\mathbb{C}^*}^{\sigma}(\hat{\mathbf{\Sigma}})]  \otimes  \C[\mathrm{H}_1(\Sigma_{\mathcal{P}}; \mathbb{Z}/2\mathbb{Z})],  
 \end{equation*}
 where again $\eta$ is the unit map. 
  Recall the isomorphism $\varphi : K_{\Delta} \xrightarrow{\cong} \mathrm{H}_1^{\sigma}(\hat{\Sigma}_{\hat{\mathcal{P}}\cup \hat{B}}; \mathbb{Z})$ of Equation \eqref{eq:iso CF Sk}.

  For each edge $e\in \mathcal{E}(\Delta)$, let 
  \begin{equation*}
  \gamma_e:= \varphi(Z_e^2). 
  \end{equation*}
 The element $\gamma_e$ in depicted in Figure \ref{fig_edgescurves}. 
 
  \begin{figure}[!h] 
\centerline{\includegraphics[width=5cm]{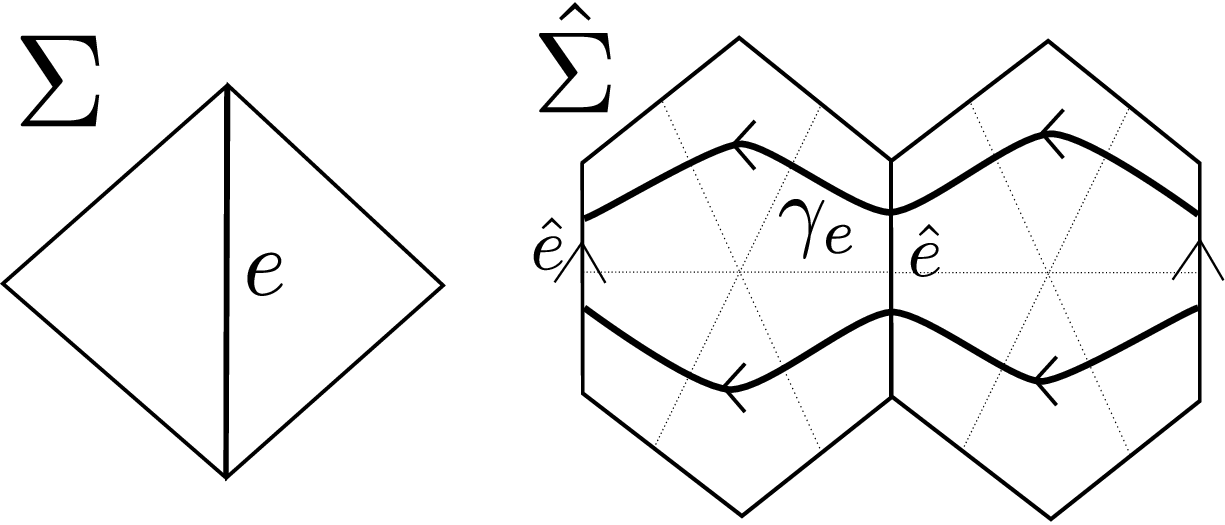} }
\caption{The multi-curve $\gamma_e$ in the double branched covering whose holonomy gives the shear-bend coordinate. } 
\label{fig_edgescurves} 
\end{figure} 
 
The kernel of $f_2$ is the subalgebra of $\mathbb{C}[\mathrm{H}_1^{\sigma}(\hat{\Sigma}_{\hat{\mathcal{P}}\cup \hat{B}}; \mathbb{Z})]$ generated by those $f_{[\hat{\gamma}]}$ such that the homology class $\pi_*([\hat{\gamma}])=0$ vanishes in $\mathrm{H}_1(\Sigma_{\mathcal{P}}; \mathbb{Z}/2\mathbb{Z})$. The set of those classes $[\hat{\gamma}]$ are sent by $\varphi^{-1}$ to the subgroup of balanced monomials $[Z_{e_1}^{k_1} \ldots Z_{e_n}^{k_n}]\in K_{\Delta}$ such that $k_e$ is even for every $e\in \mathcal{E}(\Delta)$. This group is obviously generated by the elements $Z_e^{\pm 2}$, therefore 
 the kernel of $f_2$, is the polynomial subalgebra $\mathbb{C}[f_{\gamma_e}^{\pm 1}, e\in \mathcal{E}(\Delta)]\subset \mathcal{O}[\mathcal{X}_{\mathbb{C}^*}^{\sigma}(\hat{\mathbf{\Sigma}})]  $ generated by the regular functions associated to the  elements $\gamma_e$ and their inverses. 
 So we have an isomorphism 
 \begin{equation*}
 (\mathbb{C}^*)^{\mathcal{E}(\Delta)}\cong \quotient{\mathcal{X}_{\mathbb{C}^*}^{\sigma}(\hat{\mathbf{\Sigma}})}{\mathrm{H}^1(\Sigma_{\mathcal{P}}; \mathbb{Z}/2\mathbb{Z})}. 
 \end{equation*}
 We denote by $\pi_2 : \mathcal{X}_{\mathbb{C}^*}^{\sigma}(\hat{\mathbf{\Sigma}}) \rightarrow (\mathbb{C}^*)^{\mathcal{E}(\Delta)}$ the quotient map.
  \\
  
 \begin{lemma}
 	One has 
 	$\mathcal{NA} \circ \nabla_2= \nabla_1\circ(\mathcal{NA} \otimes \id)$. 
 \end{lemma} 
\begin{proof}
By Lemma \ref{lem:descr na}, we know that $\mathcal{NA}$ sends a curve function $\tau_{\gamma}$, for  $\gamma \subset \Sigma_{\mathcal{P}}$, to plus or minus a sum of functions of admissible lifts $[\hat{\gamma}]$  in $\hat{\Sigma}$. Since $\pi(\hat{\gamma})=\gamma$ for each $\hat{\gamma} \in \mathrm{Lt}^a(\gamma)$, the map  $\mathcal{NA}^*$ intertwines the co-actions of $\mathrm{H}_1(\Sigma_{\mathcal{P}}; \mathbb{Z}/2\mathbb{Z})$ in the sense that $\Delta_2\circ \mathcal{NA}^* = \left(\mathcal{\na}^* \otimes \id \right)\circ \Delta_1$.  
\end{proof}
  Therefore, the non-abelianization map induces a surjective regular map 
  \begin{equation*}
\widetilde{\na} \co (\mathbb{C}^*)^{\mathcal{E}(\Delta)} \rightarrow \mathcal{X}_{\PSL_2}(\mathbf{\Sigma})
  \end{equation*} 
 such that the following diagram commutes:
\begin{equation*}
\begin{tikzcd}
\mathcal{X}_{\mathbb{C}^*}^{\sigma}(\hat{\mathbf{\Sigma}})
\arrow[r, "\na"] \arrow[d, "\pi_2"]
&  \mathcal{X}_{\SL_2}(\mathbf{\Sigma})
\arrow[d, "\pi_1"]
\\ (\mathbb{C}^*)^{\mathcal{E}(\Delta)}
\arrow[r, "\widetilde{\na}"]
&  \mathcal{X}_{\PSL_2}(\mathbf{\Sigma})
\end{tikzcd}
\end{equation*} 
  More precisely, the map $\widetilde{\na}$ is defined by the unique injective algebra morphism $\widetilde{\na}^*$ that makes the following diagram commute:
  \begin{equation*}
  \begin{tikzcd}
  0 \arrow[r, ""] & \mathcal{O}[ \mathcal{X}_{\PSL_2}(\mathbf{\Sigma})] \arrow[r, ""] \arrow[d, hook, "\widetilde{\na}^*"] & 
  \mathcal{O}[\mathcal{X}_{\SL_2}(\mathbf{\Sigma})] \arrow[r, "f_1"] \arrow[d, hook, "\na^*"] &  \mathcal{O}[\mathcal{X}_{\SL_2}(\mathbf{\Sigma})]  \otimes \C[\mathrm{H}_1(\Sigma_{\mathcal{P}}; \mathbb{Z}/2\mathbb{Z})] \arrow[d, hook,  "\na^*\otimes \id"] 
  \\
  0 \arrow[r, ""] &\mathbb{C}[f_{\gamma_e}^{\pm 1}, e\in \mathcal{E}(\Delta)] \arrow[r, ""] &  \mathcal{O}[\mathcal{X}_{\mathbb{C}^*}^{\sigma}(\hat{\mathbf{\Sigma}})] \arrow[r, "f_2"] & 
  \mathcal{O}[\mathcal{X}_{\mathbb{C}^*}^{\sigma}(\hat{\mathbf{\Sigma}})] \otimes  \C[\mathrm{H}_1(\Sigma_{\mathcal{P}}; \mathbb{Z}/2\mathbb{Z})]  
  \end{tikzcd}
   \end{equation*}

  	Let   $\mathrm{H}^1(\Sigma; \mathbb{Z}/2\mathbb{Z}) \hookrightarrow \mathrm{H}^1(\Sigma_{\mathcal{P}}; \mathbb{Z}/2\mathbb{Z})$ be the inclusion whose image are those elements $\chi \in  \mathrm{H}^1(\Sigma_{\mathcal{P}}; \mathbb{Z}/2\mathbb{Z})$  such that $\chi([\gamma_p])= 0 \pmod{2}$ for all $p\in \mathcal{P}$. 
  	
  	The actions $\nabla_1$ and $\nabla_2$, when restricted to $\mathrm{H}^1(\Sigma; \mathbb{Z}/2\mathbb{Z})$, preserve the subvarieties $\mathcal{X}_{\SL_2}(\mathbf{\Sigma}, c)$ and $\mathcal{X}_{\mathbb{C}^*}^{\sigma}(\hat{\mathbf{\Sigma}}, \hat{c})$. Moreover, the colored non-abelianization map is equivariant for these actions. 
  	
  	On the other hand, since the subvariety $\mathcal{X}^0_{\SL_2}(\mathbf{\Sigma}, c) \subset \mathcal{X}_{\SL_2}(\mathbf{\Sigma}, c)$ of smooth points consists of the classes of irreducible representations, it is preserved by the action of $\mathrm{H}^1(\Sigma; \mathbb{Z}/2\mathbb{Z})$. Therefore, the map $\mathcal{NA}_{| \mathcal{U}_1}$  induces, by passing to the quotient through $\mathrm{H}^1(\Sigma; \mathbb{Z}/2\mathbb{Z})$, a surjective map 
  	\begin{equation*}
  	\widetilde{\mathcal{NA}}_{| \mathcal{V}_1} : \mathcal{V}_1 \rightarrow \mathcal{V}_2,
  	\end{equation*} 
  	where $\mathcal{V}_i = \quotient{ \mathcal{U}_i}{\mathrm{H}^1(\Sigma; \mathbb{Z}/2\mathbb{Z})}$.

\subsubsection{Shear-bend coordinates}\label{sec: SB}

 The goal of this subsection is to relate the map $\widetilde{\mathcal{NA}}_{| \mathcal{V}_1}$ to the shear-bend parametrization defined by Bonahon and Liu in \cite{BonahonLiu}.  

 From now on, the coloring $c\co \mathcal{P}\to \C$ is supposed to be generic. 
 \\
 
  Consider a representation $\rho\in \mathcal{R}_{\PSL_2}(\mathbf{\Sigma})$. 
  Recall the classical action of  $\SL_2(\mathbb{C})$ on $\C^2$; in particular, any lift of $\rho(\gamma_p) \in \PSL_2(\mathbb{C})$ in  $\SL_2(\mathbb{C})$ acts on $\C^2$. 
  \begin{definition}
 	  An \textit{enhancement} $E=(E_p)_{p\in \mathcal{P}}$ of $\rho$ is the choice, for each puncture $p\in \mathcal{P}$, of an axis $E_p\in \mathbb{CP}^1$ that is invariant under the action of an arbitrary lift of  $\rho(\gamma_p) \in \PSL_2(\mathbb{C})$. 
 	  The set of the enhanced representations is acted on by $\PSL_2(\mathbb{C})$  via $g\cdot \left(\rho, (E_p)_p\right) := (g\cdot \rho \cdot g^{-1}, (g\cdot E_p)_p)$ for each $g\in \PSL_2(\mathbb{C})$; the resulting quotient is denoted by $\mathcal{M}_{\PSL_2}^e(\mathbf{\Sigma})$. 
 	  Let $pr:\mathcal{M}_{\PSL_2}^e(\mathbf{\Sigma}) \rightarrow \mathcal{M}_{\PSL_2}(\mathbf{\Sigma})$ be the surjection that sends a class $[\rho, E]$ to the class $[\rho]$. 
 	  An \textit{enhancement} of $[\rho]$ is the choice of a lift $[\rho, E]$ through $pr$. 
  \end{definition}	  

Recall Notation \ref{not:lifts}. 
Note that if $\tr(\rho(\gamma_p))\neq \pm 2$ for all $p\in \mathcal{P}$, then $[\rho]$ admits exactly $2^{|\mathcal{P}|}$ enhancements. 
In this case, one has a bijection, which depends on the leaf labeling, between the set of lifts  of the generic coloring $c$ and  the set of enhancements of $[\rho]$.  It is as follows. 

For $p\in \mathcal{P}$, let $g_p\in \SL_2(\mathbb{C})$ be a lift of $\rho(\gamma_p)\in \PSL_2(\mathbb{C})$. 
The axes $E_p$ of an enhancement $E=(E_p)_{p\in \mathcal{P}}$ of $\rho$  are eigenspaces of $g_p$. 
 For each puncture $p$, there are two such spaces; they correspond to the two eigenvalues $x_p$ and $x_p^{-1}$ of $g_p$ such that $\tr (g_p)= x_p+x_p^{-1}$. 
To an enhancement $E=(E_p)_{p\in \mathcal{P}}$ such that $E_p$ corresponds to $x_p$, one associates the lift $\hat{c}$ such that  $\hat{c}(\hat{p}_1)=x_p$ (therefore one has $\hat{c}(\hat{p}_2)=x_p^{-1}$). 
Conversely, to each lift $\hat{c}$, one associate the enhancement $E_{\hat{c}}$ such that its axis at $p$ is the eigenspace corresponding to $\hat{c}(\hat{p}_1)$.

\vspace{2mm}
\par Following Thurston \cite{ThurstonShear, ThurstonLectNotes} and Bonahon \cite{BonahonShearBend,BonahonLivre}, Bonahon and Liu defined in \cite[Section $7$]{BonahonLiu} an injective map:
$$ \SB : (\mathbb{C}^*)^{E(\Delta)} \hookrightarrow \mathcal{M}_{\PSL_2}^e(\mathbf{\Sigma})$$
called the \textit{shear-bend parametrization}.  The inverse map $\SB^{-1}: \text{Im}(\SB) \xrightarrow{\cong} (\mathbb{C}^*)^{E(\Delta)}$ sends a class $[\rho, E]$ to complex numbers $x_e([\rho, E])$ called \textit{the shear-bend coordinates} of $[\rho, E]$. We briefly sketch the construction of $\SB$. 
 
\vspace{2mm}
\par A \textit{pleated surface} is a pair $(\widetilde{f}, \rho)$ where $\rho \in \mathcal{R}_{\PSL_2}(\mathbf{\Sigma})$ and $\widetilde{f}: \widetilde{\Sigma}_{\mathcal{P}}\rightarrow \mathbb{H}^3$ is a continuous map from a universal cover $\widetilde{\Sigma}_{\mathcal{P}}$ of $\Sigma_{\mathcal{P}}$ to the hyperbolic upper half-space $\mathbb{H}^3$ such that: 
\begin{enumerate}
\item If $\widetilde{\Delta}$ is a lift of $\Delta$ in $\widetilde{\Sigma}$ and $\widetilde{e}\in \mathcal{E}(\widetilde{\Delta})$ is the lift of an edge, then $\widetilde{f}(\widetilde{e})$ is a geodesic.
\item If $\widetilde{\mathbb{T}}\in F(\widetilde{\Delta})$ is a lift of a triangle of $\Delta$, then the closure of $\widetilde{\mathbb{T}}$ is an ideal triangle in $\overline{\mathbb{H}^3}=\mathbb{H}^3\cup \mathbb{CP}^1$.
\item The map $\widetilde{f}$ is $\rho$-equivariant.
\end{enumerate}
\par Two pleated surfaces $(\widetilde{f}, \rho)$ and $(\widetilde{f}', \rho')$ are \textit{isometric} if there exist an element $A\in \PSL_2(\mathbb{C})$ and a lift $\widetilde{\phi}:\widetilde{\Sigma}_{\mathcal{P}} \rightarrow \widetilde{\Sigma}_{\mathcal{P}}$ of an isometry of $\Sigma_{\mathcal{P}}$, such that $\widetilde{f}'=A\circ \widetilde{f}\circ \widetilde{\phi}$ and $\rho'=A\rho A^{-1}$. We denote by $\PS_{\Delta}$ the set of isometry classes of pleated surfaces. A pleated surface $(\widetilde{f}, \rho)$ naturally defines an enhanced representation $(\rho, E)$: given a puncture $p\in \mathcal{P}$, the closures of the images through $\widetilde{f}$ of the lifts of the triangles adjacent to $p$ intersect in a single point $E_p\in  \mathbb{CP}^1=\partial \overline{\mathbb{H}^3}$ invariant under $\rho$ which defines the enhancement of $\rho$ at $p$.
 This construction defines an injective map:
$$ j: \PS_{\Delta} \hookrightarrow \mathcal{M}_{\PSL_2}^e(\mathbf{\Sigma}).$$
 Moreover, in \cite[Proposition $31$]{BonahonLiu},  the authors defined a bijection:
 $$ \theta : (\mathbb{C}^*)^{\mathcal{E}(\Delta)} \cong \PS_{\Delta}.$$
 The reverse map $\theta^{-1}$ sends a pleated surface $(\widetilde{f}, \rho)$ to nonzero complex numbers $x_e(\widetilde{f}, \rho)$, for each edge, $e\in \mathcal{E}(\Delta)$ as follows. Orient the edge $e$ arbitrarily and denote by $\mathbb{T}_L$ and $\mathbb{T}_R$ the triangles on the left and right respectively of $e$. Choose an arbitrary lift $\widetilde{e}\in E(\widetilde{\Delta})$. The closures of the corresponding lifts of $\mathbb{T}_L$ and $\mathbb{T}_R$ are sent to two ideal triangles $\widetilde{\mathbb{T}_L}, \widetilde{\mathbb{T}_R}$. Denote by $z_-, z_+, z_L, z_R \in \mathbb{CP}^1$ the vertices of the square $\widetilde{\mathbb{T}_L}\cup\widetilde{\mathbb{T}_R}$ where $\widetilde{f}(\widetilde{e})$ is oriented from $z_-$ to $z_+$ and $z_L, z_R$ are the vertices on the left and right respectively of $\widetilde{f}(\widetilde{e})$. We define the (exponential) shear-bend parameter $x_e$ to be the cross ratio:
 $$ x_e([\widetilde{f}, \rho]) := - \frac{(z_L - z_+)(z_r-z_-)}{(z_L-z_-)(z_R-z_+)}\in \mathbb{C}^* $$
 This number is independent of the choices previously made, invariant under isometry and defines the map $\theta^{-1}$. We refer to \cite{BonahonShearBend, BonahonLiu} for the construction of $\theta$. 
 
 \begin{definition} The \textit{shear-bend parametrization} is the composition
 $$
 \begin{tikzcd}
 \SB : (\mathbb{C}^*)^{\mathcal{E}(\Delta)} 
 \arrow[r, "\theta", "\cong"'] &\PS_{\Delta} \arrow[r, hook, "j"] & \mathcal{M}_{\PSL_2}^e(\mathbf{\Sigma}).
 \end{tikzcd}
 $$
 \end{definition}

This map is closely related to the map $\widetilde{\na}$ defined in the previous subsection, as we now explain; also compare with Hollands-Neitzke's geometric non-abelianization map described in \cite[Section $8.3$]{HollandsNeitzke}. 

Let $\hat{c}$ be a lift of $c$. Recall the bijection that associates to each $[\rho]\in \mathcal{M}_{\PSL_2}(\mathbf{\Sigma}, c)$ an enhancement $E_{\hat{c}}$. 
It gives rise to an embedding $\mathcal{M}_{\PSL_2}(\mathbf{\Sigma}, c) \subset \mathcal{M}^e_{\PSL_2}(\mathbf{\Sigma})$ that depends on $\hat{c}$.  

 	Given $[\rho^{\mathbb{C}^*}]\in \mathcal{U}_1\subset \mathcal{X}_{\mathbb{C}^*}^{\sigma}(\hat{\mathbf{\Sigma}}, \hat{c})$ define $[\rho^{\SL_2}]:= \na ([\rho^{\mathbb{C}^*}])\in \mathcal{X}^0_{\SL_2}(\mathbf{\Sigma},c)$ and $[\rho^{\PSL_2}]:=pr([\rho^{\SL_2}])\in \mathcal{X}^0_{\SL_2}(\mathbf{\Sigma},c)$. 
 	
 	 \begin{proposition}\label{prop:sb1}
 	One has $ x_e\left([\rho^{\PSL_2}, E_{\hat{c}}]\right) = f_{\gamma_e}([\rho^{\mathbb{C}^*}]).$
 	 	\end{proposition}
  	\begin{proof}
  		It is a consequence of \cite[Proposition $15$]{BonahonWong2} which results from the definition of the quantum trace. A point $[\rho^{\mathbb{C}^*}]\in \eqabchar$ is a character of the corresponding algebra of regular functions. This algebra has basis $\mathrm{H}_1^{\sigma}(\hat{\Sigma}_{\hat{\mathcal{P}}\cup\hat{B}} ; \mathbb{Z})$ which is in bijection with  $\mathcal{W}(\tau_{\Delta}, \mathbb{Z})$ by \eqref{eq:bij Theta}.  So $[\rho^{\mathbb{C}^*}]$ is described by a group morphism $\chi : \mathcal{W}(\tau_{\Delta}, \mathbb{Z})\rightarrow \mathbb{C}^*$. By definition, if $\phi_e\in \mathcal{W}(\tau_{\Delta}, \mathbb{Z})$ is the element associated to $Z_e^2 \in \mathcal{Z}_{+1}(\mathbf{\Sigma}, \Delta)$, then $\chi(\phi_e)=f_{\gamma_e}([\rho^{\mathbb{C}^*}])$. The image $[\rho^{\SL_2}]=\na([\rho^{\mathbb{C}^*}]) $ is described by a character of $\mathcal{S}_{+1}(\mathbf{\Sigma})$ obtained from $\chi$ by composing with the quantum trace map.  By \cite[Proposition $15$]{BonahonWong2}  the shear-bend coordinates  $x_e([\rho^{\PSL_2}, E_{\hat{c}}])$ are equal to $\chi(\phi_e)$,  so it proves the assertion. 
  		We emphasize that the quantum trace map was designed so that this equality holds (see \cite[Lemma $4$]{BonahonWongqTrace}).
  	\end{proof}
  \begin{corollary}\label{prop_sb}
  	 The two maps $\widetilde{\na}_{|\mathcal{V}_1} : \mathcal{V}_1 \rightarrow \mathcal{X}^0_{\PSL_2}(\mathbf{\Sigma}, c) $ and $\SB_{|\mathcal{V}_1}: \mathcal{V}_1 \rightarrow \mathcal{X}^0_{\PSL_2}(\mathbf{\Sigma}, c) \subset \widetilde{\mathcal{X}}^e_{\PSL_2}(\mathbf{\Sigma})$ are equal. In particular $\widetilde{\na}_{|\mathcal{V}_1}$ is injective.
  \end{corollary}
 	\begin{proof}
 	 By injectivity of $\SB$, an isometry class of enhanced representation $[\rho^{\PSL_2}, E_{\hat{c}}]\in \text{Im}(\SB)$ is completely determined by its shear-bend coordinates. 
 	 By Proposition \ref{prop:sb1}, if $(x_e)_{e\in \mathcal{E}(\Delta)}\in \mathcal{V}_1$ is a set of coordinates, then both $\widetilde{\na}\left( (x_e)_e \right)$ and $\SB\left( (x_e)_e \right)$ admit $(x_e)_e$ as shear-bend coordinates, so they are equal. 
 	\end{proof}
 
 \begin{proof}[Proof of Theorem \ref{theorem3_V2}]
 	The map $\mathcal{NA}_{| \mathcal{U}_1}: \mathcal{U}_1 \rightarrow \mathcal{U}_2$ is surjective by definition. Since the actions of  $\mathrm{H}^1(\Sigma; \mathbb{Z}/2\mathbb{Z})$, defined in subsection $3.4$, are free, the map $\mathcal{NA}_{| \mathcal{U}_1}$ is injective if and only if the quotient map $\mathcal{\widetilde{NA}}_{| \mathcal{V}_1}: \mathcal{V}_1 \rightarrow \mathcal{V}_2$ is injective. This latter fact is proved in Corollary \ref{prop_sb}.
 \end{proof}

\section{Towards a quantum trace for unpunctured surfaces}
We conclude the paper by formulating a conjecture which naturally derives from our study. The balanced Chekhov-Fock algebra, and thus the quantum trace, only make sense for a punctured surface $(\Sigma, \mathcal{P})$ where $\mathcal{P}\neq \emptyset$, since we need a topological triangulation to define it. However Hollands and Neiztke defined in \cite{HollandsNeitzke}  geometric non-abelianization maps for unpunctured surfaces, hence it is natural to expect the existence of a quantum non-abelianization for such surfaces as well. Let $\Sigma$ be a closed connected surface of genus $g\geq 2$. A \textit{pants decomposition} $P=\{\gamma_e\}_e$ is a maximal set of pairwise non isotopic simple closed curves in $\Sigma$. Once cutting $\Sigma$ along $P$, we obtain a disjoint union $\bigsqcup_v P_v$ of pants $P_v$, that is each $P_v$ is homeomorphic to a sphere with three disjoint discs removed. A pant $P_v$ retracts to an embedded trivalent graph $\Gamma_v\subset P_v$ with two vertices and three edges. Let $\Gamma \subset \Sigma$ be the disjoint union of the $\Gamma_v$ and denote by $V$ its set of vertices and $E$ its set of edges. The set $E$ induces a Borel-Moore class $[E]\in \mathrm{H}_1(\Sigma \setminus V; \mathbb{Z}/2\mathbb{Z})$ whose Poincar\'e-Lefschetz dual defines a regular double covering of $\Sigma\setminus V$. We denote by $\pi: \widehat{\Sigma}(P) \rightarrow \Sigma$ the induced double covering branched along $V$ and denote by $\sigma : \widehat{\Sigma}(P)\rightarrow \widehat{\Sigma}(P)$ the covering involution. By Proposition \ref{propdecomposition}, one has an isomorphism
\begin{equation}\label{eq_pants}
 \mathcal{S}_{\omega}^{\mathbb{C}^*, \sigma}(\mathbf{\widehat{\Sigma}}(P)) \cong \mathcal{W}_q^{\otimes g} \otimes \mathcal{W}_{q^2}^{\otimes 2g-3}, 
\end{equation}
when $\omega$ is a root of unity of odd order $N>1$. We claim that the algebras $\mathcal{S}_{\omega}^{\sigma}(\mathbf{\widehat{\Sigma}}(P))$ could play a similar role, for unpunctured surfaces, to the balanced Chekhov-Fock algebras. More precisely, we formulate the 

\begin{conjecture}\label{conj_pants}
There exists an injective morphism of algebras
$$ \mathrm{NA}^P_{\omega} : \mathcal{S}^{\SL_2}_{\omega}(\Sigma) \hookrightarrow \mathcal{S}_{\omega}^{\mathbb{C}^*, \sigma}(\mathbf{\widehat{\Sigma}}(P)).$$
\end{conjecture}

We now formulate some arguments in favour of our conjecture.
\begin{enumerate}
\item Conjecture \ref{conj_pants} is first motivated by the construction of Hollands and Neitzke of geometric non-abelianization map (called in Fenchel-Nielsen coordinates and associated to a pants decomposition) between $\mathcal{X}_{\mathbb{C}^*}^{\sigma}(\mathbf{\widehat{\Sigma}}(P))$ and a moduli space of framed $\SL_2(\mathbb{C})$ flat structures on $\Sigma$, similar to the non-abelianization of Gaiotto-Moore-Neitzke of Foch-Goncharov coordinates (associated to a triangulation). Since the quantum trace is a deformation of the latter, one might expect that the non-abelianization in Fenchel-Nielsen coordinates admits a quantum deformation as well.
\item Equation \eqref{eq_pants} shows that, when $\omega$ is a root of unity of odd order $N>1$, $\mathcal{S}_{\mathbb{C}^*, \omega}^{\sigma}(\mathbf{\widehat{\Sigma}}(P)) $ is semi-simple and its simple modules have dimension $N^{3g-3}$. This is precisely the PI-degree of $\mathcal{S}^{\SL_2}_{\omega}(\Sigma)$ as computed in \cite{FrohmanKaniaLe_UnicityRep} and is the dimension of the representations of $\mathcal{S}^{\SL_2}_{\omega}(\Sigma)$ defined by Bonahon and Wong in \cite{BonahonWong3}. Note that it is proved in  \cite{FrohmanKaniaLe_UnicityRep} that a generic simple module of $\mathcal{S}^{\SL_2}_{\omega}(\Sigma)$ is isomorphic to one of the representation in \cite{BonahonWong3}. The authors believe that the skein algebra representations in \cite{BonahonWong3} are obtained by composing $\mathrm{NA}^P_{\omega}$ with an irreducible representation of $ \mathcal{S}_{\omega}^{\sigma}(\mathbf{\widehat{\Sigma}}(P))$.
\item Eventually, when $\Sigma_1$ is a genus $1$ closed surface, Frohman and Gelca proved in \cite{FG00} an analogue of Conjecture \ref{conj_pants}. More precisely, they defined an injective morphism of algebras $\mathcal{S}_{\omega}(\Sigma_1) \hookrightarrow \mathcal{W}_q= \mathbb{C}\left< X^{\pm 1}, Y^{\pm 1} | XY=qYX\right>$ sending the class of a meridian to $X+X^{-1}$ and the class of longitude to $Y+Y^{-1}$. Note that $\mathcal{W}_q$ is isomorphic to the equivariant abelian skein algebra of a double covering of $\Sigma_1$ by a genus $2$ surface with two branched points. Note also that the corresponding classical Poisson morphism $(\mathbb{C}^*)^2 \rightarrow \mathcal{X}_{\SL_2}(\Sigma_1)$ is a double branched covering with four branched points. Hence we do not expect that the algebraic non-abelianization induced by $\mathrm{NA}^P_{+1}$  to be birational in the closed case. 
\end{enumerate}

\section*{Acknowledgments} 
The first author is thankful to F.Bonahon, F.Costantino, L.Funar, F.Ruffino, M.Spivakovsky, J.Toulisse,  D.Vendr\'uscolo and J.Viu Sos for useful discussions and to the University of South California and the Federal University of S\~ao Carlos for their kind hospitality during the beginning of this work.  He acknowledges  support from the grant ANR  ModGroup, the GDR Tresses, the GDR Platon, CAPES, the GEAR Network, the CNRS and the JSPS. The second author was supported by PNPD/CAPES-2013 during the first period of this project, and by "grant \#2018/19603-0, S\~ao Paulo Research Foundation (FAPESP)"  during the second period. The authors also thank the anonymous referee for interesting corrections and  D.Allegretti and K. Kim for pointing to them the references \cite{Gabella_QNonAb, KyuMiri_QNonAb}.

\bibliographystyle{amsalpha}
\bibliography{biblio}

\end{document}